\documentclass{amsart}

\usepackage{amssymb,amsmath,amsthm,amsbsy,amscd,stmaryrd}
\usepackage{hyperref}

\input xy
\xyoption{all} \xyoption{poly} \xyoption{knot}
\usepackage{color}
\usepackage{enumerate}

\usepackage{amssymb, amscd,stmaryrd,setspace,hyperref,color}

\input xy
\xyoption{all} \xyoption{poly} \xyoption{knot}\xyoption{curve}

\newcommand{\comment}[1]{}

\comment{The following is Eli Lebow's trick for making the table of contents include all definitions and labels.

\setcounter{tocdepth}{5}

}

\def\tn{\textnormal}

\def\mc{\mathcal}

\def\ZZ{{\mathbb Z}}

\def\RR{{\mathbb R}}
\def\CC{{\mathbb C}}
\def\AA{{\mathbb A}}
\def\PP{{\mathbb P}}
\def\NN{{\mathbb N}}

\def\Tor{\tn{Tor}}

\def\coker{\tn{coker }}
\def\Spec{\tn{Spec }}

\def\dim{\tn{dim}}
\def\rank{\tn{rank}}

\def\alg{{\bf \tn{-alg}}}

\def\Hom{\tn{Hom}}

\def\Map{\tn{Map}}

\def\Op{\tn{Op}}

\def\to{\rightarrow}
\def\from{\leftarrow}
\def\cross{\times}
\def\taking{\colon}
\def\inj{\hookrightarrow}

\def\too{\longrightarrow}

\def\ss{\subset}

\def\iso{\cong}
\def\m1{{-1}}
\def\|{{\;|\;}}
\def\op{^\tn{op}}
\def\loc{\tn{loc}}

\def\wt{\widetilde}

\def\we{\simeq}

\def\ullimit{\ar@{}[rd]|(.3)*+{\ulcorner}}
\def\urlimit{\ar@{}[ld]|(.3)*+{\urcorner}}
\def\lllimit{\ar@{}[ru]|(.3)*+{\llcorner}}
\def\lrlimit{\ar@{}[lu]|(.3)*+{\lrcorner}}
\def\ulhlimit{\ar@{}[rd]|(.3)*+{\diamond}}
\def\urhlimit{\ar@{}[ld]|(.3)*+{\diamond}}
\def\llhlimit{\ar@{}[ru]|(.3)*+{\diamond}}
\def\lrhlimit{\ar@{}[lu]|(.3)*+{\diamond}}
\newcommand{\clabel}[1]{\ar@{}[rd]|(.5)*+{#1}}

\newcommand{\arr}[1]{\ar@<.5ex>[#1]\ar@<-.5ex>[#1]}
\newcommand{\arrr}[1]{\ar@<.7ex>[#1]\ar@<0ex>[#1]\ar@<-.7ex>[#1]}
\newcommand{\arrrr}[1]{\ar@<.9ex>[#1]\ar@<.3ex>[#1]\ar@<-.3ex>[#1]\ar@<-.9ex>[#1]}
\newcommand{\arrrrr}[1]{\ar@<1ex>[#1]\ar@<.5ex>[#1]\ar[#1]\ar@<-.5ex>[#1]\ar@<-1ex>[#1]}

\newcommand{\To}[1]{\xrightarrow{#1}}

\newcommand{\From}[1]{\xleftarrow{#1}}

\newcommand{\push}[4]{\xymatrix{#1\ar[r]\ar[d] \ar@{}[rd]|(.7)*+{\lrcorner} & #2 \ar[d] \\ #3 \ar[r] & #4}}
\newcommand{\Push}[8]{\xymatrix{#1\ar[r]^-{#5}\ar[d]_-{#6} \ar@{}[rd]|(.7)*+{\lrcorner} & #2 \ar[d]^-{#7} \\ #3 \ar[r]_-{#8} & #4}}
\newcommand{\pull}[4]{\xymatrix{#1\ar[r]\ar[d] \ar@{}[rd]|(.3)*+{\ulcorner} & #2 \ar[d] \\ #3 \ar[r] & #4}}
\newcommand{\Pull}[8]{\xymatrix{#1\ar[r]^-{#5}\ar[d]_-{#6} \ar@{}[rd]|(.3)*+{\ulcorner} & #2 \ar[d]^-{#7} \\ #3 \ar[r]_-{#8} & #4}}
\newcommand{\hpush}[4]{\xymatrix{#1\ar[r]\ar[d] \ar@{}[rd]|(.7)*+{\diamond} & #2 \ar[d] \\ #3 \ar[r] & #4}}
\newcommand{\hPush}[8]{\xymatrix{#1\ar[r]^{#5}\ar[d]_{#6} \ar@{}[rd]|(.7)*+{\diamond} & #2 \ar[d]^{#7} \\ #3 \ar[r]_{#8} & #4}}
\newcommand{\hpull}[4]{\xymatrix{#1\ar[r]\ar[d] \ar@{}[rd]|(.3)*+{\diamond} & #2 \ar[d] \\ #3 \ar[r] & #4}}
\newcommand{\hPull}[8]{\xymatrix{#1\ar[r]^-{#5}\ar[d]_-{#6} \ar@{}[rd]|(.3)*+{\diamond} & #2 \ar[d]^-{#7} \\ #3 \ar[r]_-{#8} & #4}}

\newcommand{\Sq}[8]{\xymatrix{#1\ar[r]^-{#5}\ar[d]_-{#6} & #2 \ar[d]^-{#7} \\ #3 \ar[r]_-{#8} & #4}}

\newcommand{\adjoint}[2]{\xymatrix@1{#1\ar@<.5ex>[r] & #2 \ar@<.5ex>[l]}}
\newcommand{\Adjoint}[4]{\xymatrix@1{#2 \ar@<.5ex>[r]^-{#1} & #3 \ar@<.5ex>[l]^-{#4}}}

\def\id{\tn{id}}

\def\Sets{{\bf Sets}}
\def\sSets{{\bf sSets}}

\def\colim{\mathop{\tn{colim}}}
\def\hocolim{\mathop{\tn{hocolim}}}

\def\mcA{\mc{A}}
\def\mcB{\mc{B}}
\def\mcC{\mc{C}}
\def\mcD{\mc{D}}

\def\mcF{\mc{F}}
\def\mcG{\mc{G}}

\def\mcM{\mc{M}}
\def\mcN{\mc{N}}
\def\mcO{\mc{O}}
\def\mcP{\mc{P}}

\def\mcU{\mc{U}}
\def\mcV{\mc{V}}
\def\mcW{\mc{W}}
\def\mcX{\mc{X}}
\def\mcY{\mc{Y}}
\def\mcZ{\mc{Z}}

\def\tensor{\otimes}

\newtheorem{theorem}{Theorem}[section]
\newtheorem{lemma}[theorem]{Lemma}
\newtheorem{proposition}[theorem]{Proposition}
\newtheorem{corollary}[theorem]{Corollary}

\theoremstyle{remark}
\newtheorem{remark}[theorem]{Remark}
\newtheorem{example}[theorem]{Example}

\newtheorem{question}[theorem]{Question}
\newtheorem{guess}[theorem]{Guess}

\newtheorem{convention}[theorem]{Convention}

\theoremstyle{definition}
\newtheorem{definition}[theorem]{Definition}
\newtheorem{notation}[theorem]{Notation}

\def\Man{{\bf Man}}
\def\dMan{{\bf dMan}}

\def\CG{{\bf CG}}
\def\dM{{\bf dM}}
\def\U{{\bf U}}
\def\i{{\bf i}}
\def\EE{{\mathbb E}}
\def\E{\EE}
\def\sC{{s{\bf C^\infty}}}
\def\sR{{s\RR}}

\def\Shv{{\bf Shv}}
\def\LRS{{\bf LRS}}

\def\res{\hspace{-1pt}\centerdot}
\def\TT{{\mathbb T}}
\def\Cyl{{\bf Cyl}}

\numberwithin{equation}{theorem}

\begin{document}

\title{Derived smooth manifolds}
\author{David I. Spivak}

\begin{abstract}

We define a simplicial category called the category of derived manifolds.  It contains the category of smooth manifolds as a full discrete subcategory, and it is closed under taking arbitrary intersections in a manifold.  A derived manifold is a space together with a sheaf of local $C^\infty$-rings that is obtained by patching together homotopy zero-sets of smooth functions on Euclidean spaces.  

We show that derived manifolds come equipped with a stable normal bundle and can be imbedded into Euclidean space.  We define a cohomology theory called derived cobordism, and use a Pontrjagin-Thom argument to show that the derived cobordism theory is isomorphic to the classical cobordism theory.  This allows us to define fundamental classes in cobordism for all derived manifolds.  In particular, the intersection $A\cap B$ of submanifolds $A,B\ss X$ exists on the categorical level in our theory, and a cup product formula $$[A]\smile[B]=[A\cap B]$$ holds, even if the submanifolds are not transverse.  One can thus consider the theory of derived manifolds as a {\em categorification} of intersection theory.  

\end{abstract}

\maketitle
    
\setcounter{tocdepth}{1}
 
\tableofcontents 

\section{Introduction}

Let $\Omega$ denote the unoriented cobordism ring (though what we will say applies to other cobordism theories as well, e.g. oriented cobordism).  The fundamental class of a compact manifold $X$ is an element $[X]\in\Omega$.  By the Pontrjagin-Thom construction, such an element is classified by a homotopy class of maps from a large enough sphere $S^n$ to a large enough Thom space $MO$.  One can always choose a map $f\taking S^n\to MO$ which represents this homotopy class, is smooth (away from the basepoint), and meets the zero-section $B\ss MO$ transversely.  The pullback $f^\m1(B)$ is a compact manifold which is cobordant to $X$, so we have an equality $[X]=[f^\m1(B)]$ of elements in $\Omega$.

This construction provides a correspondence which is homotopical in nature: one begins with a homotopy class of maps and receives a cobordism class.  However, it is close to existing on the nose, in that a dense subset of all representing maps $f\taking S^n\to MO$ will be transverse to $B$ and yield an imbedded manifold rather than merely its image in $\Omega$.  If transversality were not an issue, Pontrjagin-Thom would indeed provide a correspondence between smooth maps $S^n\to MO$ and their zero-sets.

The purpose of this paper is to introduce the category of derived manifolds wherein non-transverse intersections make sense.  In this setting, $f^\m1(B)$ is a derived manifold which is derived cobordant to $X$, regardless of one's choice of smooth map $f$, and in terms of fundamental cobordism classes we have $[f^\m1(B)]=[X]$.  Our hope is that by using derived manifolds, researchers can avoid having to make annoying transversality arguments.  This could be of use in string topology or Floer homology, for example.

As an example, let us provide a short story that can only take place in the derived setting.  Consider the case of a smooth degree $d$ hypersurface $X\ss\CC P^3$ in complex projective space.  One can express the $K$-theory fundamental class of $X$ as \begin{equation}\label{eqn:hypersurface}[X]={d \choose 1}[\CC P^2]-{d \choose 2}[\CC P^1]+{d\choose 3}[\CC P^0].\end{equation}  It may be difficult to see exactly where this formula comes from; let us give the derived perspective, which should make the formula more clear.  

The union $Y$ of $d$ distinct hyperplanes in $\CC P^3$ is not smooth, but it does exist as an object in the category of derived manifolds.  Moreover, as the zero-set of a section of the line bundle $\mcO(d)$, one has that $Y$ is a degree $d$ derived submanifold of $\CC P^3$ which is derived cobordant to $X$.  As such, the fundamental class of $Y$ is equal to that of $X$, i.e. $[Y]=[X]$.  

The point is that the above formula (\ref{eqn:hypersurface}) takes on added significance as the fundamental class of $Y$, because it has the form of an inclusion-exclusion formula.  One could say that the $K$-theory fundamental class of $Y$ is obtained by adding the fundamental classes of $d$ hyperplanes, subtracting off the fundamental classes of the over-counted ${d\choose 2}$ lines of 2-fold intersection, and adding back the fundamental classes of the missed ${d\choose 3}$ points of 3-fold intersection.  Hopefully, this convinces the reader that derived manifolds may be of use.

To construct the virtual fundamental class on an arbitrary intersection of compact submanifolds, we follow an idea of Kontsevich \cite[Section 1.4]{Kont}, explained to us by Jacob Lurie.  Basically, we take our given intersection $\mcX=A\cap B$, realize it as the zeroset of a section of a vector bundle, and then deform that section until it is transverse to zero.  The result is a derived cobordism between $\mcX$ and a smooth manifold.   

While dispensing with the transversality requirement for intersecting manifolds is appealing, it does come with a cost, namely that defining the category of derived manifolds is a bit technical.  However, anyone familiar with homotopy sheaves will not be too surprised by our construction.  One starts with Lawvere's algebraic theory of $C^\infty$-rings, which are rings whose elements are closed under composition with smooth functions.  Simplicial (lax) $C^\infty$-rings are the appropriate homotopy-theoretic analogue and as such are objects in a simplicial model category.  We then form the category of local $C^\infty$-ringed spaces, wherein an object is a topological space together with a homotopy sheaf of simplicial $C^\infty$-rings whose stalks are local rings.  Euclidean space, with its (discrete) $C^\infty$-ring of smooth real-valued functions is such an object, and the zero-set of finitely many smooth functions on Euclidean space is called an affine derived manifold.  A derived manifold is a local $C^\infty$-ringed space which is obtained by patching together affine derived manifolds.  See Definition \ref{def:dman}.

\begin{notation}\label{homotopy limits}

Let $\sSets$ denote the monoidal category of simplicial sets.  A simplicial category $\mcC$ is a category enriched over $\sSets$; we denote the mapping space functor $\Map_\mcC(-,-)$.  If all of the mapping spaces in $\mcC$ are fibrant (i.e. Kan complexes), we call $\mcC$ {\em fibrant} as a simplicial category; in the following discussion we will be considering only this case.  By a {\em map} between objects $X$ and $Y$ in $\mcC$, we mean a 0-simplex in $\Map_\mcC(X,Y)$.

An object $X\in\mcC$ is called {\em homotopy initial} if for every $Y\in\mcC$, the mapping space $\Map(X,Y)$ is contractible.  Similarly, $X$ is called {\em homotopy terminal} if for every $Y\in\mcC$, the mapping space $\Map(Y,X)$ is contractible.  We say that a vertex in a contractible space is {\em homotopy-unique}.  We sometimes abuse notation and refer to a contractible space as though it were just a single point, saying something like ``{\em the} natural map $Y\to X$."     

Let $\mcC$ be a simplicial category.  The homotopy pullback of a diagram $A\To{f} B\From{g} C$ is a diagram \begin{align}\label{dia:hlimit}\hPull{A\cross_BC}{C}{A}{B.}{f'}{g'}{g}{f}\end{align} By this we mean an object $A\cross_BC$ equipped with maps $g',h',$ and $f'$ to $A, B,$ and $C$ respectively, and further equipped with homotopies between $gf'$ and $h'$ and between $fg'$ and $h'$.  Finally, we require that $A\cross_BC$ is homotopy terminal in the category of such objects.  More succinctly, Diagram \ref{dia:hlimit} expresses that for any object $X\in\mcC$ the natural map $$\Map(X,A\cross_BC)\to\Map(X,A)\cross^h_{\Map(X,B)}\Map(X,C)$$ is a weak equivalence in the usual model category of simplicial sets (see \cite[7.10.8]{Hir}), where by $\cross^h$ we mean the homotopy pullback in $\sSets$.  The diamond in the upper left corner of the square in Diagram \ref{dia:hlimit} serves to remind the reader that object in the upper left corner is a homotopy pullback, and that the diagram does not commute on the nose but up to chosen homotopies.  We can define homotopy pushouts similarly.

Two objects $X$ and $Y$ in $\mcC$ are said to be {\em equivalent} if there exist maps $f\taking X\to Y$ and $g\taking Y\to X$ such that $g\circ f$ and $f\circ g$ are homotopic to the identity maps on $X$ and $Y$.  By \cite[1.2.4.1]{Lur-HTT}, this is equivalent to the assertion that the map $\Map(Z,X)\to\Map(Z,Y)$ is a weak equivalence for all $Z\in\mcC$.  

If $\mcC$ is a discrete simplicial category (i.e. a category in the usual sense) then the homotopy pullback of a diagram in $\mcC$ is the pullback of that diagram.  The pullback of a diagram $A\To{f} B\From{g} C$ is a commutative diagram $$\Pull{A\cross_BC}{C}{A}{B.}{f'}{g'}{g}{f}$$  The symbol in the upper left corner serves to remind the reader that the object in the upper left corner is a pullback in the usual sense.  Two objects are equivalent if and only if they are isomorphic.

\begin{remark}\label{rem:model cats}

If $\mcC$ is a simplicial model category (see, e.g. \cite{Hir} for an introduction to this subject), then the full subcategory of cofibrant-fibrant objects is a simplicial category in which all mapping spaces are fibrant.  Moreover, we can replace any diagram with a diagram of the same shape in which all objects are cofibrant-fibrant.  Our definitions of homotopy pullback, homotopy pushout, and equivalence match the model category terminology.  In keeping with this, if we are in the setting of model categories, the result of a construction (such as taking a homotopy limit) will always be assumed cofibrant and fibrant.

\end{remark}

{\bf Whenever we speak of pullbacks in a simplicial category, we are always referring to homotopy pullbacks unless otherwise specified.}  Similarly, whenever we speak of terminal (resp. initial) objects, we are always referring to homotopy terminal (resp. homotopy initial) objects.  Finally, we sometimes use the word ``category" to mean ``simplicial category."    

Given a functor $F\taking\mcC\to\mcD$, we say that an object $D\in\mcD$ is in the {\em essential image} of $F$ if there is an object $C\in\mcC$ such that $F(C)$ is equivalent to $D$.

We denote the category of smooth manifolds by $\Man$; whenever we refer to a manifold, it is always assumed smooth.  It is discrete as a simplicial category.  In other words, we {\em do not} include any kind of homotopy information in $\Man$.  

\end{notation}

We now recall a few well-known facts and definitions about $\Man$.  Every manifold $A$ has a tangent bundle $T_A\to A$ which is a vector bundle whose rank is equal to the dimension of $A$.  A morphism of smooth manifolds $f\taking A\to B$ induces a morphism $T_f\taking T_A\to f^*T_B$ of vector bundles on $A$, from the tangent bundle of $A$ to the pullback of the tangent bundle on $B$.  We say that $f$ is an {\em immersion} if $T_f$ is injective and we say that $f$ is a {\em submersion} if $T_f$ is surjective.  A pair of maps $f\taking A\to B$ and $g\taking C\to B$ are called {\em transverse} if the induced map $f\amalg g\taking A\amalg C\to B$ is a submersion.  If $f$ and $g$ are transverse, then their fiber product (over $B$) exists in $\Man$.\\

\subsection{Results}

In this paper we hope to convince the reader that we have a reasonable category in which to do intersection theory on smooth manifolds.  The following definition expresses what we mean by ``reasonable."  Definition \ref{general cup} expresses what we mean by ``doing intersection theory" on such a category.  The main result of the paper, Theorem \ref{main theorem}, is that there is a simplicial category which satisfies Definition \ref{general cup}

\begin{definition}\label{def geo}

A simplicial category $\mcC$ is called {\em geometric} if it satisfies the following axioms:\begin{enumerate}

\item\label{Axiom fib}{\bf Fibrant. } For any two objects $\mcX,\mcY\in\mcC$, the mapping space $\Map_\mcC(\mcX,\mcY)$ is a fibrant simplicial set.

\item\label{Axiom sm} {\bf Smooth manifolds. } There exists a fully faithful functor $\i\taking\Man\to\mcC$.  We say that $M\in\mcC$ is a manifold if it is in the essential image of $\i$.

\item\label{Axiom ml}{\bf Manifold limits. } The functor $\i$ commutes with transverse intersections.  That is, if $A\to M$ and $B\to M$ are transverse, then a homotopy limit $\i(A)\cross_{\i(M)}\i(B)$ exists in $\mcC$ and the natural map $$\i(A\cross_MB)\to\i(A)\cross_{\i(M)}\i(B)$$ is an equivalence in $\mcC$.  

Furthermore, the functor $\i$ preserves the terminal object (i.e. the object $\i(\RR^0)$ is homotopy terminal in $\mcC$).

\item\label{Axiom us} {\bf Underlying spaces. } Let $\CG$ denote the discrete simplicial category of compactly generated Hausdorff spaces.  There exists an ``underlying space" functor $\U\taking\mcC\to\CG$, such that the diagram $$\xymatrix{\Man\ar[r]^\i\ar[d]&\mcC\ar[dl]^\U\\ \CG}$$ commutes, where the vertical arrow is the usual underlying space functor on smooth manifolds.  Furthermore, the functor $\U$ commutes with finite limits when they exist.

\end{enumerate}

\begin{remark}\label{rem:top terminology}

When we speak of an object (respectively, a morphism or a set of morphisms) in $\mcC$ having some topological property (e.g. Hausdorff or compact object, proper morphism, open cover, etc.), we mean that the underlying object (resp. the underlying morphism or set of morphisms) in $\CG$ has that property.

\end{remark}

\end{definition}

Since any discrete simplicial category has fibrant mapping spaces, it is clear that $\Man$ and $\CG$ are geometric.   

If $\mcC$ is a geometric category and $M\in\Man$ is a manifold, we generally abuse notation and write $M$ in place of $\i(M)$, as though $M$ were itself an object of $\mcC$.

\begin{remark}

Note that in Axiom \ref{Axiom ml} we are not requiring that $\i$ commute with all limits which exist in $\Man$, only those which we have deemed appropriate.  For example, if one has a line $L$ and a parabola $P$ which meet tangentially in $\RR^2$, their fiber product $L\cross_{\RR^2}P$ {\em does} exist in the category of manifolds (it is a point).  However, limits like these are precisely the kind we wish to avoid!  We are searching for a category in which intersections are in some sense stable under perturbations (see Definition \ref{general cup}, Condition (\ref{Cond cpf})), and thus we should not ask $\i$ to preserve intersections which are not stable in this sense.

\end{remark}

\begin{remark}

In all of the axioms of Definition \ref{def geo}, we are working with simplicial categories, so when we speak of pullbacks and puhouts, we mean homotopy pullbacks and homotopy puhouts.  Axiom \ref{Axiom us} requires special comment however.  We take $\CG$, the category of compactly generated hausdorff spaces, as a {\em discrete} simplicial category, so homotopy pullbacks and pushouts are just pullbacks and pushouts in the usual sense.  The underlying space functor $\U$ takes finite homotopy pullbacks in $\mcC$ to pullbacks in $\CG$.

\end{remark}

Again, our goal is to find a category in which intersections of arbitrary submanifolds exist at the categorical level and descend correctly to the level of cobordism rings.  We make this precise in the following definition.

\begin{definition}\label{general cup}

We say that a simplicial category $\mcC$ {\em has the general cup product formula in cobordism} if the following conditions hold.  \begin{enumerate}

\item \label{Cond geo}{\bf Geometric. } The simplicial category $\mcC$ is geometric in the sense of Definition \ref{def geo}.

\item \label{Cond int}{\bf Intersections. } If $M$ is a manifold and $A$ and $B$ are submanifolds (possibly not transverse), then there exists a homotopy pullback $A\cross_MB$ in $\mcC$, which we denote by $A\cap B$ and call the {\em derived intersection} of $A$ and $B$ in $M$.  

\item \label{Cond db}{\bf Derived cobordism. } There exists an equivalence relation on the compact objects of $\mcC$ called derived cobordism, which extends the cobordism relation on manifolds.  That is, for any manifold $T$, there is a ring $\Omega^{der}(T)$ called {\em the derived cobordism ring over $T$}, and the functor $\i\taking\Man\to\mcC$ induces a homomorphism of cobordism rings over $T$, $$\i_*\taking\Omega(T)\to\Omega^{der}(T).$$  We further impose the condition that $\i_*$ be an injection.

\item \label{Cond cpf}{\bf Cup product formula. }  If $A$ and $B$ are compact submanifolds of a manifold $M$ with derived intersection $A\cap B:=A\cross_MB$, then the cup product formula \begin{equation}\label{eqn cpf}[A]\smile[B]=[A\cap B]\end{equation} holds, where $[-]$ is the functor taking a compact derived submanifold of $M$ to its image in the derived cobordism ring $\Omega^{der}(M)$, and where $\smile$ denotes the multiplication operation in that ring (i.e. the cup product).  

\end{enumerate}

\end{definition}

Without the requirement (Condition \ref{Cond db}) that $\i_*\taking\Omega(T)\to\Omega^{der}(T)$ be an injection, the general cup product formula could be trivially attained.  For example, one could extend $\Man$ by including non-transverse intersections which were given no more structure than their underlying space, and the derived cobordism relation could be chosen to be maximal (i.e. one equivalence class); then the cup product formula would trivially hold.

In fact, when we eventually prove that there is a category which has the general cup product formula, we will find that $\i_*$ is not just an injection but an isomorphism (see Theorem \ref{refined main theorem}).  We did not include that as an axiom here, however, because it does not seem to us to be an inherently necessary aspect of a good intersection theory.

The category of smooth manifolds does not have the general cup product formula because it does not satisfy condition (\ref{Cond int}).  Indeed, suppose that $A$ and $B$ are submanifolds of $M$.  If $A$ and $B$ meet transversely, then their intersection will naturally have the structure of a manifold, and the cup product formula \ref{eqn cpf} will hold.  If they do not, then one of two things can happen: either their intersection cannot be given the structure of a manifold (so in the classical setting, Equation \ref{eqn cpf} does not have meaning) or their intersection can be given the structure of a smooth manifold, but it {\em will not} satisfy Equation \ref{eqn cpf}.

Therefore, we said that a category which satisfies the conditions of Definition \ref{general cup} satisfies the {\em general} cup product formula because condition (\ref{Cond cpf}) holds even for non-transverse intersections.  Of course, to accomplish this, one needs to find a more refined notion of intersection, i.e. find a setting in which homotopy limits will have the desired properties. 

Suppose that a simplicial category $\mcC$ has the general cup product formula in cobordism.  It follows that any cohomology theory $E$ which has fundamental classes for compact manifolds (i.e. for which there exists a map $MO\to E$) also has fundamental classes for compact objects of $\mcC$, and that these satisfy the cup product formula as well.  Returning to our previous example, the union of $d$ hyperplanes in complex projective space is a derived manifold which, we will see, is derived cobordant to a smooth degree $d$ hypersurface (see Example \ref{main examples}).  Thus, these two subspaces have the same $K$-theory fundamental classes, which justifies Equation \ref{eqn:hypersurface}.

Our main result is that the conditions of Definition \ref{general cup} can be satisfied.  

\begin{theorem}\label{main theorem}

There exists a simplicial category $\dMan$, called {\em the category of derived manifolds}, which has the general cup product formula in cobordism, in the sense of Definition \ref{general cup}.  

\end{theorem}

The category $\dMan$ is defined in Definition \ref{def:dman}, and the above theorem is proved as Theorem \ref{dman sat gfdi}.  See also Definition \ref{good for doing it} for a list of axioms satisfied by $\dMan$.

\begin{remark}

We do not offer a uniqueness counterpart to Theorem \ref{main theorem}.  We purposely left Definition \ref{general cup} loose, because we could not convince the reader that any more structure on a category $\mcC$ was necessary in order to say that it ``has the general cup product formula."  For example, we could have required that the morphism $\i_*$ be an isomorphism instead of just an injection (this is indeed the case for $\dMan$, see Corollary \ref{cor: refined main theorem}); however, doing so would be hard to justify as being necessary.  Because of the generality of Definition \ref{general cup}, we are precluded from offering a uniqueness result here.

\end{remark}

Finally, the following proposition justifies the need for simplicial categories in this work.

\begin{proposition}

If $\mcC$ is a discrete simplicial category (i.e. a category in the usual sense), then $\mcC$ cannot have the general cup product formula in cobordism.

\end{proposition}

\begin{proof}

We assume that Conditions (\ref{Cond geo}), (\ref{Cond int}), and (\ref{Cond db}) hold, and we show that Condition (\ref{Cond cpf}) cannot.

Since $\mcC$ is geometric, the object $\RR^0$ (technically $\i(\RR^0)$) is homotopy terminal in $\mcC$.  Since $\mcC$ is discrete, $\RR^0$ is terminal in $\mcC$, all equivalences in $\mcC$ are isomorphisms, and all homotopy pullbacks in $\mcC$ are categorical pullbacks.  Let $0\taking\RR^0\to\RR$ be the origin, and let $X$ be defined as the pullback in the diagram $$\Pull{X}{\RR^0}{\RR^0}{\RR.}{}{}{0}{0}$$  A morphism from an arbitrary object $Y$ to $X$ consists of two morphisms $Y\to\RR^0$ which agree under composition with $0$.  For any $Y$, there is exactly one such morphism $Y\to X$ because $\RR^0$ is terminal in $\mcC$.  That is, $X$ represents the same functor as $\RR^0$ does, so $X\iso\RR^0$, i.e. $\RR^0\cap\RR^0=\RR^0$.  This equation forces Condition (\ref{Cond cpf}) to fail. 

Indeed, to see that the cup product formula $$[\RR^0]\smile[\RR^0]=^?[\RR^0\cap\RR^0]$$ does not hold in $\Omega(\RR)$, note that the left-hand side is homogeneous of degree 2, whereas the right-hand side is homogeneous of degree 1 in the cohomology ring.

\end{proof}

\subsection{Structure of the paper}

We have decided to present this paper in a hierarchical fashion.  In the introduction, we presented the goal: to find a geometric category that has the general cup product formula in cobordism (see Definition \ref{general cup}).  

In Section 2, we present a set of axioms that suffice to achieve this goal.  In other words, any category that satisfies the axioms of Definition \ref{good for doing it} is said to be ``good for doing intersection theory on manifolds," and we prove in Theorem \ref{refined main theorem} that such a category has the general cup product formula.  Of course, we could have chosen our axioms in a trivial way, but this would not have given a useful layer of abstraction.  Instead, we chose axioms that resemble axioms satisfied by smooth manifolds.  These axioms imply the general cup product formula, but are not implied by it.

In Sections 5 - 9, we give an explicit formulation of a category that is good for doing intersection theory.  This category can be succinctly described as ``the category of homotopical $C^\infty$-schemes of finite type."  To make this precise and prove that it satisfies the axioms of Definition \ref{good for doing it} takes five sections.  We lay out our methodology for this undertaking in Section 4.

Finally, in Section 10, we discuss related constructions.  First we see the way that derived manifolds are related to Jacob Lurie's ``structured spaces" (\cite{Lur-DAG5}).  Then we discuss manifolds with singularities, Chen spaces, diffeological spaces, and synthetic $C^\infty$-spaces, all of which are generalizations of the category $\Man$ of smooth manifolds.  In this section we hope to show how the theory of derived manifolds fits into the existing literature.

\subsection{Acknowledgments}

This paper is a reformulation and a simplification of my PhD dissertation.  Essentially, my advisor for that project was Jacob Lurie, whom I thank for many enlightening and helpful conversations, as well as for suggesting the project.

I would like to thank Dan Dugger for suggestions which improved the readability of this paper tremendously, as well as for his advice throughout the rewriting process.  The referee reports were also quite useful in debugging and clarifying the paper.  I thank Peter Teichner for suggesting that I move to Boston to work directly with Jacob, as well as for many useful conversations.

Finally, I would like to thank Mathieu Anel, Tom Graber, Phil Hirschhorn, Mike Hopkins, Andr\'{e} Joyal, Dan Kan, Grace Lyo, Haynes Miller, and Dev Sinha for useful conversations and encouragement at various stages of this project.

\section{The axioms}

Theorem \ref{main theorem} makes clear our objectives: to find a simplicial category in which the general cup product formula holds.  In this section, specifically in Definition \ref{good for doing it}, we provide a set of axioms which \begin{itemize}\item naturally extend corresponding properties of smooth manifolds and \item together imply Theorem \ref{main theorem}.\end{itemize}  This is an attempt to give the reader the ability to work with derived manifolds (at least at a surface level) without fully understanding their technical underpinnings.  

In the following section, Section \ref{main}, we will prove Theorem \ref{main theorem} from the axioms presented in Definition \ref{good for doing it}.  Then in Section \ref{layout} we will give an outline of the internal structure of a simplicial category which satisfies the axioms in Definition \ref{good for doing it}.  Finally, in the remaining sections we will fulfill the outline given in Section \ref{layout}, proving our main result in Section \ref{gfdi}.

\begin{definition}\label{good for doing it}

A simplicial category $\dM$ is called {\em good for doing intersection theory on manifolds} if it satisfies the following axioms: \begin{enumerate} 

\item\label{Axiom geo}{\bf Geometric. } The simplicial category $\dM$ is geometric in the sense of Definition \ref{def geo}.  That is, roughly, $\dM$ has fibrant mapping spaces, contains the category $\Man$ of smooth manifolds, has reasonable limits, and has underlying Hausdorff spaces.

\item\label{Axiom os} {\bf Open subobjects.} \begin{quotation} \begin{definition}\label{open subobject} Suppose that $\mcX\in\dM$ is an object with underlying space $X=\U(\mcX)$ and that $j\taking V\inj X$ is a open inclusion of topological spaces.  We say that a map $k\taking\mcV\to\mcX$ in $\dM$ is an {\em open subobject over $j$} if it is Cartesian over $j$; i.e. \begin{itemize} \item $\U(\mcV)=V,$ \item $\U(k)=j$, and \item If $k'\taking \mcV'\to\mcX$ is a map with $\U(\mcV')=V'$ and $\U(k')=j'$, such that $j'$ factors through $j$, then $k'$ factors homotopy-uniquely through $k$; that is, the space of dotted arrows making the diagram $$\xymatrix{\mcV'\ar@{|->}[d]_\U\ar@/^1pc/[rr]^{k'}\ar@{.>}[r]&\mcV\ar@{|->}[d]_\U\ar[r]_{k}&\mcX\ar@{|->}[d]_\U\\ V'\ar@/_1pc/[rr]_{j'}\ar[r]&V\ar[r]^j&X,}$$commute is contractible.\end{itemize} \end{definition}\end{quotation} For any $\mcX\in\dM$ and any open inclusion $j$ as above, there exists an open subobject $k$ over $j$.  Moreover, if a map $f\taking Z\to X$ of topological spaces underlies a map $g\taking\mcZ\to\mcX$ in $\dM$, then for any open neighborhood $U$ of $f(Z)$, the map $g$ factors through the open subobject $\mcU$ over $U$.

\item \label{Axiom u}{\bf Unions. } \begin{enumerate}\item Suppose that $\mcX$ and $\mcY$ are objects of $\dM$ with underlying spaces $X$ and $Y$ and that $i\taking\mcU\to\mcX$ and $j\taking\mcV\to\mcY$ are open subobjects with underlying spaces $U$ and $V$.  If $a\taking\mcU\to\mcV$ is an equivalence, and if the union of $X$ and $Y$ along $U\iso V$ is Hausdorff (so $X\cup Y$ exists in $\CG$) then the union $\mcX\cup\mcY$ (i.e. the colimit of $j\circ a$ and $i$) exists in $\dM$, and one has as expected $\U(\mcX\cup\mcY)=X\cup Y$.  \item If $f\taking\mcX\to\mcZ$ and $g\taking\mcY\to\mcZ$ are morphisms whose restrictions to $\mcU$ agree, then there is a morphism $\mcX\cup\mcY\to\mcZ$ which restricts to $f$ and $g$ respectively on $\mcX$ and $\mcY$.\end{enumerate}

\item\label{Axiom fl} {\bf Finite limits. } Given objects $\mcX,\mcY\in\dM$, a smooth manifold $M$, and maps $f\taking\mcX\to M$ and $g\taking\mcY\to M$, there exists an object $\mcZ\in\dM$ and a homotopy pullback diagram \begin{align}\label{dia:arb pullback}\hPull{\mcZ}{\mcY}{\mcX}{M.}{}{i}{g}{f}\end{align} We denote $\mcZ$ by $\mcX\cross_M\mcY$. If $\mcY=\RR^0$, $M=\RR^k$, and $g\taking\RR^0\to\RR^k$ is the origin, we denote $\mcZ$ by $\mcX_{f=0}$, and we call $i$ the canonical inclusion of the zeroset of $f$ into $\mcX$.

\item\label{Axiom lm} {\bf Local models. } We say that an object $\mcU\in\dM$ is a {\em local model} if there exists a smooth function $f\taking\RR^n\to\RR^k$ such that $\mcU\iso\RR^n_{f=0}$.  The {\em virtual dimension} of $\mcU$ is $n-k$.  For any object $\mcX\in\dM$, the underlying space $X=\U(\mcX)$ can be covered by open subsets in such a way that the corresponding open subobjects of $\mcX$ are all local models.  More generally, any open cover of $\U(\mcX)$ by open sets can be refined to an open cover whose corresponding open subobjects are local models.

\item\label{Axiom i} {\bf Local extensions for imbeddings.} \begin{quotation}\begin{definition}\label{def of imbed} For any map $f\taking\mcY\to\RR^n$ in $\dM$, the canonical inclusion of the zeroset $\mcY_{f=0}\to\mcY$ is called a {\em model imbedding}.  A map $g\taking\mcX\to\mcY$ is called an {\em imbedding} if there is a cover of $\mcY$ by open subobjects $\mcY_i$ such that, if we set $\mcX_i=g^\m1(\mcY_i)$, the induced maps $g|_{\mcX_i}\taking\mcX_i\to\mcY_i$ are model imbeddings.  Such open subobjects $\mcY_i\ss\mcY$ are called {\em trivializing neighborhoods} of the imbedding. \end{definition}\end{quotation}

Let $g\taking\mcX\to\mcY$ be an imbedding and let $h\taking\mcX\to\RR$ be a map in $\dM$.  Then there exists a dotted arrow such that the diagram $$\xymatrix{\mcX\ar[r]^h\ar[d]_g&\RR\\\mcY\ar@{-->}[ur]&}$$ commutes up to homotopy.  

\item\label{Axiom snb} {\bf Normal bundle. } Let $M$ be a smooth manifold and $\mcX\in\dM$ with underlying space $X=\U(\mcX)$.   If $g\taking\mcX\to M$ is an imbedding, then there exists an open neighborhood $U\ss M$ of $\mcX$, a real vector bundle $E\to U$, and a section $s\taking U\to E$ such that $$\hPull{\mcX}{U}{U}{E,}{g}{g}{z}{s}$$ is a homotopy pullback diagram, where $z\taking U\to E$ is the zero section of the vector bundle.  Let $g=\U(g)$ also denote the underlying map $X\to U$; then the pullback bundle $g^*(E)$ on $X$ is unique up to isomorphism.  We call $g^*(E)$ the {\em normal bundle of $X$ in $M$} and $s$ a {\em defining section}.

\end{enumerate}

Objects in $\dM$ will be called {\em derived manifolds (of type $\dM$)} and morphisms in $\dM$ will be called {\em morphisms of derived manifolds (of type $\dM$)}.  

\end{definition}

\begin{remark}

We defined the virtual dimension of a local model $\mcU=\RR^n_{f=0}$ in Axiom \ref{Axiom lm}.  We often drop the word ``virtual" and refer to the virtual dimension of $\mcU$ simply as {\em the dimension} of $\mcU$.

We will eventually define the virtual dimension of an arbitrary derived manifold as the Euler characteristic of its cotangent complex (Definition \ref{def:dimension}).  For now, the reader only needs to know that if $\mcZ,\mcX,\mcY$, and $M$ are as in Diagram (\ref{dia:arb pullback}) and these objects have constant dimension $z,x,y$, and $m$ respectively, then $z+m=x+y$, as expected.

\end{remark}

Let us briefly explain the definition of imbedding (Definition \ref{def of imbed}) given in Axiom \ref{Axiom i}.  If we add the word ``transverse" in the appropriate places, we are left with the usual definition of imbedding for smooth manifolds.  This is proven in Proposition \ref{justification of imbed def}.

Recall that a map of manifolds is called a (smooth) imbedding if the induced map on the total spaces of their respective tangent bundles is an injection.  Say that a map of manifolds $X\to U$ is {\em the inclusion of a level manifold} if there exists a smooth function $f\taking U\to\RR^n$, transverse to $0\taking\RR^0\to\RR^n$, such that $X\iso U_{f=0}$ over $U$.

\begin{proposition}\label{justification of imbed def}

Let $X$ and $Y$ be smooth manifolds, and $g\taking X\to Y$ a smooth map.  Then $g$ is an imbedding if and only if there is a cover of $Y$ by open submanifolds $U^i$ such that, if we set $X_i=g^\m1(U^i)$, each of the induced maps $g|_{X_i}\taking X_i\to U^i$ is the inclusion of a level manifold.

\end{proposition}

\begin{proof}[Sketch of Proof]

We may assume $X$ is connected.  If $g$ is a smooth imbedding of codimension $d$, let $U$ be a tubular neighborhood, take the $U^i\ss U$ to be open subsets that trivialize the normal bundle of $X$, and take the $f^i\taking U^i\to\RR^k$ to be identity on the fibers.  The zero sets of the $f^i$ are open subsets of $X$, namely $U^i_{f^i=0}\iso X_i$.  Since they are smooth of the correct codimension, the $f^i$ are transverse to zero.

For the converse, note that the property of being an imbedding is local on the target, so we may assume that $X$ is the preimage of the origin under a map $f\taking U\to\RR^k$ that is transverse to zero, where $U\ss Y$ is some open subset.  The induced map $X\to U$ is clearly injective on tangent bundles.

\end{proof}

We now present a refinement of Theorem \ref{main theorem}, which we will prove as Corollary \ref{cor: refined main theorem} in the following section.  Recall that $\i\taking\Man\to\dMan$ denotes the inclusion guaranteed by Axiom \ref{Axiom geo} of Definition \ref{good for doing it}.

\begin{theorem}\label{refined main theorem} 

If $\dM$ is good for doing intersection theory on manifolds, then $\dM$ has the general cup product formula in cobordism, in the sense of Definition \ref{general cup}.  Moreover, for any manifold $T$, the  functor $$\i_*\taking\Omega(T)\to\Omega^{der}(T)$$ is an isomorphism between the classical cobordism ring and the derived cobordism ring (over $T$).

\end{theorem}

\begin{example}\label{main examples}

Let $\dM$ denote a simplicial category which is good for doing intersection theory.  By the unproven theorem (\ref{refined main theorem}), we have a cobordism theory $\Omega^{der}$ and a nice cup product formula for doing intersection theory.  We now give several examples which illustrate various types of intersections.  The last few examples are cautionary.

\begin{description}

\item[Transverse planes] Consider a $k$-plane $K$ and an $\ell$-plane $L$ inside of projective space $\PP^n$.  If $K$ and $L$ meet transversely, then they do so in a $(k+\ell-n)$-plane, which we denote $A$.  In $\dM$ there is an equivalence $A\iso K\cross_{\PP^n}L.$  Of course, this descends to an equality $[A]=[K]\smile[L]$ both in the cobordism ring $\Omega(\PP^n)$ and in the derived cobordism ring $\Omega^{der}(\PP^n)$.

\item[Non-transverse planes] Suppose now that $K$ and $L$ are as above but do not meet transversely.  Their topological intersection $A'$ will have the structure of a smooth manifold but of ``the wrong dimension"; i.e. $\dim(A)>k+\ell-n$.  Moreover, the formula $[A']=^?[K]\smile[L]$ will {\em not} hold in $\Omega(\PP^n)$.

However, the intersection of $K$ and $L$ as a derived manifold is different from $A'$; let us denote it $A$.  Although the underlying spaces $\U(A')=\U(A)$ will be the same, the virtual dimension of $A$ will be $k+\ell-n$ as expected.  Moreover, the formula $[A]=[K]\smile[L]$ holds in the derived cobordism ring $\Omega^{der}(\PP^n)$.   (The formula does not makes sense in $\Omega(\PP^n)$ because $A$ is not a smooth manifold.)

\item[Fiber products and zerosets] Let $M\to P\from N$ be a diagram of smooth manifolds with dimensions $m,p,$ and $n$.  The fiber product $\mcX$ of this diagram exists in $\dM$, and the (virtual) dimension of $\mcX$ is $m+n-p$. 

For example, if $f_1,\ldots f_k$ is a finite set of smooth functions $M\to\RR$ on a manifold $M$, then their zeroset is a derived manifold $\mcX$, even if the $f_1,\ldots,f_k$ are not transverse.  To see this, let $f=(f_1,\ldots,f_k)\taking M\to\RR^k$ and realize $\mcX$ as the fiber product in the diagram $$\hPull{\mcX}{\RR^0}{M}{\RR^k.}{}{}{0}{f}$$  The dimension of $\mcX$ is $m-k$, where $m$ is the dimension of $M$.

For example, let $T$ denote the 2-dimensional torus and let $f\taking T\to\RR$ denote a Morse function.  If $p\in\RR$ is a critical value of index 1, then the pullback $f^\m1(p)$ is a ``figure 8" (as a topological space).  It comes with the structure of a derived manifold of dimension 1.  It is derived cobordant both to a pair of disjoint circles and to a single circle; however, it is not isomorphic as a derived manifold to either of these because its underlying topological space is different.

\item[Euler classes] Let $M$ denote a compact smooth manifold and let $p\taking E\to M$ denote a smooth vector bundle; consider $M$ as a submanifold of $E$ by way of the zero section $z\taking M\to E$.  The Euler class $e(p)$ is often thought of as the cohomology class represented by the intersection of $M$ with itself inside $E$.  However, classically one must always be careful to perturb $M\ss E$ before taking this self intersection. In the theory of derived manifolds, it is not necessary to choose a perturbation.  The fiber product $M\cross_EM$ exists as a compact derived submanifold of $M$, and one has $$e(p)=[M\cross_EM].$$ 

\item[Vector bundles] More generally, let $M$ denote a smooth manifold and let $p\taking E\to M$ denote a smooth vector bundle and $z\taking M\to E$ the zero section.  Given an arbitrary section $s\taking M\to E$, the zero set $Z(s):=z(M)\cap s(M)$ of $s$ is a derived submanifold of $M$.  If $s$ is transverse to $z$ inside $E$, then $Z(s)$ is a submanifold of $M$, and its manifold structure coincides with its derived manifold structure.  Otherwise, $Z(s)$ is a derived manifold that is not equivalent to any smooth manifold.

Changing $s$ by a linear automorphism of $E$ (over $M$) does not change the homotopy type of the derived manifold $Z(s)$.  Arbitrary changes of section do change the homotopy type: if $s$ and $t$ are any two sections of $E$, then $Z(s)$ is not generally equivalent to $Z(t)$ as a derived manifold.  However these two derived manifolds will be {\em derived cobordant}.  The derived cobordism can be given by a straight-line homotopy in $E$.

\item[Failure of Nullstellensatz] Suppose that $X$ and $Y$ are varieties (resp. manifolds), and that $X\to Y$ is a closed imbedding.  Let $I(X)$ denote the ideal of functions on $Y$ which vanish on $X$; given an ideal $J$, let $Z(J)$ denote the zeroset of $J$ in $Y$.  A classical version of the Nullstellensatz states that if $k$ is an algebraically closed field, then $I$ induces a bijection between the Zariski closed subsets of affine space $\AA^n=\Spec k[x_1,\ldots,x_n]$ and the radical ideals of the ring $k[x_1,\ldots,x_n]$.  A corollary is that for any closed subset $X\ss Y$, one has $X=Z(I(X))$; i.e. $X$ is the zero-set of the ideal of functions which vanish on $X$.

This radically fails in the derived setting.  The simultaneous zero-set of $n$ functions $Y\to\RR$ always has codimension $n$ in $Y$.  For example, the $x$-axis in $\RR^2$ is the zero-set of a single function $y\taking \RR^2\to\RR$, as a derived manifold.  However, $y$ is not the only function that vanishes on the $x$-axis -- for example, so do the functions $2y, 3y, y^2,$ and $0.$  If we find the simultaneous zero set of all five functions, the resulting derived manifold has codimension 5 inside of $\RR^2$.  Its underlying topological space is indeed the $x$-axis, but its structure sheaf is very different from that of $\RR$.  

Thus, if we take a closed subvariety of $Y$ and quotient the coordinate ring of $Y$ by the infinitely many functions which vanish on it, the result will have infinite codimension.  The formula $X=Z(I(X))$ fails in the derived setting (i.e. both for derived manifolds and for derived schemes).

We note one upshot of this.  Given a closed submanifold $N\ss M$, one cannot identify $N$ with the zeroset of the ideal sheaf of functions on $M$ which vanish on $N$.  Given an arbitrary closed subset $X$ of a manifold, one may wish to find an appropriate derived manifold structure on $X$ -- this cannot be done in a canonical way as it can in classical algebraic geometry, unless $X$ is a local complete intersection (see ``Vector Bundles" example above).

\item[Unions not included] Note that the union of manifolds along closed subobjects does {\em not} generally come equipped with the structure of a derived manifold.  The reason we include this cautionary note is that, in the introduction, we spoke of the union $Y$ of $d$ hyperplanes in $\CC P^n$.  However, we were secretly regarding $Y$ as the zeroset of a single section of the bundle $\mcO(d)$ on $\CC P^n$, and not as a union.  We referred to it as a union only so as to aid the readers imagination of $Y$.  See the ``Vector bundles" example above for more information on $Y$.

\item[Twisted cubic] We include one more cautionary example.  Let $C$ denote the twisted cubic in $\PP^3$; i.e. the image of the map $[t^3,t^2u,tu^2,u^3]\taking \PP^1\to\PP^3$.  Scheme-theoretically, the curve $C$ cannot be defined as the zeroset of two homogeneous polynomials on $\PP^3$, but it can be defined as the zeroset of three homogeneous polynomials on $\PP^3$ (or more precisely 3 sections of the line bundle $\mcO(2)$ on $\PP^3$); namely $C$ is the zero-set of the polynomials $f_1=xz-y^2, f_2=yw-z^2, f_3=xw-yz$.  This might lead one to conclude that $C$ is $3-3=0$ dimensional as a derived manifold (see Axiom \ref{Axiom lm}).

It is true that the zeroset of $f_1,f_2,f_3$ is a zero-dimensional derived manifold, however this zeroset is probably not what one means by $C$.  Instead, one should think of $C$ as locally the zeroset of two functions.  That is, $C$ is the zeroset of a certain section of a rank 2 vector bundle on $\PP^3$.  As such, $C$ is a one-dimensional derived manifold.  

The reason for the discrepancy is this.  The scheme-theoretic intersection does not take into account the dependency relations among $f_1,f_2,f_3$.  While these three functions are globally independent, they are locally dependent everywhere.  That is, for every point $p\in C$, there is an open neighborhood on which two of these three polynomials generate the third (ideal-theoretically).   This is not an issue scheme-theoretically, but it is an issue in the derived setting.  

\end{description}

\end{example}

\section{Main results}\label{main}

In this section, we prove Theorem \ref{refined main theorem}, which says that any simplicial category that satisfies the axioms presented in the last section (Definition \ref{good for doing it}) has the general cup product formula (Definition \ref{general cup}).

Before we do so, we prove an imbedding theorem (Proposition \ref{imbedding}) for compact derived manifolds, which says that any compact derived manifold can be imbedded into a large enough Euclidean space.  The proof very closely mimics the corresponding proof for compact smooth manifolds, except we do not have to worry about the rank of Jacobians.  

Fix a category $\dM$ which is good for doing intersection theory on manifolds, in the sense of Definition \ref{good for doing it}.  In this section we refer to objects in $\dM$ as derived manifolds and to morphisms in $\dM$ as morphisms of derived manifolds.  When we speak of an ``Axiom," we are referring to the axioms of Definition \ref{good for doing it}

Before proving Proposition \ref{imbedding}, let us state a few lemmas.

\begin{lemma}\label{subobject pullback}

Let $f\taking\mcX\to\mcY$ and $g\taking\mcY'\to\mcY$ be morphisms of derived manifolds, such that $g$ is an open subobject.  Then there exists an open subobject $\mcX'\to\mcX$ and a homotopy pullback diagram $$\hPull{\mcX'}{\mcY'}{\mcX}{\mcY.}{}{}{g}{f}$$ 

\end{lemma}

\begin{proof}[Sketch of proof]

Apply $\U$ and let $X'$ denote the preimage of $\U(\mcY')$ in $\U(\mcX)$.  The corresponding open subobject $\mcX'\to\mcX$, guaranteed by Axiom \ref{Axiom os}, satisfies the universal property of the homotopy fiber product. 

\end{proof}

We state one more lemma about imbeddings.

\begin{lemma}\label{imbedding lemma}

\begin{enumerate}\item The pullback of an imbedding is an imbedding.  \item If $M$ is a manifold, $\mcX$ is a derived manifold, and $f\taking\mcX\to M$ is a morphism, then the graph $\Gamma(f)\taking\mcX\to\mcX\cross M$ is an imbedding.\end{enumerate}

\end{lemma}

\begin{proof}

The first result is obvious by Definition \ref{def of imbed}, Lemma \ref{subobject pullback}, and basic properties of pullbacks.  For the second result, we may assume that $\mcX$ is an affine derived manifold and $M=\RR^p$.  We have a homotopy pullback diagram $$\hpull{\mcX}{\RR^0}{\RR^n}{\RR^k}$$ and by Axiom \ref{Axiom i}, the map $f\taking\mcX\to\RR^p$ is homotopic to some composite $\mcX\to\RR^n\To{f'}\RR^p$.  

By Proposition \ref{justification of imbed def}, the graph $\Gamma(f')\taking\RR^n\to\RR^n\cross\RR^p$ is an imbedding, and since the diagram $$\hPull{\mcX}{\mcX\cross\RR^p}{\RR^n}{\RR^n\cross\RR^p}{\Gamma(f)}{}{}{\Gamma(f')}$$ is a homotopy pullback, it follows from the first result that the top map is an imbedding too, as desired.

\end{proof}

The next theorem says that any compact derived manifold can be imbedded into Euclidean space.  This result is proved for smooth manifolds in \cite[II.10.7]{Bre}, and we simply adapt that proof to the derived setting.

\begin{proposition}\label{imbedding}

Let $\dM$ be good for intersection theory, and let $\mcX\in\dM$ be an object whose underlying space $\U(\mcX)$ is compact.  Then $\mcX$ can be imbedded into some Euclidean space $\RR^N$.

\end{proposition}

\begin{proof}

By Axiom \ref{Axiom lm}, we can cover $\mcX$ by local models $\mcV_i$; let $V_i=\U(\mcV_i)$ denote the underlying space.  For each $\mcV_i$, there is a homotopy pullback diagram $$\hPull{\mcV_i}{\RR^0}{\RR^{m_i}}{\RR^{n_i}}{}{x^i}{0}{f^i}$$ in $\dM$.  

For each $i$, let $\beta^i\taking\RR^{m_i}\to\RR$ be a smooth function that is 1 on some open disk $D^{m_i}$ (centered at the origin), $0$ outside of some bigger open disk, and nonnegative 
everywhere.  Define $z^i=\beta^i\circ x^i\taking\mcV_i\to\RR.$  Then $z^i$ is identically 1 on an open subset $V_i'\ss V_i$ (the preimage of $D^{m_i}$) and identically $0$ 
outside some closed neighborhood of $V'_i$ in $V_i$.  Let $\mcV_i'$ denote the open subobject of $\mcV_i$ over $V_i'\ss V_i$.  Define functions $y^i\taking \mcV_i\to\RR^{m_i}$ by multiplication: $y^i=z^ix^i.$

By construction, we have $$y^i|_{\mcV'_i}=x^i|_{\mcV'_i},$$ and each $y^i$ is constantly equal to $0$ outside of a closed neighborhood of 
$V'_i$ in $V_i$.  We can thus extend the $z^i$ and the $y^i$ to all of $\mcX$ by making them zero outside of $V_i$ (Axiom \ref{Axiom u}(b)).  By Axiom \ref{Axiom lm} and the compactness of $X=\U(\mcX)$, we can choose a finite number of indices $i$ so that the $\mcV_i'$ cover all of $\mcX$, say for indices $1\leq i\leq k$.  For each $i$, we have an all-Cartesian diagram \begin{eqnarray}\label{eqn:imbed}\xymatrix{\mcV_i'\ulhlimit\ar[r]\ar[d]_{y^i}&\mcV_i\ulhlimit\ar[r]\ar[d]_{x^i}&\RR^0\ar[d]^0\\ D^{m_i}\ar[r]&\RR^{m_i}\ar[r]_{f^i}&\RR^{n_i},}\end{eqnarray} by Lemma \ref{subobject pullback}.

Let $N=\sum_{i=1}^k m_i$.   Let $b^1\taking\RR^N\to\RR^{m_1}$ denote the first $m_1$ coordinate projections, $b^2\taking\RR^N\to\RR^{m_2}$ the next $m_2$ coordinate projections, and so on for each $1\leq i\leq k$.    The sequence $W=(y^1,\ldots,y^k)$ gives a map $$W\taking\mcX\to\RR^N,$$ such that for each $1\leq i\leq k$ one has $b^i\circ W=y^i$.  

We will show that $W$ is an imbedding in the sense of Definition \ref{def of imbed}; i.e. that the restriction of $W$ to each $\mcV_i'$ comes as the inclusion of the zeroset of smooth functions $c^i$ on a certain open subset of $\RR^N$.  The work has already been done; we just need to tease out what we already have.  The $c^i$ should act like $f^i$ on the relevant coordinates for $\mcV_i'$ and should act like coordinate projections everywhere else.  With that in mind, we define for each $1\leq i\leq k$, the function 
$$c^i=(b^1,\ldots,b^{i-1},f^i,b^{i+1},\ldots,b^k)\taking\RR^N\to\RR^{N-m_i+n_i}.$$  

We construct the following diagram 
$$\xymatrix{\mcV_i'\ulhlimit\ar[r]\ar[d]_{y^i}&\mcV_i\ar[d]_{x^i}\ar[r]\ulhlimit&\RR^0\ar[d]^0\\ 
D^{m_i}\ulhlimit\ar[d]\ar[r]&\RR^{m_i}\ar[d]\ar[r]^{f^i}\ulhlimit&\RR^{n_i}\ar[d]\\ D^{m_i}\cross\RR^{N-m_i}\ar[r]&\RR^N\ar[r]^-{c^i}&\RR^{(N-m_i)+n_i}}$$  The lower right-hand vertical map is a coordinate imbedding.  The lower right square and the lower left square are pullbacks in the category of manifolds, so they are homotopy pullback in $\dM$ by Axiom \ref{Axiom geo}.  The upper squares are pullbacks (see equation \ref{eqn:imbed}).  Therefore the diagram is (homotopy) all-Cartesian.  The vertical composite $\mcV_i'\to D^{m_i}\cross\RR^{N-m_i}$ is the restriction of $W$ to $\mcV_i'$, and it is also the zero set of the horizontal composite $D^{m_i}\cross\RR^{N-m_i}\to\RR^{(N-m_i)+n_i}$.  Since $D^{m_i}\cross\RR^{N-m_i}\to\RR^N$ is the inclusion of an open subset, we have shown that $W$ is an imbedding in the sense of Definition \ref{def of imbed}.

\end{proof}

The following relative version of Proposition \ref{imbedding} is proven in almost exactly the same way.  Recall that a map $f$ of topological spaces is said to be {\em proper} if the inverse image under $f$ of any compact subspace is compact.  A morphism of derived manifolds is said to be proper if the underlying morphism of topological spaces is proper.

\begin{corollary}\label{imbedding 2}

Let $\dM$ be good for intersection theory, let $\mcX\in\dM$ be a derived manifold, and let $M\in\Man$ be a manifold.  Suppose that $f\taking\mcX\to M$ is proper, $A\ss M$ is a compact subset and $A\ss f(\U(\mcX)).$  Let $A'$ denote the interior of $A$, and let $\mcX'=f^\m1(A')$.  There exists an imbedding $W\taking\mcX'\to\RR^N\cross A'$ such that the diagram $$\xymatrix{\mcX'\ar[r]\ar[d]_W&\mcX\ar[dd]^-f\\\RR^N\cross A'\ar[d]_-{\pi_2}\\A'\ar[r]&M}$$ commutes, and the composite $\pi_2\circ W\taking\mcX'\to A'$ is proper.

\end{corollary}

\begin{proof}[Sketch of proof]

Copy the proof of Proposition \ref{imbedding} verbatim until it requires the compactness of $X=\U(\mcX)$, which does not hold in our case.  Instead, use the compactness of $f^\m1(A)$ and choose finitely many indices $1\leq i\leq k$ such that the union $\cup_i\mcV_i'$ contains $f^\m1(A)$.   Continue to copy the proof verbatim, except replace instances of $\mcX$ with $\cup_i\mcV_i'$.  In this way, we prove that one has an imbedding of the derived manifold $$W'\taking\cup_i\mcV_i'\to\RR^N.$$  Let $f'\taking\mcX'\to A'$ be the pullback of $f$, and note that $\mcX'=f^\m1(A')$ is contained in $\cup_i\mcV_i$.  The restriction $W'|_{\mcX'}\taking\mcX'\to\RR^N$ and the map $f'$ together induce a map $W\taking\mcX'\to\RR^N\cross A'$, which is an imbedding by Lemma \ref{imbedding lemma}.  By construction, $f'=\pi_2\circ W$ is the pullback of $f$, so it is proper.

\end{proof}

\subsection{Derived cobordism}

Fix a simplicial category $\dM$ that is good for doing intersection theory on manifolds in the sense of Definition \ref{good for doing it}.  One generally defines cobordism using the idea of manifolds with boundary.  Under the above definitions, manifolds with boundary are {\em not} derived manifolds.

One can define derived manifolds with boundary in a way that emulates Definition \ref{good for doing it}, i.e. one could give axioms for a simplicial category with local models that look like generalizations of manifolds with boundary, etc.  One could then prove that the local definition was equivalent to a global one, i.e. one could prove that derived manifolds with boundary can be imbedded into Euclidean half space.  All of this was the approach of the author's dissertation, \cite{Spi}.  For the current presentation, in the interest of space, we choose to dispense with all that and define cobordism using the idea of proper maps to $\RR$.

To orient the reader to Definition \ref{def:cobordism}, we reformulate the usual cobordism relation for manifolds using this approach.  Two compact smooth manifolds $Z_0$ and $Z_1$ are cobordant if there exists a manifold (without boundary) $X$ and a proper map $f\taking X\to\RR$ such that \begin{itemize}\item for the points $i=0,1\in\RR$, the map $f$ is $i$-collared in the sense of Definition \ref{def:collared}, and \item for $i=0,1$, there is an isomorphism $Z_i\iso f^\m1(i)$, also written $Z_i\iso X_{f=i}$.\end{itemize}  This is precisely the definition we give for derived cobordism, except that we allow $Z_0,Z_1,$ and $X$ to be a derived manifolds.

\begin{definition}\label{def:collared}

Let $\mcX$ denote a derived manifold and $f\taking\mcX\to\RR$ a morphism.  Given a point $i\in\RR$, we say that $f$ is {\em $i$-collared} if there exists an $\epsilon>0$ and an equivalence \begin{eqnarray}\label{eqn collar}\mcX_{|f-i|<\epsilon}\we\mcX_{f=i}\cross(-\epsilon,\epsilon),\end{eqnarray} where $\mcX_{|f-i|<\epsilon}$ denotes the open subobject of $\mcX$ over $(i-\epsilon,i+\epsilon)\ss\RR$ (see Lemma \ref{subobject pullback}). 

\end{definition}

\begin{definition}\label{def:cobordism}

Let $\dM$ be good for intersection theory.  Compact derived manifolds $\mcZ_0$ and $\mcZ_1$ are said to be {\em derived cobordant} if there exists a derived manifold $\mcX$ and a proper map $f\taking\mcX\to\RR$, such that for $i=0$ and for $i=1$, the map $f$ is $i$-collared and there is an equivalence $\mcZ_i\we\mcX_{f=i}$.  The map $f\taking\mcX\to\RR$ is called {\em a derived cobordism between $\mcZ_0$ and $\mcZ_1$}.  We refer to $\mcZ_0\amalg\mcZ_1$ as the {\em boundary} of the cobordism.

If $T$ is a CW complex, then a {\em derived cobordism over $T$} is a pair $(a,f)$, where $f\taking\mcX\to\RR$ is a derived cobordism and $a\taking\U\mcX\to T$ is a continuous map to $T$.

If $T$ is a manifold, we denote the ring whose elements are derived cobordism classes over $T$ (with sum given by disjoint union and product given by fiber product over $T$) by $\Omega^{der}(T)$.  We use $\Omega^{der}$ to denote $\Omega^{der}(\RR^0)$.  

\end{definition}

\begin{remark}

In Definition \ref{def:cobordism} we speak of derived cobordism classes, even though we have not yet shown that derived cobordism is an equivalence relation on compact derived manifolds.  We will prove this fact in Proposition \ref{prop:der cob is equiv}.

\end{remark}

\begin{remark}

We did not define oriented derived manifolds, so all of our cobordism rings $\Omega(T)$ and derived cobordism rings $\Omega^{der}(T)$ should be taken to be unoriented.  

However, the oriented case is no harder, we must simply define oriented derived manifolds and oriented derived cobordisms.  To do so, imbed a compact derived manifold $\mcX$ into some $\RR^n$ (which is possible by Proposition \ref{imbedding}), and consider the normal bundle guaranteed by Axiom \ref{Axiom snb}.  We define an orientation on $\mcX$ to be an orientation on some such normal bundle.  One orients derived cobordisms similarly.  

All results which we prove in the unoriented case also work in the oriented case.

\end{remark}

\begin{remark}\label{cobordism helper}

One can show that if $U$ is an open neighborhood of the closed interval $[0,1]\ss\RR$ then any proper map $f\taking\mcX\to U$ can be extended to a proper map $g\taking\mcX\to\RR$ with an isomorphism $f^\m1([0,1])\iso g^\m1([0,1])$ over $[0,1]$.  So, one can consider such an $f$ to be a derived cobordism between $f^\m1(0)$ and $f^\m1(1)$.

\end{remark}

\begin{proposition}\label{prop:der cob is equiv}

Let $\dM$ be good for intersection theory.  Derived cobordism is an equivalence relation on compact objects in $\dM$.

\end{proposition}

\begin{proof}

Derived cobordism is clearly symmetric.  

To see that it is reflexive, let $\mcX$ be a compact derived manifold and consider the projection map $\mcX\cross\RR\to\RR$. It is a derived cobordism between $\mcX$ and $\mcX$.

Finally, we must show that derived cobordism is transitive.  Suppose that $f\taking\mcX\to\RR$ is a derived cobordism between $\mcZ_0$ and $\mcZ_1$ and that $g\taking\mcY\to\RR$ is a derived cobordism between $\mcZ_1$ and $\mcZ_2$.   By Axiom \ref{Axiom u}, we can glue open subobjects $\mcX_{f<1+\epsilon}\ss\mcX$ and $\mcY_{g>-\epsilon}\ss\mcY$ together along the common open subset $\mcZ_1\cross(-\epsilon,\epsilon)$ to obtain a derived manifold $\mcW$ together with a proper map $h\taking\mcW\to\RR$ such that for $i=0,2$, we have $\mcW_{h=i}=\mcZ_i$.  The result follows.

\end{proof}

\begin{lemma}

Let $\dM$ be good for intersection theory.  The functor $\i\taking\Man\to\dM$ induces a homomorphism on cobordism rings $$\i_*\taking\Omega\to\Omega^{der}.$$

\end{lemma}

\begin{proof}

It suffices to show that manifolds which are cobordant are derived cobordant.  If $Z_0$ and $Z_1$ are manifolds which are cobordant, then there exists a compact manifold $X$ with boundary $Z_0\amalg Z_1$.  It is well-known that one can imbed $X$ into $\RR^n$ in such a way that for $i=0,1$ we have $Z_i\iso X_{x_n=i}$, where $x_n$ denotes the last coordinate of $\RR^n$.  Because each $Z_i$ has a collar neighborhood, we can assume that $X_{|x_n-i|<\epsilon}\iso X_{x_n=i}\cross(-\epsilon,\epsilon)$ for some $\epsilon>0$.  The preimage $\wt{X}\ss X$ of $(-\epsilon,1+\epsilon)$ is a manifold (without boundary), hence a derived manifold (under $\i$).  By Remark \ref{cobordism helper}, $x_n\taking \wt{X}\to\RR$ is a derived cobordism between $Z_0$ and $Z_1$.

\end{proof}

\begin{theorem}\label{cobord iso der cobord}

Let $\dM$ be good for intersection theory.  The functor $$\i_*\taking\Omega\to\Omega^{der}$$ is an isomorphism.

\end{theorem}

\begin{proof}

We first show that $\i_*$ is surjective; i.e. that every compact derived manifold $\mcZ$ is derived cobordant to a smooth manifold.  Let $W\taking\mcZ\to\RR^N$ denote a closed imbedding, which exists by Proposition \ref{imbedding}.  Let $U\ss \RR^N$, $E\to U$, and $s,z\taking U\to E$ be the open neighborhood, vector bundle, zero section, and defining section from Axiom \ref{Axiom snb}, so that the diagram $$\hPull{\mcZ}{U}{U}{E,}{g}{g}{z}{s}$$ is a homotopy pullback.  Note that $U$ and $E$ are smooth manifolds, and that the image $z(U)\ss E$ is a closed subset.

Since $\mcZ$ is compact, we can choose a compact subset $U'\ss U$ whose interior contains $\mcZ$.  Let $A\ss U$ be the compliment of the interior of $U'$.  Then $s(A)\cap z(U)=\emptyset$, so in particular they are transverse as closed submanifolds of $E$.  

By \cite[App 2, p. 24]{Sto}, there exists a regular homotopy $H\taking [0,1]\cross U\to E$ such that \begin{itemize}\item $H_0=s\taking U\to E$,\item $H_1=t$, where $t\taking U\to E$ is transverse to the closed subset $z(U)\ss E$, and \item for all $0\leq i\leq 1$, we have $H_i|_A=s|_A$.
\end{itemize}  For any $\epsilon>0$ we can extend $H$ to a homotopy $H\taking (-\epsilon,1+\epsilon)\cross U\to E$, with the same three bulleted properties, by making it constant on $(-\epsilon,0]$ and $[1,1+\epsilon)$.  

Consider the all-Cartesian diagram $$\xymatrix{\mcZ_i\ulhlimit\ar[r]\ar[d]&\mcP\ulhlimit\ar[r]\ar[d]_a&U\ar[d]^z\\ \RR^0\cross U\ulhlimit\ar[r]^-{i\cross\id_U}\ar[d]_\pi&(-\epsilon,1+\epsilon)\cross U\ar[r]^-H\ar[d]^\pi&E\\ \RR^0\ar[r]_i&\RR&}$$ for $i=0,1$.  Notice that $\mcZ_0=\mcZ$ and that $\mcZ_1$ is a smooth manifold.  It suffices to show that the map $\pi\circ a\taking\mcP\to\RR$ is proper, i.e. that for each $0\leq i\leq 1$, the intersection of $H_i(U)$ and $z(U)$ is compact.  Since $H_i(A)\cap z(U)=\emptyset$, one has $$H_i(U)\cap z(U)=H_i(U')\cap z(U),$$ and the right hand side is compact.  Thus $H$ is a derived cobordism between $\mcZ$ and a smooth manifold.

To prove that $\i_*$ is injective, we must show that if smooth manifolds $M_0,M_1$ are derived cobordant, then they are smoothly cobordant.  Let $f\taking\mcX\to\RR$ denote a derived cobordism between $M_0$ and $M_1$.  Let $A'$ denote the open interval $(-1,2)\ss\RR$, and let $\mcX'=f^\m1(A').$  By Corollary \ref{imbedding 2} we can find an imbedding $W\taking\mcX'\to\RR^N\cross A'$ such that $$f':=f|_{\mcX'}=\pi_2\circ W$$ is proper.  Note that for $i=0,1$ we still have $\mcX'_{f'=i}=M_i$.  By Remark \ref{cobordism helper}, $f'$ induces a derived cobordism between $M_0$ and $M_1$ with the added benefit of factoring through an imbedding $\mcX'\to\RR^N$ into Euclidean space.

Again we use Axiom \ref{Axiom snb} to find a vector bundle and section $s$ on an open subset of $\RR^N\cross\RR$ whose zeroset is $\mcX'$.  We apply \cite[App 2, p.24]{Sto} to find a regular homotopy between $s$ and a section $t$ which is transverse to zero, all the while keeping the collar closed submanifold $M_0\cross (-\epsilon,\epsilon)\amalg M_1\cross(1-\epsilon,1+\epsilon)$ fixed.  The zeroset of $t$ is a smooth cobordism between $M_0$ and $M_1$.

\end{proof}

\begin{corollary}\label{cor: refined main theorem}

If $\dM$ is good for doing intersection theory on manifolds, then $\dM$ has the general cup product formula in cobordism, in the sense of Definitions \ref{def geo} and \ref{general cup}.  Moreover, for any manifold $T$, the functor $$\i_*\taking\Omega(T)\to\Omega^{der}(T)$$ is an isomorphism between the classical cobordism ring and the derived cobordism ring (over $T$).

\end{corollary}

\begin{proof}

Suppose $\dM$ is good for doing intersection theory on manifolds.  Since $T$ is assumed to be a CW complex, if $Z\ss X$ is a closed subset of a metrizable topological space, any map $f\taking Z\to T$ extends to a map $f'\taking U\to T$, where $U\ss X$ is an open neighborhood of $Z$.  One can now modify the proof of Theorem \ref{cobord iso der cobord} so that all constructions are suitably ``over $T$," which implies that $i_*$ is an isomorphism.  Let us quickly explain how to make this modification. 

We begin with a compact derived manifold $\mcZ=(Z,\mcO_Z)$ and a map of topological spaces $\sigma\taking Z\to T$.  Since $Z$ is Hausdorff and paracompact, it is metrizable.  It follows that the map $\sigma$ extends to a map $\sigma'\taking U\to T$, where $U\ss\RR^N$ is open neighborhood of $T$ and $\sigma'|_Z=\sigma$.  By intersecting if necessary, we can take this $U$ to be the open neighborhood given in the proof of Theorem \ref{cobord iso der cobord}.  The vector bundle $p$ is canonically defined over $T$ (via $\sigma'$), as are the sections $s,z$.  One can continue in this way and see that the proof extends without any additional work to the relative setting (over $T$).  

Now that we have proved the second assertion, that $i_*$ is an isomorphism, we go back and show the first, i.e. that $\dM$ has the general cup product formula in cobordism.  The first three conditions of Definition \ref{general cup} follow from Axiom \ref{Axiom geo}, Axiom \ref{Axiom fl}, and Theorem \ref{cobord iso der cobord}.  

To prove the last condition, let $M$ be a manifold and suppose that $j\taking A\to M$ and $k\taking B\to M$ are compact submanifolds.  There is a map $H\taking A\cross\RR\to M$ such that $H_0=j$ and such that $H_1\taking A\to M$ is transverse to $k$.  For $i=0,1$, let $\mcX_i$ denote the intersection of (the images of) $H_i$ and $k$.  Note that $\mcX_0=A\cross_MB$ is a compact derived manifold which we also denote $A\cap B$.  Further, $\mcX_1$ is a compact smooth manifold, often called the ``transverse intersection of $A$ and $B$" (because $A$ is ``made transverse to $B$"), and $\mcX_0$ and $\mcX_1$ are derived cobordant over $M$.  

It is well known that $[\mcX_1]=[A]\smile[B]$ as elements of $\Omega(M)$.  Since $\Omega(M)\iso\Omega^{der}(M)$, and since $[\mcX_0]=[\mcX_1]$ as elements of $\Omega^{der}(M)$, the formula $$[A]\smile [B]=[A\cap  B]$$ holds.

\end{proof}

\section{Layout for the construction of $\dMan$}\label{layout}

In the next few sections we will construct a simplicial category $\dMan$ which is good for doing intersection theory on manifolds, in the sense of Definition \ref{good for doing it}.  An object of $\dMan$ is called a {\em derived manifold} and a morphism in $\dMan$ is called a {\em morphism of derived manifolds}.

A derived manifold could be called a ``homotopical $C^\infty$-scheme of finite type."  In the coming sections we will build up to a precise definition.  In the current section we will give a brief outline of the construction.  

Let us first recall the process by which one defines a scheme in algebraic geometry.  One begins with the category of commutative rings, which can be defined as the category of algebras of a certain algebraic theory (see \cite[3.3.5.a]{Bor2}).  One then defines a ringed space to be a space $X$ together with a sheaf of rings $\mcO_X$ on $X$.  A local ringed space is a ringed space in which all stalks are local rings.  One must then functorially assign to each commutative ring $A$ a local ringed space $\Spec A$, called its prime spectrum.  Once that is done, a scheme is defined as a local ringed space which can be covered by open subobjects, each of which is isomorphic to the prime spectrum of a ring.  Note that the purpose of defining local ringed spaces is that morphisms of schemes have to be morphisms of local ringed spaces, not just morphisms of ringed spaces.

If one is interested only in schemes of finite type (over $\ZZ$), one does not need to define prime spectra for all rings.  One could get away with just defining the category of affine spaces $\AA^n=\Spec \ZZ[x_1,\ldots,x_n]$ as local ringed spaces.  One then uses as local models the fiber products of affine spaces, as taken in the category of local ringed spaces.

We recall quickly the notion of an algebraic theory \cite{Law-Alg}.  An algebraic theory is a category $\TT$ with objects $\{T^i\| i\in\NN\}$, such that $T^i$ is the $i$-fold product of $T^1$.  A $\TT$-algebra is a product preserving functor from $\TT$ to $\Sets$.  For example, the category of rings is the category of algebras on the algebraic theory with objects $\AA_\ZZ^n$ and morphisms given by polynomial maps between affine spaces.  A lax simplicial $\TT$-algebra is a functor from $\TT$ to $\sSets$ that descends to a product-preserving functor on the homotopy category.  

The basic outline of our construction of derived manifolds follows the above construction fairly closely.  The differences are that \begin{itemize}\item Rings are not sufficient as our basic objects.  We need a smooth version of the theory, whose algebras are called $C^\infty$-rings.
\item Everything must be done homotopically.  Our basic objects are in fact lax simplicial $C^\infty$-rings, which means that the defining functor is not required to be ``product preserving on the nose" but instead weakly product preserving. 
\item Furthermore, our sheaves are homotopy sheaves, which means that they satisfy a homotopical version of descent. 
\item Our affine spaces are quite familiar: they are simply the Euclidean spaces $\RR^n$ (as smooth manifolds).  A morphism of affine spaces is a smooth function $\RR^n\to\RR^m$.
\item We use as our local models the homotopy fiber products of affine spaces, as taken in the category of local $C^\infty$-ringed spaces.\end{itemize}

A derived manifold is hence a smooth, homotopical version of a scheme.  In section \ref{cinf} we define our basic objects, the lax $C^\infty$-rings, and prove some lemmas about their relationship with $C^\infty$-rings (in the usual sense) and with commutative rings.  In Section \ref{lrs}, we define local $C^\infty$-ringed spaces and derived manifolds (Definition \ref{def:dman}).  We then must discuss cotangent complexes for derived manifolds in Section \ref{cot}.  In Section \ref{proofs}, we give the proofs of several technical results.  Finally, in Section \ref{gfdi} we prove that our category of derived manifolds is good for doing intersection theory on manifolds, in the sense of Definition \ref{good for doing it}.  

\begin{convention}\label{convention injective}

In this paper we will rely heavily on the theory of model categories and their localizations.  See \cite{Hir} or \cite{Hov} for a good introduction to this subject.  

If $X$ is a category, there are two common model structures on the simplicial presheaf category $\sSets^X$, called the injective and the projective model structures.  In this work, {\bf we always use the injective model structure}.  If we speak of a model structure on $\sSets^X$ without specifying it, we mean the injective model structure.  The injective model structure on $\sSets^X$ is given by object-wise cofibrations, object-wise weak equivalences, and fibrations determined by the right-lifting property with respect to acyclic cofibrations.  

With the injective model structure, $\sSets^X$ is a left proper, combinatorial, simplicial model category.  As stated above, weak equivalences and cofibrations are determined object-wise; in particular every object in $\sSets^X$ is cofibrant.  See, for example, \cite{Barwick} or \cite{Lur-HTT} for further details.

\end{convention}

\section{$C^\infty$-rings}\label{cinf}

Let $\E$ denote the full subcategory of $\Man$ spanned by the Euclidean spaces $\RR^i$, for $i\in\NN$; we refer to $\E$ as {\em the Euclidean category}.  Lawvere, Dubuc, Moerdijk and Reyes, and others have studied $\E$ as an algebraic theory (see \cite{Law-Cat}, \cite{Dub}, \cite{MR}).  The $\E$-algebras, which are defined as the product preserving functors from $\E$ to sets, are called $C^\infty$-rings.  We use a homotopical version, in which we replace sets with simplicial sets, and strictly product-preserving functors with functors which preserve products up to weak equivalence.

Let $\sSets^\E$ denote the simplicial model category of functors from $\E$ to the category of simplicial sets.  As usual (see Convention \ref{convention injective}), we use the injective model structure on $\sSets^\E$, in which cofibrations and weak equivalences are determined object-wise, and fibrations are determined by the right-lifting property. 

For $i\in\NN$, let $H_i\in\sSets^\E$ denote the functor $H_i(\RR^n)=\Hom_\E(\RR^i,\RR^n).$   For each $i,j\in\NN$, let \begin{equation}\label{eqn pij} p_{i,j}\taking H_i\amalg H_j\to H_{i+j}\end{equation} denote the natural map induced by coordinate projections $\RR^{i+j}\to\RR^i$ and $\RR^{i+j}\to\RR^j$.  We also refer to $H_i$ as $H_{\RR^i}$ (see Example \ref{example cinf}).

Note that if $F\taking\EE\to\sSets$ is any functor then the Yoneda lemma gives a natural isomorphism of simplicial sets $$\Map(H_i,F)\iso F(\RR^i).$$

\begin{definition}\label{definition sC}

With notation as above, define the category of {\em lax simplicial $C^\infty$-rings}, denoted $\sC$, to be the localization of $\sSets^\E$ at the set $P=\{p_{i,j}|i,j\in\NN\}$ (see \cite[3.1.1]{Hir}).  

We often refer to the objects of $\sC$ simply as $C^\infty$-rings, dropping the words ``lax simplicial."  We can identify a discrete object in $\sC$ with a (strict) $C^\infty$-ring in the classical sense (see \cite{MR}), and thus we refer to these objects as {\em discrete $C^\infty$-rings.}  

\end{definition}

The model category $\sC$ is a left proper, cofibrantly-generated simplicial model category (\cite[3.4.4]{Hir}), in which all objects are cofibrant.  Let $F\in\sSets^\E$ be a fibrant object, considered as an object of $\sC$.  Then $F$ is fibrant in $\sC$ if and only if it is local with respect to all the $p_{i,j}$; i.e. for each $i,j\in\NN$ the natural map $$\Map(H_{i+j},F)\to\Map(H_i\amalg H_j,F)$$ is a weak equivalence.   Since each $H_i$ is the functor represented by $\RR^i$, this is equivalent to the condition that the natural map $$F(\RR^{i+j})\to F(\RR^i)\cross F(\RR^j)$$ be a weak equivalence. In other words, $F\taking\E\to\sSets$ is fibrant in $\sC$ if and only if it is weakly product preserving (and fibrant as an object of $\sSets^\E$).  

\begin{remark}

There is a model category of strict simplicial $C^\infty$-rings, consisting of (strictly) product preserving functors from $\E$ to $\sSets$.  It is Quillen equivalent to our $\sC$ (see \cite[15.3]{Lur-DAG1}).  The reason we use the lax version is that, as a localization of a left proper model category, it is clearly left proper.  (Note that, by \cite{rezk}, we could have used a simplicial version of the theory to obtain a proper model of strict algebras, but this model is more difficult to use for our purposes.)

This left properness of $\sC$ is necessary for the further localization we use to define the category of homotopy sheaves of $C^\infty$-rings.  The use of homotopy sheaves is one of the key ways in which our theory  differs from the synthetic differential geometry literature.

\end{remark}
 
\begin{proposition}\label{we on und sset}

Suppose that $\phi\taking F\to G$ is a morphism of fibrant $C^\infty$-rings.  Then $\phi$ is a weak equivalence if and only if $\phi(\RR)\taking F(\RR)\to G(\RR)$ is a weak equivalence of simplicial sets.

\end{proposition}

\begin{proof}

Follows from basic facts about Bousfield localization.

\comment{Since $F$ and $G$ are fibrant in $\sC$, they are weakly equivalent in $\sC$ if and only if they are weakly equivalent in $\sSets^\E$.  By definition, this is the case if and only if, for all $i$, the induced map $\phi(\RR^i)\taking F(\RR^i)\to G(\RR^i)$ is a weak equivalence of simplicial sets.  Since $F$ and $G$ are weakly product preserving, the map $\phi(\RR^i)$ is a weak equivalence for all $i$ if and only if $\phi(\RR)$ is a weak equivalence.  The result is proved.}

\end{proof}

Note that if $F\in\sC$ is fibrant and if for all $i\in\NN$ the simplicial set $F(\RR^i)$ is a discrete set, then $F$ is a $C^\infty$-ring in the classical sense -- that is, it is a strictly product preserving functor from $\E$ to $\Sets$.  

\begin{lemma}\label{pi_0 on sC}

The functor $\pi_0\taking\sSets^\EE\to\Sets^\EE$ sends fibrant objects in $\sC$ to $C^\infty$-rings (in the classical sense).  It sends object-wise fibrant homotopy pushout diagrams in $\sC$ to pushouts of $C^\infty$-rings.

\end{lemma}

\begin{proof}

If $F\in\sC$ is fibrant then $\Map(p_{i,j},F)$ is a weak equivalence of simplicial sets for each $p_{i,j}\in P$.  Hence $\pi_0\Map(p_{i,j},F)$ is a bijection of sets.  Since each $p_{i,j}$ is a map between discrete objects in $\sC$, we have $\pi_0\Map(p_{i,j},F)=\Hom(p_{i,j},\pi_0F)$, so $\pi_0F$ is indeed a $C^\infty$-ring.

To prove the second assertion, suppose $$\Psi=(A\From{f}B\To{g}C)$$ is a diagram of fibrant objects in $\sC$.  Factor $f$ as a cofibration $B\to A'$ followed by an acyclic fibration $A'\to A$.  The homotopy colimit of $\Psi$ is given by the usual colimit of the diagram $\Psi'$ of functors $A'\from B\to C$.   Applying $\pi_0$ commutes with taking colimits.  Since $A'$ and $A$ are both fibrant and weakly equivalent in $\sC$, they are weakly equivalent in $\sSets^\EE$, so $\pi_0A'\to\pi_0A$ is an isomorphism in $\Sets^\EE$.  Hence we have $$\pi_0\hocolim(\Psi)\iso\pi_0\colim(\Psi')\iso\colim(\pi_0\Psi')\iso\colim(\pi_0\Psi),$$ completing the proof.

\end{proof}

The following lemma is perhaps unnecessary, but we include it to give the reader more of an idea of how classical (i.e. discrete) $C^\infty$-rings work.

\begin{lemma}\label{cinf rings are rings}

Let $F$ be a discrete fibrant $C^\infty$-ring.  Then the set $F(\RR)$ naturally has the structure of a ring.

\end{lemma}

\begin{proof}

Let $R=F(\RR)\in\Sets$ and note that $F(\RR^i)\iso R^i,$ and in particular $F(\RR^0)=\{*\}$ is a singleton set.  Let $0,1\taking\RR^0\to\RR$ denote the additive and multiplicative units, let $+,\vartimes\taking\RR^2\to\RR$ denote the addition and multiplication functions, and let $\iota\taking\RR\to\RR$ denote the additive inverse.  All of these are smooth functions and are hence morphisms in $\E$.  Applying $F$, we obtain elements $F(0),F(1)\taking\{*\}\to R$, two binary functions $F(+),F(\vartimes)\taking R\cross R\to R$, and a function $F(\iota)\taking R\to R$.  Since all of the ring axioms can be written as the commutativity condition on a diagram, and since $F$ is a functor and preserves commutative diagrams, one sees that the operations $F(0),F(1),F(+),F(\vartimes)$, and $F(\iota)$ satisfy all the axioms for $R$ to be a ring.  

\end{proof}

If $F\in\sC$ is fibrant but not necessarily discrete, then the ring axioms hold up to homotopy.  Note, however,  that even in the discrete case, $R=F(\RR)$ has much more structure than just that of a ring.  Any smooth function $\RR^n\to\RR^m$ gives rise to a function $R^n\to R^m$, satisfying all appropriate commutative diagrams (e.g. the functions $3^{a+b}$ and $3^a3^b$ taking $\RR^2\to\RR$ are equal, so they are sent to equal vertices of the mapping space $\Map(R^2,R)$).

\begin{example}\label{example cinf}

Let $M$ be a manifold.  Let $H_M\taking\E\to\sSets$ be defined by $H_M(\RR^i):=\Hom_\Man(M,\RR^i)$.  One checks easily that $H_M$ is a discrete fibrant $C^\infty$-ring.  We usually denote $H_M$ by $C^\infty(M)$, though this notation tends to obscure the role of $H_M$ as a functor, instead highlighting the value of $H_M$ on $\RR$, the set of $C^\infty$-functions $M\to\RR$.

Since $\sC$ is a model category, it is closed under homotopy colimits.  The homotopy colimit of the diagram $C^\infty(\RR^0)\from C^\infty(\RR)\to C^\infty(\RR^0)$, induced by the diagram of manifolds $\RR^0\To{0}\RR\From{0}\RR^0$, is an example of a non-discrete $C^\infty$-ring.

\end{example}

\begin{definition}

Let $\E^{alg}$ denote the subcategory of $\E$ whose objects are the Euclidean spaces $\RR^i$, but in which we take as morphisms only those maps $\RR^i\to\RR^j$ which are given by $j$ polynomials in the $i$ coordinate functions on $\RR^i$.  For each $i,j\in\NN$, let $H_i(\RR^j)=\Hom_{\E^\alg}(\RR^i,\RR^j)$ and let $p_{i,j}\taking H_i\amalg H_j\to H_{i+j}$ be as above (Equation \ref{eqn pij}).

Let $\sR$ denote the localization of the injective model category $\sSets^{\E^{alg}}$ at the set $\{p_{i,j}|i,j\in\NN\}$.  We call $\sR$ the category of {\em (lax) simplicial $\RR$-algebras}.

The functor $\E^{alg}\to\E$ induces a functor $U\taking\sC\to\sR$, which we refer to as {\em the underlying $\RR$-algebra functor}.  It is a right Quillen functor, and as such preserves fibrant objects.  Similarly, the functor $\sC\to\sSets$ given by $F\mapsto F(\RR)$ is a right Quillen functor, which we refer to as {\em the underlying simplicial set functor}.  Both of these right Quillen functors preserve {\em cofibrations} as well because cofibrations in all three model categories are monomorphisms.

\end{definition}

\begin{corollary}\label{we on und alg}

The functor $U\taking\sC\to\sR$ preserves and reflects weak equivalences between fibrant objects.

\end{corollary}

\begin{proof}

Since $U$ is a right Quillen functor, it preserves fibrant objects.  The result now follows from Proposition \ref{we on und sset} and the corresponding fact about $\sR$.

\end{proof}

We do not need the following proposition, but include it for the reader's convenience and edification.  For example, it may help orient the reader to Section \ref{cot} on cotangent complexes.

\begin{proposition}\label{prop:sr and dgas}

The model category $\sR$ of simplicial $\RR$-algebras is Quillen equivalent to the model category of connective commutative differential graded $\RR$-algebras.

\end{proposition}

\begin{proof}

In \cite[1.1 (3)]{ss}, Schwede and Shipley prove that the model category of connective commutative differential grade $\RR$-algebras is Quillen equivalent to the model category of strict simplicial commutative $\RR$-algebras.  This is in turn Quillen equivalent to the category of lax simplicial commutative $\RR$-algebras by \cite[15.3]{Lur-DAG1}.

\end{proof}

\begin{definition}\label{local cinf rings}

A $C^\infty$-ring $F$ is called {\em local} if its underlying discrete $\RR$-algebra $\pi_0(U(F))$ is local in the usual sense, and a morphism $\phi\taking F\to G$ of local $C^\infty$-rings is called {\em local} if its underlying morphism of discrete $\RR$-algebras $\pi_0(U(\phi))$ is local in the usual sense.

\end{definition}

We will give a more intuitive version of the locality condition in Proposition \ref{equiv locality cond}.  

\begin{remark}

We regret the overuse of the word local; when we ``localize" a model category to add weak equivalences, when we demand a ``locality condition" on the stalks of a ringed space, and later when we talk about derived manifolds as having ``local models" (as the local models for manifolds are Euclidean spaces), we are using the word local in three different ways.  But each is in line with typical usage, so it seems that overloading the word could not be avoided.  

\end{remark}

\section{Local $C^\infty$-ringed spaces and derived manifolds}\label{lrs}

In this section we define the category of (lax simplicial) $C^\infty$-ringed spaces, a subcategory called the category of local $C^\infty$-ringed spaces, and a full subcategory of that called the category of derived manifolds.  These definitions resemble those of ringed spaces, local ringed spaces, and schemes, from algebraic geometry.  Ordinary $C^\infty$-ringed spaces and $C^\infty$-schemes have been studied for quite a while: See \cite{Law-Cat},\cite{Dub}, and \cite{MR}.

Let $X$ be a topological space, and let  $\Op(X)$ denote the category of open inclusions in $X$.   A homotopy sheaf of simplicial sets on $X$ is a functor $F\taking\Op(X)\op\to\sSets$ which is fibrant as an object of $\sSets^{\Op(X)\op}$ and which satisfies a ``homotopy descent condition."  Roughly, the descent condition says that given open sets $U$ and $V$, a section of $F$ over each one, and a {\em choice of homotopy} between the restricted sections on $U\cap V$, there is a homotopically unique section over $U\cup V$ which restricts to the given sections.  Jardine showed in \cite{Jar} that there is a model category structure on $\sSets^{\Op(X)\op}$ in which the fibrant objects are homotopy sheaves on $X$; we denote this model category $\Shv(X,\sSets)$.  

In \cite{DHI}, the authors show that $\Shv(X,\sSets)$ is a localization of the injective model structure on $\sSets^{\Op(X)\op}$ at a certain set of morphisms (called the hypercovers) to obtain $\Shv(X,\sSets)$.  From this, one deduces that $\Shv(X,\sSets)$ is a left proper, cofibrantly generated simplicial model category in which all objects are cofibrant.  A weak equivalence between fibrant objects of $\Shv(X,\sSets)$ is a morphism which restricts to a weak equivalence on every open subset of $X$.

We wish to find a suitable category of homotopy sheaves of $C^\infty$-rings.  This will be obtained as a localization of the injective model structure on $\sSets^{\Op(X)\op\cross\E}$.  

\begin{proposition}

Let $A$ and $B$ be categories, and let $P$ be a set of morphisms in $\sSets^A$ and $Q$ a set of morphisms in $\sSets^B$.  There exists a localization of $\sSets^{A\cross B}$, called {\em the factor-wise localization} and denoted $\mcM$, in which an object $F\in\mcM$ is fibrant if and only if \begin{enumerate}\item $F$ is fibrant as an object of $\sSets^{A\cross B}$\item for each $a\in A$, the induced object $F(a,-)\in\sSets^B$ is $Q$-local, and \item for each $b\in B$, the induced object $F(-,b)\in\sSets^A$ is $P$-local.\end{enumerate}

\end{proposition}

\begin{proof}

The projection $A\cross B\to A$ induces a Quillen pair $$\Adjoint{L_A}{\sSets^A}{\sSets^{A\cross B}}{R_A}$$ (see Convention \ref{convention injective}).  Similarly, there is a Quillen pair $(L_B,R_B)$.  The union of the image of $P$ under $L_A$ and the image of $Q$ under $L_B$ is a set of morphisms in $\sSets^{A\cross B}$, which we denote $R=L_A(P)\amalg L_B(Q)$.   Let $\mcM$ denote the localization of $\sSets^{A\cross B}$ at $R$.  The result follows from \cite[3.1.12]{Hir}.

\end{proof}

\begin{definition}

Let $X$ be a topological space.  Let $Q$ be the set of hypercovers in $\sSets^{\Op(X)\op}$ and let $P=\{p_{i,j}|i,j\in\NN\}$ denote the set of maps from Definition \ref{definition sC}.  We denote by $\Shv(X,\sC)$ the factor-wise localization of $\sSets^{\Op(X)\op\cross\E}$ with respect to $Q$ and $P$, and refer to it as {\em the model category of sheaves of $C^\infty$-rings  on $X$}.

We refer to $F\in\Shv(X,\sC)$ as {\em a sheaf of $C^\infty$-rings on $X$} if $F$ is fibrant; otherwise we simply refer to it as an object of $\Shv(X,\sC)$.

\end{definition}

Similarly, one defines the model category of sheaves of simplicial $\RR$-algebras, denoted $\Shv(X,\sR)$, as the factor-wise localization of $\sSets^{\Op(X)\op\cross\E^{alg}}$ with respect to the same sets,  $Q$ and $P$.

Given a morphism of topological spaces $f\taking X\to Y$, one has push-forward and pullback functors $f_*\taking\Shv(X,\sC)\to\Shv(Y,\sC)$ and $f^\m1\taking\Shv(Y,\sC)\to\Shv(X,\sC)$ as usual.  The functor $f^\m1$ is left adjoint to $f_*$, and this adjunction is a Quillen adjunction.  Recall that when we speak of sheaves on $X$, we always mean fibrant objects in $\Shv(X)$.  To that end, we write $f^*\taking\Shv(Y,\sC)\to\Shv(X,\sC)$ to denote the composition of $f^\m1$ with the fibrant replacement functor.  

\begin{definition}\label{local}

Let $X\in\CG$ be a topological space.  An object $F\in\Shv(X,\sC)$ is called {\em a sheaf of local $C^\infty$-rings on $X$} if \begin{enumerate}\item $F$ is fibrant, and \item for every point $p\taking\{*\}\to X$, the stalk $p^*F$ is a local $C^\infty$-ring.\end{enumerate}  In this case, the pair $(X,F)$ is called a {\em local $C^\infty$-ringed space}.

Let $F$ and $G$ be sheaves of local $C^\infty$-rings on $X$.  A morphism $a\taking F\to G$ is called {\em a local morphism} if, for every point $p\taking \{*\}\to X$, the induced morphism on stalks $p^*(a)\taking p^*(F)\to p^*(G)$ is a local morphism (see Definition \ref{local cinf rings}).  We denote by $\Map_\loc(F,G)$ the simplicial subset of $\Map(F,G)$ spanned by the vertices $a\in\Map(F,G)_0$ which represent local morphisms $a\taking F\to G$.

Let $(X,\mcO_X)$ and $(Y,\mcO_Y)$ denote local $C^\infty$-ringed spaces.  A {\em morphism of local $C^\infty$-ringed spaces} $$(f,f^\sharp)\taking(X,\mcO_X)\to(Y,\mcO_Y)$$ consists of a map of topological spaces $f\taking X\to Y$ and a morphism $f^\sharp\taking f^*\mcO_Y\to\mcO_X$, such that $f^\sharp$ is a local morphism of sheaves of $C^\infty$-rings on $X$.  

More generally, we define a simplicial category $\LRS$ whose objects are the local ringed spaces and whose mapping spaces have as vertices the morphisms of local $C^\infty$-ringed spaces.  Precisely, we define for local $C^\infty$-ringed spaces $(X,\mcO_X)$ and $(Y,\mcO_Y)$ the mapping space $$\Map_\LRS((X,\mcO_X),(Y,\mcO_Y)):=\coprod_{f\in\Hom_{\CG}(X,Y)}\Map_\loc(f^*\mcO_Y,\mcO_X).$$

\end{definition}

\begin{example}

Given a local $C^\infty$-ringed space $(X,\mcO_X)$, any subspace is also a local $C^\infty$-ringed space.  For example, a manifold with boundary is a local $C^\infty$-ringed space.  We do not define derived manifolds with boundary in this paper, but we could do so inside the category of local $C^\infty$-ringed spaces.  In fact, that was done in the author's dissertation \cite{Spi}.

If $i\taking U\ss X$ is the inclusion of an open subset, then we let $\mcO_U=i^*\mcO_X$ be the restricted sheaf, and we refer to the local $C^\infty$-ringed space $(U,\mcO_U)$ as the {\em open subobject of $\mcX$ over $U\ss X$}.

\end{example}

\begin{definition}\label{def:equiv}

A map $(f,f^\sharp)\taking(X,\mcO_X)\to(Y,\mcO_Y)$ is an {\em equivalence of local $C^\infty$-ringed spaces} if $f\taking X\to Y$ is a homeomorphism of topological spaces and $f^\sharp\taking f^*\mcO_Y\to\mcO_X$ is a weak equivalence in $\Shv(X,\sC)$.

\end{definition}

\begin{remark}

The relation which we called equivalence of local $C^\infty$-ringed spaces in Definition \ref{def:equiv} is clearly reflexive and transitive.  Since the sheaves $\mcO_X$ and $\mcO_Y$ are assumed cofibrant-fibrant, it is symmetric as well (see for example \cite[7.5.10]{Hir}).  

\end{remark}

\begin{remark}

Note that the notion of equivalence in Definition \ref{def:equiv} is quite strong.  If $(f,f^\sharp)$ is an equivalence, the underlying map $f$ of spaces is a homeomorphism, and the morphism $\pi_0(f^\sharp)$ is an isomorphism of the underlying sheaves of $C^\infty$-rings.  Therefore equivalence in this sense should not be thought of as a generalization of homotopy equivalence, but instead a generalization of diffeomorphism.

\end{remark}

In the next lemma, we will need to use the {\em mapping cylinder} construction for a morphism $f\taking A\to B$ in a model category.  To define it, factor the morphism $f\amalg\id_B\taking A\amalg B\to B$ as a cofibration followed by an acyclic fibration, $$\xymatrix@1{A\amalg B\;\;\ar@{>->}[r]& Cyl(f)\ar@{->>}^\we[r]& B;}$$ the intermediate object $Cyl(f)$ is called the mapping cylinder.  Note that if $f$ is a weak equivalence, then so are the induced cofibrations $A\to\Cyl(f)$ and $B\to\Cyl(f)$; consequently, $A$ and $B$ are strong deformation retracts of $\Cyl(f)$.

\begin{lemma}\label{lemma:union}

Suppose that $\mcX=(X,\mcO_X)$ and $\mcY=(Y,\mcO_Y)$ are local $C^\infty$-ringed spaces and that $\mcU=(U,\mcO_U)$ and $\mcU'=(U',\mcO_{U'})$ are open subobjects of $\mcX$ and $\mcY$ respectively, and suppose that $\mcU$ and $\mcU'$ are equivalent.  Then in particular there is a homeomorphism $U\iso U'$.  

If the union $X\amalg_UY$ of underlying topological spaces is Hausdorff, then there is a local $C^\infty$-ringed space denoted $\mcX\cup\mcY$ with underlying space $X\amalg_UY$ and structure sheaf $\mcO$, such that the diagram \begin{eqnarray}\label{dia:union}\xymatrix{\mcU\ar[r]\ar[d]\ar[rd]^k&\mcY\ar[d]^j\\\mcX\ar[r]_i&\mcX\cup\mcY}\end{eqnarray} commutes (up to homotopy), and the natural maps $i^*\mcO\to\mcO_X, j^*\mcO\to\mcO_Y,$ and $k^*\mcO\to\mcO_U$ are weak equivalences.

\end{lemma}

\begin{proof}

We define the structure sheaf $\mcO$ as follows.  First, let $V=U$, let $\mcO_V$ denote the mapping cylinder for the equivalence $\mcO_U\to\mcO_{U'}$, and let $\mcV=(V,\mcO_V)$.  Then the natural maps $\mcU,\mcU'\to\mcV$ are equivalences, and we have natural maps $\mcV\to\mcX$ and $\mcV\to\mcY$ which are equivalent to the original subobjects $\mcU\to\mcX$ and $\mcU'\to\mcY$.   

Redefine $k$ to be the natural map $k\taking V\to X\amalg_VY$.  We take $\mcO$ to be the homotopy limit in $\Shv(X\amalg_VY,\sC)$ given by the diagram $$\xymatrix{\mcO\ar[r]\ar[d]\ulhlimit&j_*\mcO_Y\ar[d]\\i_*\mcO_X\ar[r]&k_*\mcO_V.}$$  

On the open set $X\ss X\cup Y$, the structure sheaf $\mcO$ restricts to $\mcO_X(X)\cross_{\mcO_V(V)}\mcO_Y(V)$.  Since $\mcV$ is an open subobject of $\mcY$, we in particular have an equivalence $\mcO_Y(V)\we\mcO_V(V)$.  This implies that $\mcO(X)\to\mcO_X(X)$ is a weak equivalence.  The same holds for any open subset of $X$, so we have a weak equivalence $i^*\mcO\iso\mcO_X$.  Symmetric reasoning implies that $j^*\mcO\to\mcO_Y$ is also a weak equivalence. 

The property of being a local sheaf is local on $X\amalg_VY$, and the open sets in $X$ and in $Y$ together form a basis for the topology on $X\amalg_VY$.  Thus, $\mcO$ is a local sheaf on $X\amalg_VY$.

\end{proof}

\begin{proposition}\label{unions and colimits in LRS}

Let $\mcX, \mcY,\mcU$ and $\mcU'$ be as in Lemma \ref{lemma:union}.  Then Diagram \ref{dia:union} is a homotopy colimit diagram.  

\end{proposition}

\begin{proof}

Let $\mcZ\in\LRS$ denote a local $C^\infty$-ringed space, and let $\mcV\we\mcU\we\mcU'$, $\mcX\cup\mcY = (X\amalg_VY,\mcO)$, and $k\taking V\to X\amalg_VY$ be as in the proof of Lemma \ref{lemma:union}.  We must show that the natural map \begin{align}\label{eqn:col1}\Map(\mcX\cup\mcY,\mcZ)\to\Map(\mcX,\mcZ)\cross_{\Map(\mcV,\mcZ)}\Map(\mcY,\mcZ)\end{align} is a weak equivalence.

Recall that for local $C^\infty$-ringed spaces $\mcA,\mcB$, the mapping space is defined as $$\Map_\LRS(\mcA,\mcB)=\coprod_{f\taking A\to B}\Map_{\loc}(f^*\mcO_B,\mcO_A),$$ i.e. it is a disjoint union of spaces, indexed by maps of underlying topological spaces $A\to B$.  Since $X\amalg_VY$ is the colimit of $X\from V\to Y$ in $\CG$, the indexing set of $\Map_\LRS(\mcX\cup\mcY,\mcZ)$ is the fiber product of the indexing sets for the mapping spaces on the right hand side of Equation \ref{eqn:col1}.  Thus, we may apply Lemma \ref{holim lemma} to reduce to the case of a single index $f\taking X\amalg_VY\to Z$.  That is, we must show that the natural map \small\begin{align}\label{eqn:col2}\Map_\loc(f^*\mcO_Z,\mcO)\to\Map_\loc(i^*f^*\mcO_Z,\mcO_X)\cross_{\Map_\loc(k^*f^*\mcO_Z,\mcO_V)}\Map_\loc(j^*f^*\mcO_Z,\mcO_Y),\end{align}\normalsize indexed by $f$, is a weak equivalence.  By definition of $\mcO$, we have the first weak equivalence in the following display \small\begin{align*}\tag{6.8.3}\Map_\loc(f^*\mcO_Z,\mcO)&\we\Map_\loc(f^*\mcO_Z,i_*\mcO_X\cross_{k_*\mcO_V}j_*\mcO_Y)\\&\we\Map_\loc(f^*\mcO_Z,i_*\mcO_X)\cross_{\Map_\loc(f^*\mcO_Z,k_*\mcO_V)}\Map_\loc(f^*\mcO_Z,j_*\mcO_Y),\end{align*}\normalsize and the second weak equivalence follows from the fact that the property of a map being local is itself a local property.

Note that $i,j,$ and $k$ are open inclusions.  We have reduced to the following: given an open inclusion of topological spaces $\ell\taking A\ss B$ and given sheaves $\mcF$ on $A$ and $\mcG$ on $B$, we have a solid-arrow diagram $$\xymatrix{\Map_\loc(\mcF,\ell_*\mcG)\ar@{-->}[r]\ar[d]&\Map_\loc(\ell^*\mcF,\mcG)\ar[d]\\\Map(\mcF,\ell_*\mcG)\ar[r]_\iso&\Map(\ell^*\mcF,\mcG)}$$ coming from the adjointness of $\ell^*$ and $\ell_*$.  We must show that the dotted map exists and is an isomorphism.

Recall that a simplex is in $\Map_\loc$ if all of its vertices are local morphisms.  So it suffices to show that local morphisms $\mcF\to\ell_*\mcG$ are in bijective correspondence with local morphisms $\ell^*\mcF\to\mcG$. This is easily checked by taking stalks: one simply uses the fact that for any point $a\in A$, a basis of open neighborhoods of $a$ are sent under $\ell$ to a basis of open neighborhoods of $\ell(a)$, and vice versa.  

Now we can see that the right-hand sides of Equations \ref{eqn:col2} and 6.8.3 are equivalent, which shows that Equation \ref{eqn:col1} is indeed a weak equivalence, as desired.

\end{proof}

\begin{definition}\label{def:union}

Let $\mcX, \mcY,\mcU$ and $\mcU'$ be as in Lemma \ref{lemma:union}.  We refer to $\mcX\cup\mcY$ as the {\em union} of $\mcX$ and $\mcY$ along the equivalent local $C^\infty$-ringed spaces $\mcU\we\mcU'$.

\end{definition}

\begin{proposition}\label{i fully faithful}

There is a fully faithful functor $\i\taking\Man\to\LRS$.

\end{proposition}

\begin{proof}

Given a manifold $M$, let $\i(M)$ denote the local $C^\infty$-ringed space $(M,C^\infty_M)$ whose underlying space is that of $M$, and such that for any open set $U\ss M$, the discrete $C^\infty$-ring $C^\infty_M(U)$ is the functor $\EE\to\Sets$ given by $$C^\infty_M(U)(\RR^n)=\Hom_\Man(U,\RR^n).$$  It is easy to check that $C^\infty_M$ is a fibrant object in $\Shv(M,\sC)$.  It satisfies the locality condition in the sense of Definition \ref{local} because, for every point $p$ in $M$, the smooth real-valued functions which are defined on a neighborhood of $p$ but are not invertible in any neighborhood of $p$ are exactly those functions that vanish at $p$.

We need to show that $\i$ takes morphisms of manifolds to morphisms of local $C^\infty$-ringed spaces.   If $f\taking M\to N$ is a smooth map, $p\in M$ is a point, and one has a smooth map $g\taking N\to \RR^m$ such that $g$ has a root in every sufficiently small open neighborhood of $f(p)$, then $g\circ f$ has a root on every sufficiently small neighborhood of $p$. 

It only remains to show that $\i$ is fully faithful.  It is clear by definition that $\i$ is injective on morphisms, so we must show that any local morphism $\i(M)\to\i(N)$ in $\LRS$ comes from a morphism of manifolds.  Suppose that $(f,f^\sharp)\taking (M,C^\infty_M)\to (N,C^\infty_N)$ is a local morphism; we must show that $f\taking M\to N$ is smooth.  This does not even use the locality condition: for every chart $V\ss N$ with $c\taking V\iso \RR^n$, a smooth map $f^\sharp(c)\taking f^\m1(V)\to\RR^n$ is determined, and these are compatible with open inclusions.

\end{proof}

\begin{notation}\label{notation:abs val}

Suppose $(X,\mcO_X)$ is a $C^\infty$-ringed space.  Recall that $\mcO_X$ is a fibrant sheaf of (simplicial) $C^\infty$-rings on $X$, so that for any open subset $U\ss X$, the object $\mcO_X(U)$ is a (weakly) product preserving functor from $\E$ to $\sSets$.  The value which ``matters most" is $$|\mcO_X(U)|:=\mcO_X(U)(\RR),$$ because every other value $\mcO_X(U)(\RR^n)$ is weakly equivalent to an $n$-fold product of it.  

We similarly denote $|F|:=F(\RR)$ for a $C^\infty$-ring $F$.  Lemmas \ref{we on und sset} and \ref{cinf rings are rings} further demonstrate the importance of $|F|$.  

\end{notation}

The theorem below states that for local $C^\infty$-ringed spaces $\mcX=(X,\mcO_X)$, the sheaf $\mcO_X$ holds the answer to the question ``what are the real valued functions on $\mcX$?"

\begin{theorem}\label{structure theorem 1}

Let $(X,\mcO_X)$ be a local $C^\infty$-ringed space, and let $(\RR,C^\infty_\RR)$ denote (the image under $\i$ of) the manifold $\RR\in\Man$.  There is a homotopy equivalence $$\Map_\LRS((X,\mcO_X),(\RR,C^\infty_\RR))\To{\iso} |\mcO_X(X)|.$$

\end{theorem}

We prove Theorem \ref{structure theorem 1} in Section \ref{proofs} as Theorem \ref{structure theorem}.

As in Algebraic Geometry, there is a ``prime spectrum" functor taking $C^\infty$-rings to local $C^\infty$-ringed spaces.  We will not use this construction in any essential way, so we present it as a remark without proof.

\begin{remark}\label{rmk:spec for cinf}

The global sections functor $\Gamma\taking\LRS\to\sC$, given by $\Gamma(X,\mcO_X):=\mcO_X(X)$, has a right adjoint, denoted Spec (see \cite{Dub}).  Given a $C^\infty$-ring $R$, let us briefly explain $\Spec R=(\Spec R,\mcO)$.  The points of the underlying space of $\Spec R$ are the maximal ideals in $\pi_0R$, and a closed set in the topology on $\Spec R$ is a set of points on which some element of $\pi_0R$ vanishes.  The sheaf $\mcO$ assigns to an open set $U\ss\Spec R$ the $C^\infty$-ring $R[\chi_U^\m1]$, where $\chi_U$ is any element of $R$ that vanishes on the complement of $U$.

The unit of the $(\Gamma,\tn{Spec})$ adjunction is a natural transformation, $\eta_\mcX\taking (X,\mcO_X)\to\Spec\mcO_X(X)$.  If $\mcX$ is a manifold then $\eta_\mcX$ is an equivalence of local $C^\infty$-ringed spaces (see \cite[2.8]{MR}).

\end{remark}

The category of local $C^\infty$-ringed spaces contains a full subcategory, called the category of derived smooth manifolds, which we now define.  Basically, it is the full subcategory of $\LRS$ spanned by the local $C^\infty$-ringed spaces that can be covered by affine derived manifolds, where an affine derived manifold is the vanishing set of a smooth function on affine space.  See also Axiom 5 of Definition \ref{good for doing it}.

\begin{definition}\label{def:dman}

An {\em affine derived manifold} is a pair $\mcX=(X,\mcO_X)\in\LRS$, where $\mcO_X\in\sC$ is fibrant, and which can be obtained as the homotopy limit in a diagram of the form $$\hPull{(X,\mcO_X)}{\RR^0}{\RR^n}{\RR^k.}{}{g}{0}{f}$$  We sometimes refer to an affine derived manifold as {\em a local model for derived manifolds.}  We refer to the map $g\taking\mcX\to\RR^n$ as the {\em  canonical inclusion of the zeroset}.

A {\em derived smooth manifold} (or {\em derived manifold}) is a local $C^\infty$-ringed space $(X,\mcO_X)\in\LRS$, where $\mcO_X\in\sC$ is fibrant, and for which there exists an open covering $\cup_i U_i=X$, such that each $(U_i,\mcO_X|_{U_i})$ is an affine derived manifold.  

We denote by $\dMan$ the full subcategory of $\LRS$ spanned by the derived smooth manifolds, and refer to it as the (simplicial) category of derived manifolds.  A morphism of derived manifolds is called an {\em equivalence of derived manifolds} if it is an equivalence of local $C^\infty$-ringed spaces.

\end{definition}

Manifolds canonically have the structure of derived manifolds; more precisely, we have the following lemma.

\begin{lemma}\label{man in dman}

The functor $\i\taking\Man\to\LRS$ factors through $\dMan$.

\end{lemma}

\begin{proof}

Euclidean space $(\RR^n,C^\infty_{\RR^n})$ is a principle derived manifold.

\end{proof}

An imbedding of derived manifolds is a map $g\taking\mcX\to\mcY$ that is locally a zeroset.  Precisely, the definition is given by Definition \ref{def of imbed}, with the simplicial category $\dM$ replaced by $\dMan$.  Note that if $g$ is an imbedding, then the morphism $g^\sharp\taking g^*\mcO_Y\to\mcO_X$ is surjective on $\pi_0$, because it is a pushout of such a morphism.

\section{Cotangent Complexes}\label{cot}

The idea behind cotangent complexes is as follows.  Given a manifold $M$, the cotangent bundle $T^*(M)$ is a vector bundle on $M$, which is dual to the tangent bundle.  A smooth map $f\taking M\to N$ induces a map on vector bundles $$T^*(f)\taking T^*(N)\to T^*(M).$$  Its cokernel, $\coker(T^*(f))$, is not necessarily a vector bundle, and is thus generally not discussed in the basic theory of smooth manifolds.  However, it is a sheaf of modules on $M$ and does have geometric meaning: its sections measure tangent vectors in the fibers of the map $f$.   Interestingly, if $f\taking M\to N$ is an embedding, then $T^*(f)$ is a surjection and the cokernel is zero.  In this case its {\em kernel}, $\ker(T^*(f))$, has meaning instead -- it is dual to the normal bundle to the imbedding $f$.

In general the cotangent complex of $f\taking M\to N$ is a sheaf on $M$ which encodes all of the above information (and more) about $f$.  It is in some sense the ``universal linearization" of $f$.  

The construction of the cotangent complex $L_f$ associated to any morphism $f\taking A\to B$ of commutative rings was worked on by several people, including Andr{\'e} \cite{And}, Quillen \cite{Qui-CCR}, Lichtenbaum and Schlessinger \cite{LS}, and Illusie \cite{Ill}.  Later, Schwede introduced cotangent complexes in more generality \cite{Sch}, by introducing spectra in model categories and showing that $L_f$ emerges in a canonical way when one studies the stabilization of the model category of $A$-algebras over $B$.

The cotangent complex for a morphism of $C^\infty$-ringed spaces can be obtained using the process given in Schwede's paper \cite{Sch}.  Let us call this the $C^\infty$-cotangent complex.  It has all of the usual formal properties of cotangent complexes (analogous to those given in Theorem \ref{properties of cot}).  Unfortunately, to adequately present this notion requires a good bit of setup, namely the construction of spectra in the model category of sheaves of lax simplicial $C^\infty$-rings, following \cite{Sch}.  Moreover, the $C^\infty$-cotangent complex is unfamiliar and quite technical, and it is difficult to compute with.  

Underlying a morphism of $C^\infty$-ringed spaces is a morphism of ringed spaces.  It too has an associated cotangent complex, which we call the ring-theoretic cotangent complex.  Clearly, the $C^\infty$-cotangent complex is more canonical than the ring-theoretic cotangent complex in the setting of $C^\infty$-ringed spaces.    However, for the reasons given in the above paragraph, we opt to use the ring-theoretic version instead.  This version does not have quite as many useful formal properties as does the $C^\infty$-version (Property (4) is weakened), but it has all the properties we will need.

In \cite{Ill}, the cotangent complex for a morphism $A\to B$ of simplicial commutative rings is a simplicial $B$-module.  The model category of simplicial modules over $B$ is equivalent to the model category of non-negatively graded chain complexes over $B$, by the Dold-Kan correspondence.  We use the latter approach for simplicity.

Let us begin with the properties that we will use about ring-theoretic cotangent complexes.  

\begin{theorem}\label{properties of cot}

Let $X$ be a topological space.  Given a morphism $f\taking R\to S$ of sheaves of simplicial commutative rings on $X$, there exists a complex of sheaves of $S$-modules, denoted $L_f$ or $L_{S/R}$, called the {\em cotangent complex associated to $f$}, with the following properties.  \begin{enumerate}\item Let $\Omega^1_f=H_0L_f$ be the $0$th homology group.  Then, as a sheaf on $X$, one has that $\Omega^1_f$ is the usual $S$-module of K\"{a}hler differentials of $S$ over $R$.\item The cotangent complex is functorially related to $f$ in the sense that a morphism of arrows $i\taking f\to f'$, i.e. a commutative diagram of sheaves in $\sR$ $$\xymatrix{R\ar[r]\ar[d]_{f}&R'\ar[d]^{f'}\\ S\ar[r]&S',}$$ induces a morphism of $S$-modules $$L_i\taking L_{f}\to L_{f'}.$$\item If $i\taking f\to f'$ is a weak equivalence (i.e. if both the top and bottom maps in the above square are weak equivalences of simplicial commutative rings), then $L_i\taking L_{f}\to L_{f'}$ is a weak equivalence of $S$-modules, and its adjoint $L_{f}\tensor_SS'\to L_{f'}$ is a weak equivalence of $S'$-modules.\item If the diagram from (2) is a homotopy pushout of sheaves of simplicial commutative rings, then the induced morphsim $L_i\taking L_f\to L_{f'}$ is a weak equivalence of modules.\item To a pair of composable arrows $R\To{f}S\To{g}T$, one can functorially assign an exact triangle in the homotopy category of $T$-modules, $$L_{S/R}\tensor_ST\to L_{T/R}\to L_{T/S}\to L_{S/R}\tensor_ST[-1].$$\end{enumerate}

\end{theorem}

\begin{proof}

The above five properties are proved in Chapter 2 of \cite{Ill}, as Proposition 1.2.4.2, Statement 1.2.3, Proposition 1.2.6.2, Proposition 2.2.1, and Proposition 2.1.2, respectively.

\end{proof}

All of the above properties of the cotangent complex for morphisms of sheaves have contravariant corollaries for morphisms of ringed spaces, and almost all of them have contravariant corollaries for maps of $C^\infty$-ringed spaces.  For example, a pair of composable arrows of $C^\infty$-ringed spaces $\mcX\From{f}\mcY\From{g}\mcZ$ induces an exact triangle in the homotopy category of $\mcO_Z$-modules, $$g^*L_{\mcY/\mcX}\to L_{\mcZ/\mcX}\to L_{\mcZ/\mcY}\to g^*L_{\mcY/\mcX}[-1].$$  It is in this sense that Properties (1), (2), (3), and (5) have contravariant corollaries.

The exception is Property (4); one asks, does a homotopy pullback in the category of $C^\infty$-ringed spaces induce a weak equivalence of cotangent complexes?  In general, the answer is {\em no} because the ringed space underlying a homotopy pullback of $C^\infty$-ringed spaces is not the homotopy pullback of the underlying ringed spaces.  In other words the ``underlying simplicial $\RR$-algebra" functor $U\taking\sC\to\sR$ does not preserve homotopy colimits.  However, there are certain types of homotopy colimits that $U$ does preserve.  In particular, if $\mcX\to\mcZ\from\mcY$ is a diagram of $C^\infty$-ringed spaces in which one of the two maps is an {\em imbedding} of derived manifolds, then taking underlying rings {\em does commute} with taking the homotopy colimit (this follows from Corollary \ref{com fun for imb}), and the cotangent complexes will satisfy Property (4) of Theorem \ref{properties of cot}.  This is the only case in which cotangent complexes will be necessary for our work.  Thus, we restrict our attention to the ring-theoretic version rather than the $C^\infty$-version of cotangent complexes.  

Here is our analogue of Theorem \ref{properties of cot}.

\begin{corollary}\label{prop of our cot}

Given a local $C^\infty$-ringed space $\mcX=(X,\mcO_X)$, let $U(\mcO_X)$ denote the underlying sheaf of simplicial commutative rings on $X$.  Given a morphism of local $C^\infty$-ringed spaces $f\taking\mcX\to\mcY$, there exists a complex of sheaves of $U(\mcO_X)$-modules on $X$, denoted $L_f$ or $L_{\mcX/\mcY}$, called the {\em ring-theoretic cotangent complex associated to $f$} (or just the {\em cotangent complex for $f$}), with the following properties.  \begin{enumerate}\item Let $\Omega^1_f=H_0L_f$ be the $0$th homology group.  Then $\Omega^1_f$ is the usual $\mcO_X$-module of K\"{a}hler differentials of $\mcX$ over $\mcY$.\item The cotangent complex is contravariantly related to $f$ in the sense that a morphism of arrows $i=(i_0,i_1)\taking f\to f'$, i.e. a commutative diagram in $\LRS$ $$\xymatrix{\mcY\ar[r]^{i_1}&\mcY'\\\mcX\ar[u]^f\ar[r]_{i_0}&\mcX'\ar[u]_{f'},}$$ induces a morphism of $U(\mcO_X)$-modules $$L_i\taking i_0^*L_{f'}\to L_f.$$  \item If $i\taking f\to f'$ is an equivalence (i.e. it induces two equivalences of sheaves), then $L_i\taking i_0^*L_{f'}\to L_f$ is a weak equivalence of $U(\mcO_X)$-modules.\item If the diagram from (2) is a homotopy pullback in $\LRS$ and either $f'$ or $i_1$ is an immersion, then the induced morphism $L_i\taking L_{f'}\to L_f$ is a weak equivalence of $U(\mcO_X)$-modules.\item To a pair of composable arrows $\mcX\To{f}\mcY\To{g}\mcZ$, one can functorially assign an exact triangle in the homotopy category of $U(\mcO_X)$-modules, $$f^*L_{\mcY/\mcZ}\to L_{\mcX/\mcZ}\to L_{\mcX/\mcY}\to f^*L_{\mcY/\mcZ}[-1].$$ \end{enumerate}

\end{corollary}

\begin{proof}

The map $f\taking\mcX\to\mcY$ induces $U(f^\sharp)\taking U(f^*\mcO_Y)\to U(\mcO_X)$, and we set $L_f:=L_{U(f^\sharp)}$, which is a sheaf of $U(\mcO_X)$-modules. 

Let $h\taking X\to Y$ be a morphism of topological spaces, and let $\mcF\to\mcG$ be a map of sheaves of simplicial commutative rings on $Y$.  By \cite[1.2.3.5]{Ill}, since inductive limits commute with the cotangent complex functor for simplicial commutative rings, one has an isomorphism $$h^*(L_{\mcG/\mcF})\to L_{h^*\mcG/h^*\mcF}.$$  Properties (1), (2), and (5) now follow from Theorem \ref{properties of cot}.

Sheaves on local $C^\infty$-ringed spaces are assumed cofibrant-fibrant, and weak equivalences between fibrant objects are preserved by $U$ (Corollary \ref{we on und alg}).  Property (3) follows from Theorem \ref{properties of cot}.

Finally, if $f'$ or $i_1$ is an immersion, then taking homotopy pullback commutes with taking underlying ringed spaces, by Corollary \ref{com fun for imb}, and so Property (4) also follows from Theorem \ref{properties of cot}.

\end{proof}

If $\mcX$ is a local $C^\infty$-ringed space, we write $L_\mcX$ to denote the cotangent complex associated to the unique map $t\taking\mcX\to\RR^0$.  It is called the {\em absolute cotangent complex} associated to $\mcX$.

\begin{corollary}\label{cot for euc}

Let $t\taking\RR^n\to\RR^0$ be the unique map.  Then the cotangent complex $L_t$ is 0-truncated, and its $0$th homology group $$\Omega^1_t\iso C^\infty_{\RR^n}\langle dx_1,\ldots,dx_n\rangle$$ is a free $C^\infty_{\RR^n}$ module of rank $n$.  

Let $p\taking\RR^0\to\RR^n$ be any point.  Then the cotangent complex $L_p$ on $\RR^0$ has homology concentrated in degree 1, and $H_1(L_p)$ is an $n$-dimensional real vector space.

\end{corollary}

\begin{proof}

The first statement follows from Property (1) of Theorem \ref{properties of cot}.  The second statement follows from the exact triangle, given by Property (5), for the composable arrows $\RR^0\To{p}\RR^n\To{t}\RR^0$.

\end{proof}

Let $\mcX=(X,\mcO_X)$ be a derived manifold and $L_\mcX$ its cotangent complex.  For any point $x\in X$, let $L_{\mcX,x}$ denote the stalk of $L_\mcX$ at $x$; it is a chain complex over the field $\RR$ so its homology groups are vector spaces.  Let $e(x)$ denote the alternating sum of the dimensions of these vector spaces.  As defined, $e\taking X\to\ZZ$ is a just function between sets.

\begin{corollary}\label{euler of der}

Let $\mcX=(X,\mcO_X)$ be a derived manifold, $L_\mcX$ its cotangent complex, and $e\taking X\to\ZZ$ the pointwise Euler characteristic of $L_\mcX$ defined above.  Then $e$ is continuous (i.e. locally constant), and for all $i\geq 2$ we have $H_i(L_\mcX)=0$. 

\end{corollary}

\begin{proof}

We can assume that $\mcX$ is an affine derived manifold; i.e. that there is a homotopy limit square of the form $$\xymatrix{\mcX\ar[r]^t\ar[d]_i\ulhlimit&\RR^0\ar[d]^0\\\RR^n\ar[r]_f&\RR^k.}$$  By Property (4), the map $L_i\taking L_t\to L_f$ is a weak equivalence of sheaves.   Recall that $L_\mcX$ is shorthand for $L_t$, so it suffices to show that the Euler characteristic of $L_f$ is constant on $\mcX$.

The composable pair of morphisms $\RR^n\To{f}\RR^k\to\RR^0$ induces an exact triangle $$f^*L_{\RR^k}\to L_{\RR^n}\to L_f\to f^*L_{\RR^k}[-1].$$  By Corollary \ref{cot for euc}, this reduces to an exact sequence of real vector spaces \begin{align}\label{eqn:euler}0\to H_1(L_f)\to \RR^k\to \RR^n\to H_0(L_f)\to 0.\end{align}  Note also that for all $i\geq 2$ we have $H_i(L_f)=0$, proving the second assertion.  The first assertion follows from the exactness of (\ref{eqn:euler}), because $$\rank(H_0(L_f))-\rank(H_1(L_f))=n-k$$ at all points in $\mcX$.

\end{proof}

\begin{definition}\label{def:dimension}

Let $\mcX=(X,\mcO_X)$ be a derived manifold, and $e\taking X\to\ZZ$ the function defined in Corollary \ref{euler of der}.  For any point $x\in X$, the value $e(x)\in\ZZ$ is called {\em the virtual dimension}, or just {\em the dimension}, of $\mcX$ at $x$, and denoted $\dim_x\mcX$.

\end{definition}

\begin{corollary}\label{cot for imb}

Suppose that $\mcX$ is a derived manifold and $M$ is a smooth manifold.  If $i\taking\mcX\to M$ is an imbedding, then the cotangent complex $L_i$ has homology concentrated in degree 1.

The first homology group $H_1(L_i)$ is a vector bundle on $\mcX$, called {\em the conormal bundle of $i$} and denoted $\mcN_i$ or $\mcN_{\mcX/M}$.  The rank of $\mcN_i$ at a point $x\in\mcX$ is given by the formula $$\rank_x\mcN_i=\dim_{i(x)} M-\dim_x\mcX.$$  

In case $\mcX\iso P$ is a smooth manifold, the bundle $\mcN_{P/M}$ is the dual to the usual normal bundle for the imbedding.  In particular, if $i\taking P\to E$ is the zero section of a vector bundle $E\to P$, then $\mcN_{P/E}$ is canonically isomorphic to the dual $E^\vee$ of $E$.

\end{corollary}

\begin{proof}

All but the final assertion can established locally on $\mcX$, and we proceed as follows.  Imbeddings $i$ of derived manifolds are locally of the form $$\xymatrix{\mcX\ar[r]\ar[d]_i\ulhlimit&\RR^0\ar[d]^z\\\RR^n\ar[r]_f&\RR^k.}$$   By Property (4), we have a quasi-isomorphism $L_i\we L_z$.  The claim that $L_i$ is locally free and has homology concentrated in degree 1 now follows from Corollary \ref{cot for euc}.  Note that the conormal bundle $\mcN_i=H_1(L_i)$ has rank $k$.

The exact triangle for the composable morphisms $\mcX\To{i}\RR^n\To{f}\RR^k$ implies that the Euler characteristic of $L_\mcX$ is $n-k$, and the second assertion follows.

For the final assertion, we use the exact sequence $$0\to\mcN_{P/M}\to i^*\Omega^1_M\to\Omega^1_P\to 0.$$

\end{proof}

\subsection{Other calculations}

In this subsection we prove some results which will be useful later.

\begin{lemma}\label{Nakayama}

Suppose given a diagram $$\xymatrix{\mcX\ar[r]^g\ar[d]_f&\mcX'\ar[dl]^{f'}\\ \mcY&}$$ of local smooth-ringed spaces such that $g,f,$ and 
$f'$ are closed immersions, the underlying map of topological spaces $g:X\to X'$ is a homeomorphism, and the induced map $g^*L_{f'}\to L_f$ is a quasi-isomorphism.  Then $g$ is an equivalence.

\end{lemma}

\begin{proof}

It suffices to prove this on stalks; thus we may assume that $X$ and $X'$ are points.  Let $U_X$ and $U_{X'}$ denote the local simplicial commutative rings underlying $\mcO_X$ and $\mcO_{X'}$, and let $U_g\taking U_{X'}\to U_X$ denote the map induced by $g$.  By Corollary \ref{we on und alg}, it suffices to show that $U_g$ is a weak equivalence.  

Since $g$ is a closed immersion, $U_g$ is surjective; let $I$ be the kernel of $U_g$.  It is proved in Corollary \cite[III.1.2.8.1]{Ill} that the conormal bundle $H_1(L_{U_g})$ is isomorphic to $I/I^2$.  By the distinguished triangle associated to the composition $f'\circ g=f$, we find that $L_{U_g}=0$, so $I/I^2=0$.  Thus, by Nakayama's lemma, $I$=0, so $U_g$ is indeed a weak equivalence.

\end{proof}

\begin{proposition}\label{fib pro cot iso}

Suppose that $p\taking E\to M$ is a vector bundle.  Suppose that $s:M\to E$ is a section of $p$ such that the diagram \begin{eqnarray}\label{dia:fib 
pro}\Sq{(X,\mcO_X)}{M}{M}{E}{f}{f}{s}{z}\end{eqnarray}
commutes, and that the diagram of spaces underlying (\ref{dia:fib pro}) is a pullback; i.e. $X=M\cross_EM$ in $\CG$.  The diagram induces a  
morphism of vector bundles $$\lambda_s:f^*(E)^\vee\to \mcN_f,$$ which is an isomorphism if and only if Diagram \ref{dia:fib pro} is a pullback in $\LRS$.

\end{proposition}

\begin{proof}

On any open subset $U$ of $X$, we have a commutative square $$\xymatrix{\mcO_E(U)\ar[r]\ar[d]&\mcO_M(U)\ar[d]\\\mcO_M(U)\ar[r]&\mcO_X(U)}$$ of sheaves on $X$.  By Property (2) of cotangent complexes, this square induces a morphism $f^*L_z\to L_f$.  By Corollary \ref{cot for imb}, we may identify $H_1(f^*L_z)$ with $f^*(E)^\vee$, and $H_1(L_f)$ with $\mcN_f$, and we let $\lambda_s\taking f^*(E)^\vee\to\mcN_f$ denote the induced map.

If Diagram \ref{dia:fib pro} is a pullback then $\lambda_s$ is an isomorphism by Property (4), so we have only to prove the converse.

Suppose that $\lambda_s$ is an isomorphism, let $\mcX'$ be the fiber product in the diagram 
$$\hPull{\mcX'}{M}{M}{E,}{f'}{f'}{s}{z}$$ and let $g:\mcX\to \mcX'$ be the induced map.  Note that on underlying spaces $g$ is a homeomorphism.  Since the composition 
$X\To{g} 
X'\To{f'}M$, namely $f$, is a closed immersion, so is $g$.  By assumption and Property (4), $g$ induces an isomorphism on cotangent complexes 
$g^*L_{f'}\To{\iso}f^*L_z\To{\iso}L_f$. The result follows from Lemma \ref{Nakayama}.

\end{proof}

Here is a kind of linearity result.  

\begin{proposition}\label{linearity in sections}

Suppose that $M$ is a manifold and $\mcX_1$ and $\mcX_2$ are derived manifolds which are defined as pullbacks $$\xymatrix{\mcX_1\ar[r]^{j_1}\ar[d]_{j_1}\ulhlimit&M\ar[d]^{z_1}&&\mcX_2\ar[r]^{j_2}\ar[d]_{j_2}\ulhlimit&M\ar[d]^{z_2}\\M\ar[r]_{s_1}&E_1&&M\ar[r]_{s_2}&E_2,}$$ where $E_i\to M$ is a vector bundle, $s_i$ is a section, and $z_i$ is the zero section, for $i=1,2$.  Then $\mcX_1$ and $\mcX_2$ are equivalent as derived manifolds over $M$ if and only if there exists an open neighborhood $U\ss M$ containing both $\mcX_1$ and $\mcX_2$, and an isomorphism $\sigma\taking E_1\to E_2$ over $U$, such that $s_2|_U=\sigma\circ s_1|_U$.

\end{proposition}

\begin{proof}

Suppose first $U\ss M$ is an open neighborhood of both $\mcX_1$ and $\mcX_2$, and denote $E_i|_U$ and $s_i|_U$ by $E_i$ and $s_i$, respectively for $i=1,2$.  Suppose that $\sigma\taking E_1\to E_2$ is an isomorphism.  Suppose that $s_2=\sigma\circ s_1$.  Consider the diagram $$\xymatrix{\mcX_1\ulhlimit\ar[r]^i\ar[d]_i&U\ar@{=}[r]\ar[d]^{z_1}&U\ar[d]^{z_2}\\U\ar[r]_-{s_1}&E_1\ar[r]_\sigma&E_2,}$$ where the left-hand square is Cartesian.  In fact, the right-hand square is Cartesian as well, because $\sigma$ is an isomorphism of vector bundles, and in particular fixes the zero section.  Thus we see that $\mcX_1$ and $\mcX_2$ are equivalent as derived manifolds over $U$.  

For the converse, suppose that we have an equivalence $f\taking\mcX_2\to\mcX_1$ such that $j_2=j_1\circ f$.  In particular, on underlying topological spaces we have $X_1=X_2=:X$, and on cotangent complexes we have a quasi-isomorphism $f^*L_{j_1}\To{\iso} L_{j_2}$.  In particular, this implies that $H_1(L_{z_1})$ is isomorphic to $H_1(L_{z_2})$, so $E_1$ and $E_2$ are isomorphic vector bundles on $U$ by Proposition \ref{fib pro cot iso}.  We can write $E$ to denote both bundles, and adjust the sections as necessary.

Consider the diagram $$\xymatrix@=15pt{&\mcX_1\ar[rr]^(.6){j_1}\ar[dd]_(.7){j_1}&&M\ar[dd]_(.6){z_1}\ar@{=}[dl]\\\mcX_2\ar[ru]^f\ar[rr]^(.7){j_2}\ar[dd]_(.6){j_2}&&M\ar[dd]_(.7){z_2}\\&M\ar@{=}[dl]\ar[rr]_(.7){s_1}&&E\\M\ar[rr]_(.6){s_2}&&E&}$$  By the commutativity of the left-hand square of the diagram and by Property (4) of cotangent complexes, one has a chain of quasi-isomorphisms $$j_2^*L_{s_1}=f^*j_1^*L_{s_1}\To{\iso}f^*L_{j_1}\To{\iso}L_{j_2}\To{\iso}j_2^*L_{s_2}.$$  

We write part of the long-exact sequences coming from the composable arrows $U\To{s_1}E\to\RR^0$ and $U\To{s_2}E\to\RR^0$ and furnish some morphisms to obtain the solid-arrow diagram: $$\xymatrix{0\ar[r]&H_1(L_{s_1})\ar[r]\ar[d]_\iso&\Omega^1_{E}\ar[r]^{s_1^*}\ar@{..>}[d]_\tau&\Omega^1_U\ar@{=}[d]\ar[r]&0\\0\ar[r]&H_1(L_{s_2})\ar[r]&\Omega^1_{E}\ar[r]_{s_2^*}&\Omega^1_U\ar[r]&0.}$$  By the 5-lemma, there is an isomorphism $\tau\taking\Omega^1_E\to\Omega^1_E$ making the diagrams commute. 

Now every section $s\taking U\to E$ of a vector bundle $E\to U$ induces a pullback map $s^*\taking\Omega^1_E\to\Omega^1_U$, and two sections induce equivalent pullback maps if and only if they differ by an automorphism of $E$ fixing $U$.  Since $s_1$ and $s_2$ induce equivalent pullback maps, there is an automorphism $\sigma\taking E\to E$ with $s_2=\sigma\circ s_1$.  

\end{proof}

\section{Proofs of technical results}\label{proofs}

Recall from Section \ref{cinf} that $H_{\RR^i}\in\sC$ is the discrete $C^\infty$-ring corepresenting $\RR^i\in\EE$.  A smooth map $f\taking\RR^i\to\RR^j$ (contravariantly) induces a morphism of $C^\infty$-rings $H_{\RR^j}\to H_{\RR^i}$, which we often denote by $f$ for convenience.  Recall also that in $\sC$, there is a canonical weak equivalence $H_{\RR^{i+j}}\To{\we}H_{\RR^i}\amalg H_{\RR^j}.$

Let $U\taking\sC\to\sR$ denote the ``underlying simplicial commutative ring" functor from Corollary \ref{we on und alg}.

\begin{lemma}\label{hadamard}

Let $m,n,p\in\NN$ with $p\leq m$.  Let $x\taking\RR^{n+m}\to\RR^m$ denote the projection onto the first $m$ coordinates, let $g\taking\RR^p\to\RR^m$ denote the inclusion of a $p$-plane in $\RR^m$, and let $\Psi$ be the diagram $$\xymatrix{H_{\RR^m}\ar[r]^x\ar[d]_g&H_{\RR^{n+m}}\\H_{\RR^p}}$$ of $C^\infty$-rings.  The homotopy colimit of $\Psi$ is weakly equivalent to $H_{\RR^{n+p}}$.  

Moreover, the application of $U\taking\sC\to\sR$ commutes with taking homotopy colimit of $\Psi$ in the sense that the natural map $$\hocolim(U\Psi)\to U\hocolim(\Psi)$$ is a weak equivalence of simplicial commutative $\RR$-algebras.

\end{lemma}

\begin{proof}

We may assume that $g$ sends the origin to the origin.  For now, we assume that $p=0$ and $m=1$.

Consider the all-Cartesian diagram of manifolds $$\xymatrix{\RR^n\ar[r]\ar[d]\ullimit&\RR^{n+1}\ar[r]\ar[d]_x\ullimit&\RR^n\ar[d]\\\RR^0\ar[r]_g&\RR\ar[r]&\RR^0.}$$  Apply $H\taking\Man\op\to\sC$ to obtain the diagram $$\xymatrix{H_{\RR^n}&H_{\RR^{n+1}}\ar[l]&H_{\RR^n}\ar[l]\\H_{\RR^0}\ar[u]&H_{\RR}\ar[l]_g\ar[u]^x&H_{\RR^0}\ar[u]\ar[l].}$$  Since $H$ sends products in $\E$ to homotopy pushouts in $\sC$, the right-hand square and the big rectangle are homotopy pushouts.  Hence the left square is as well, proving the first assertion (in case $p=0$ and $m=1$).

For notational reasons, let $U^i$ denote $U(H_{\RR^i})$, so that $U^i$ is the (discrete simplicial) commutative ring whose elements are smooth maps $\RR^i\to\RR$.
Define a simplicial commutative $\RR$-algebra $D$ as the homotopy colimit in the diagram $$\hPush{U^1}{U^{n+1}}{U^0}{D.}{x}{g}{}{}$$  We must show that the natural map $D\to U^n$ is a weak equivalence of simplicial commutative rings.  Recall that the homotopy groups of a homotopy pushout of simplicial commutative rings are the $\Tor$ groups of the corresponding tensor product of chain complexes.

Since $x$ is a nonzerodivisor in the ring $U^{n+1}$, the homotopy groups 
of $D$ are $$\pi_0(D)=U^{n+1}\tensor_{U^1}U^0; \tn{ and } \pi_i(D)=0, i>0.$$ Thus, $\pi_0(D)$ can be identified with the equivalence 
classes of smooth functions $f\taking\RR^{n+1}\to\RR$, where $f\sim g$ if $f-g$ is a multiple of $x$.

On the other hand, we can identify elements of the ring $U^n$ of smooth functions on $\RR^n$  with the equivalence classes of functions $f\taking\RR^{n+1}\to\RR$, where $f\sim g$ if 
$$f(0,x_2,x_3,\ldots x_{n+1})=g(0,x_2,x_3,\ldots,x_{n+1}).$$  Thus to prove that the map $D\to U^n$ is a weak 
equivalence, we must only show that if a smooth function $f(x_1,\ldots,x_{n+1})\taking\RR^{n+1}\to\RR$ vanishes whenever $x_1=0$, then $x_1$ 
divides $f$.  This is called Hadamard's lemma, and it follows from the definition of smooth functions.  Indeed, given a smooth function $f(x_1,\ldots,x_{n+1})\taking\RR^{n+1}\to\RR$ which vanishes whenever $x_1=0$, define a 
function $g\taking\RR^{n+1}\to\RR$ by the formula $$g(x_1,\ldots,x_{n+1})=\lim_{a\to x_1}\frac{f(a,x_2,x_3,\ldots,x_{n+1})}{a}.$$  It is clear 
that $g$ is smooth, and $xg=f$.  

We have now proved both assertions in the case that $p=0$ and $m=1$; we continue to assume $p=0$ and prove the result for general $m+1$ by induction.  The inductive step follows from the all-Cartesian diagram below, in which each vertical arrow is an inclusion of a plane and each horizontal arrow is a coordinate projection: $$\xymatrix{\RR^n\ar[r]\ar[d]\ullimit&\RR^0\ar[d]\\\RR^{m+n}\ar[r]\ar[d]\ullimit&\RR^m\ar[r]\ar[d]\ullimit&\RR^0\ar[d]\\\RR^{n+m+1}\ar[r]&\RR^{m+1}\ar[r]&\RR.}$$  Apply $H$ to this diagram.  The arguments above implies that the two assertions hold for the right square and bottom rectangle; thus they hold for the bottom left square.  The inductive hypothesis implies that the two assertions hold for the top square, so they hold for the left rectangle, as desired.

For the case of general $p$, let $k\taking\RR^m\to\RR^{m-p}$ denote the projection orthogonal to $g$.  Applying $H$ to the all-Cartesian diagram $$\xymatrix{\RR^{n+p}\ar[r]\ar[d]\ullimit&\RR^p\ar[r]\ar[d]^g\ullimit&\RR^0\ar[d]\\\RR^{n+m}\ar[r]_x&\RR^m\ar[r]_k&\RR^{m-p},}$$ the result holds for the right square and the big rectangle (for both of which $p=0$), so it holds for the left square as well.

\end{proof}

\begin{remark}

Note that (the non-formal part of) Lemma \ref{hadamard} relies heavily on the fact that we are dealing with smooth maps.  It is this lemma which fails in the setting of topological manifolds, piecewise linear manifolds, etc.  

\end{remark}

Let $\mcM$ be a model category, and let ${\bf \Delta}$ denote the simplicial indexing category.  Recall that a simplicial resolution of an object $X\in\mcM$ consists of a simplicial diagram $X'\res\taking{\bf\Delta}\op\to\mcM$ and an augmentation map $X'\res\to X$, such that the induced map $\hocolim(X')\to X$ is a weak equivalence in $\mcM$.  

Conversely, the geometric realization of a simplicial object $Y\res\taking{\bf\Delta}\op\to\mcM$ is the object in $\mcM$ obtained by taking the homotopy colimit of the diagram $Y\res$.  

\begin{proposition}\label{und pre geo rea}

The functors $U\taking\sC\to\sR$ and $-(\RR)\taking\sC\to\sSets$ each preserve geometric realizations.

\end{proposition}

\begin{proof}

By \cite[A.1]{Lur-DAG1}, the geometric realization of any simplicial object in any of the model categories $\sC$, $\sR$, and $\sSets$ is equivalent to the diagonal of 
the corresponding bisimplicial object.  Since both $U$ and $-(\RR)$ are functors which preserve the diagonal, they each commute with geometric realization.

\end{proof}

Let $-(\RR)\taking\sC\to\sSets$ denote the functor $F\mapsto F(\RR)$.  It is easy to see that $-(\RR)$ is a right Quillen functor.  Its left adjoint is $(-\tensor H_\RR)\taking K\mapsto K\tensor H_\RR$ (although note that $K\tensor H_\RR$ is not generally fibrant in $\sC$).  We call a $C^\infty$-ring {\em free} if it is in the essential image of this functor $-\tensor H_\RR$, and similarly, we call a morphism of free $C^\infty$-rings {\em free} if it is in the image of this functor.

Thus, a morphism $d'\taking H_{\RR^k}\to H_{\RR^\ell}$ is free if and only if it is induced by a function $d\taking\{1,\ldots,k\}\to\{1,\ldots,\ell\}$.  To make $d'$ explicit, we just need to provide a dual map $\RR^\ell\to\RR^k$; for each $1\leq i\leq k$, we supply the map $\RR^\ell\to\RR$ given by projection onto the $d(i)$ coordinate.  

If $K$ is a simplicial set such that each $K_i$ is a finite set with cardinality $|K_i|=n_i$, then $X:=K\tensor H_\RR$ is the simplicial $C^\infty$-ring with $X_i=H_{\RR^{n_i}}$, and for a map $[\ell]\to[k]$ in $\bf{\Delta}$, the structure map $X_k\to X_\ell$ is the free map defined by $K_k\to K_\ell$.

\begin{lemma}\label{fun sim res}

Let $X$ be a simplicial $C^\infty$-ring.  There exists a functorial simplicial resolution $X'\res\to X$ in which $X'_n$ is a free simplicial $C^\infty$-ring for each $n$.  

Moreover, if $f\taking X\to Y$ is a morphism of $C^\infty$-rings and $f'\res\taking X'\res\to Y'\res$ is the induced map on simplicial resolutions, then for each $n\in\NN$, the map $f'_n\taking X'_n\to Y'_n$ is a morphism of free $C^\infty$-rings.

\end{lemma}

\begin{proof}

Let $R$ denote the underlying simplicial set functor $-(\RR):\sC\to\sSets$, and let $F$ denote its left adjoint.  The 
comonad $FR$ gives rise to an augmented simplicial $C^\infty$-ring, 
$$\xymatrix@1{\Phi=\cdots\arrr{r}&FRFR(X)\arr{r}&FR(X)\ar[r]&X.}$$  By \cite[8.6.10]{Wei}, the induced augmented simplicial set 
$$\xymatrix@1{R\Phi=\cdots\arrr{r}&RFRFR(X)\arr{r}&RFR(X)\ar[r]&R(X)}$$ is a weak equivalence.  By Propositions \ref{we on und alg} and \ref{und pre geo rea}, $\Phi$ is a simplicial resolution.

The second assertion is clear by construction.

\end{proof}

In the following theorem, we use a basic fact about simplicial sets: if $g\taking F\to H$ is a fibration of simplicial sets and $\pi_0g$ is a surjection of sets, then for each $n\in\NN$ the function $g_n\taking F_n\to H_n$ is surjective.  This is proved using the left lifting property for the cone point inclusion $\Delta^0\to\Delta^{n+1}$.

\begin{theorem}\label{commuting functors}

Let $\Psi$ be a diagram $$\xymatrix{F\ar[r]^f\ar[d]_g&G\\H}$$ of cofibrant-fibrant $C^\infty$-rings.  Assume that $\pi_0g\taking\pi_0F\to\pi_0H$ is a surjection.  Then application of $U\taking\sC\to\sR$ commutes with taking the homotopy colimit of $\Psi$ in the sense that the natural map $$\hocolim(U\Psi)\to U\hocolim(\Psi)$$ is a weak equivalence of simplicial commutative rings.

\end{theorem}

\begin{proof}

We prove the result by using simplicial resolutions to reduce to the case proved in Lemma \ref{hadamard}.  We  begin with a series of replacements and simplifications of the diagram $\Psi$, each of which preserves both $\hocolim(U\Psi)$ and $U\hocolim(\Psi)$.  

First, replace $g$ with a fibration and $f$ with a cofibration.  Note that since $g$ is surjective on $\pi_0$ and is a fibration, it is surjective in each degree; note also that $f$ is injective in every degree.  

Next, replace the diagram by the simplicial resolution given by Lemma \ref{fun sim res}.  This is a diagram $H'\res\From{g'\res}F'\res\To{f'\res}G'\res$, which has the same homotopy colimit.  We can compute this homotopy colimit by first computing the homotopy colimits $\hocolim(H'_n\From{g'_n}F'_n\To{f'_n}G'_n)$ in each degree, and then taking the geometric realization.   Since $U$ preserves geometric realization (Lemma \ref{und pre geo rea}), we may assume that $\Psi$ is a diagram $H\From{g}F\To{f}G$, in which $F,G,$ and $H$ are free simplicial $C^\infty$-rings, and in which $g$ is surjective and $f$ is injective.  By performing another simplicial resolution, we may assume further that $F,G,$ and $H$ are discrete. 

A free $C^\infty$-ring is the filtered colimit of its finitely generated (free) sub-$C^\infty$-rings; hence we may assume that $F,G,$ and $H$ are finitely generated.  In other words, each is of the form $S\tensor H_\RR$, where $S$ is a finite set.  We are reduced to the case in which $\Psi$ is the diagram $$\xymatrix{H_{\RR^m}\ar[r]^f\ar[d]_g&H_{\RR^n}\\ H_{\RR^p},}$$ where again $g$ is surjective and $f$ is injective.  Since $f$ and $g$ are free maps, they are induced by maps of sets $f_1\taking\{1,\ldots m\}\inj\{1,\ldots,n\}$ and $g_1\taking\{1,\ldots,m\}\to\{1,\ldots,p\}$.  The map $\RR^n\to\RR^m$ underlying $f$ is given by $$(a_1,\ldots,a_n)\mapsto (a_{f(1)},\ldots,a_{f(m)}),$$ which is a projection onto a coordinate plane through the origin, and we may arrange that it is a projection onto the first $m$ coordinates.  The result now follows from Lemma \ref{hadamard}.

\end{proof}

\begin{corollary}\label{com fun for imb}

Suppose that the diagram $$\xymatrix{(A,\mcO_A)\ar[r]^{G}\ar[d]_{F}\ulhlimit&(Y,\mcO_Y)\ar[d]^f\\(X,\mcO_X)\ar[r]_g&(Z,\mcO_Z)}$$ is a homotopy pullback of local $C^\infty$-ringed spaces.  If $g$ is an imbedding then the underlying diagram of ringed spaces is also a homotopy pullback.

\end{corollary}

\begin{proof}

In both the context of local $C^\infty$-ringed spaces and local ringed spaces, the space $A$ is the pullback of the diagram $X\to Z\from Y$.  The sheaf on $A$ is the homotopy colimit of the diagram $$F^*\mcO_X\From{F^*(g^\sharp)} F^*g^*\mcO_Z\To{G^*(f^\sharp)} G^*\mcO_Y,$$ either as a sheaf of $C^\infty$-rings or as a sheaf of simplicial commutative rings.  Note that taking inverse-image sheaves commutes with taking underlying simplicial commutative rings.

Since $g$ is an imbedding, we have seen that $g^\sharp\taking g^*\mcO_Z\to\mcO_X$ is surjective on $\pi_0$, and thus so is its pullback $F^*(g^\sharp)$.  The result now follows from Theorem \ref{commuting functors}.

\end{proof}

We will now give another way of viewing the ``locality condition" on ringed spaces.  Considering sections of the structure sheaf as functions to affine space, a ringed space is local if these functions pull covers back to covers.  This point of view can be found in \cite{SGA4} and \cite{BD}, for example.

Recall the notation $|F|=F(\RR)$ for a $C^\infty$-ring $F$; see Notation \ref{notation:abs val}.  Recall also that $C^\infty(\RR)$ denotes the free (discrete) $C^\infty$-ring on one generator.

\begin{definition}\label{inverse image}

Let $X$ be a topological space and $\mcF$ a sheaf of $C^\infty$-rings on $X$.  Given an open set $U\ss\RR$, let $\chi_U\taking\RR\to\RR$ denote a characteristic function of $U$ (i.e. $\chi_U$ vanishes precisely on $\RR-U$).  Let $f\in|\mcF(X)|_0$ denote a global section.  We will say that an open subset $V\ss X$ is {\em contained in the preimage under $f$ of $U$} if there exists a dotted arrow making the diagram of $C^\infty$-rings \begin{eqnarray}\label{dia:preimage}\xymatrix{C^\infty(\RR)\ar[r]^f\ar[d]&\mcF(X)\ar[d]^{\rho_{X,V}}\\ C^\infty(\RR)[\chi_U^\m1]\ar@{..>}[r]&\mcF(V)}\end{eqnarray} commute up to homotopy.  We say that $V$ {\em is the preimage under $f$ of $U$}, and by abuse of notation write $V=f^\m1(U)$, if it is maximal with respect to being contained in the preimage.  Note that these notions are independent of the choice of characteristic function $\chi_U$ for a given $U\ss\RR$.  Note also that since localization is an epimorphism of $C^\infty$-rings (as it is for ordinary commutative rings; see \cite[2.2,2.6]{MR}), the dotted arrow is unique if it exists.  

If $f,g\in|\mcF(X)|_0$ are homotopic vertices, then for any open subset $U\ss\RR$, one has $f^\m1(U)=g^\m1(U)\ss X$.  Therefore, this preimage functor is well-defined on the set of connected components $\pi_0|\mcF(X)|$.  Furthermore, if $f\in|\mcF(X)|_n$ is any simplex, all of its vertices are found in the same connected component, roughly denoted $\pi_0(f)\in\pi_0|\mcF(X)|$, so we can write $f^\m1(U)$ to denote $\pi_0(f)^\m1(U)$.

\end{definition}

\begin{example}

The above definition can be understood from the viewpoint of algebraic geometry.  Given a scheme $(X,\mcO_X)$ and a global section $f\in\mcO_X(X)=\pi_0\mcO_X(X)$, one can consider $f$ as a scheme morphism from $X$ to the affine line $\AA^1$.  Given a principle open subset $U=\Spec(k[x][g^\m1])\ss\AA^1$, we are interested in its preimage in $X$.  This preimage will be the largest $V\ss X$ on which the map $k[x]\To{f}\mcO_X(X)$ can be lifted to a map $k[x][g^\m1]\to\mcO_X(V)$.  This is the content of Diagram \ref{dia:preimage}.

Up next, we will give an alternate formulation of the condition that a sheaf of $C^\infty$-rings be a local in terms of these preimages.  In the algebro-geometric setting, it comes down to the fact that a sheaf of rings $\mcF$ is a sheaf of {\em local rings} on $X$ if and only if, for every global section $f\in\mcF(X)$, the preimages under $f$ of a cover of $\Spec(k[x])$ form a cover of $X$.  

\end{example}

\begin{lemma}\label{ope cov and fac}

Suppose that $U\ss\RR$ is an open subset of $\RR$, and let $\chi_U$ denote a characteristic function for $U$.  Then $C^\infty(U)=C^\infty(\RR)[\chi_U^\m1]$.  

There is a natural bijection between the set of points $p\in\RR$ and the set of $C^\infty$-functions $A_p\taking C^\infty(\RR)\to C^\infty(\RR^0)$.  Under this correspondence, $p$ is in $U$ if and only if $A_p$ factors through $C^\infty(\RR)[\chi_U^\m1]$.

\end{lemma}

\begin{proof}

This follows from \cite[1.5 and 1.6]{MR}.

\end{proof}

Recall that $F\in\sC$ is said to be a local $C^\infty$-ring if the commutative ring underlying $\pi_0F$ is a local ring (see Definition \ref{local cinf rings}).

\begin{proposition}\label{equiv locality cond}

Let $X$ be a space and $\mcF$ a sheaf of finitely presented $C^\infty$-rings on $X$.  Then $\mcF$ is local if and only if, for any cover of $\RR$ by open subsets $\RR=\bigcup_iU_i$ and for any open $V\ss X$ and local section $f\in\pi_0|\mcF(V)|$, the preimages $f^\m1(U_i)$ form a cover of $V$.

\end{proposition}

\begin{proof}

Choose a representative for $f\in\pi_0|F|$, call it $f\in|F|_0$ for simplicity, and recall that it can be identified with a map $f\taking C^\infty(\RR)\to F$, which is unique up to homotopy.

Both the property of being a sheaf of local $C^\infty$-rings and the above ``preimage of a covering is a covering" property is local on $X$.  Thus we may assume that $X$ is a point.  We are reduced to proving that a $C^\infty$-ring $F$ is a local $C^\infty$-ring if and only if, for any cover of $\RR$ by open sets $\RR=\bigcup_iU_i$ and for any element $f\in|F|_0$, there exists an index $i$ and a dotted arrow making the diagram $$\xymatrix{C^\infty(\RR)\ar[d]\ar[r]^f&F\\C^\infty(\RR)[\chi_{U_i}^\m1]\ar@{..>}[ur]&}$$ commute up to homotopy.   

Suppose first that for any cover of $\RR$ by open subsets $\RR=\bigcup_iU_i$ and for any element $f\in F_0$, there exists a lift as above.  Let $U_1=(-\infty,\frac{1}{2})$ and let $U_2=(0,\infty)$.  By assumption, either $f$ factors through $C^\infty(U_1)$ or through $C^\infty(U_2)$, and $1-f$ factors through the other by Lemma \ref{ope cov and fac}.  It is easy to show that any element of $\pi_0F$ which factors through $C^\infty(U_2)$ is invertible (using the fact that $0\not\in U_2$).  Hence, either $f$ or $1-f$ is invertible in $\pi_0F$, so $\pi_0F$ is a local $C^\infty$-ring.

Now suppose that $F$ is local, i.e. that it has a unique maximal ideal, and suppose $\RR=\bigcup_iU_i$ is an open cover.  Choose $f\in|F|_0$, considered as a map of $C^\infty$-rings $f\taking C^\infty(\RR)\to F$.  By \cite[3.8]{MR}, $\pi_0F$ has a unique point $F\to\pi_0F\to C^\infty(\RR^0)$. By Lemma \ref{ope cov and fac}, there exists $i$ such that the composition $C^\infty(\RR)\To{f} F\to C^\infty(\RR^0)$ factors through $C^\infty(\RR)[\chi_{U_i}^\m1]$, giving the solid arrow square $$\xymatrix{C^\infty(\RR)\ar[r]^f\ar[d]&F\ar[d]\\C^\infty(\RR)[\chi_{U_i}^\m1]\ar[r]\ar@{..>}[ur]&C^\infty(\RR^0).}$$  Since $\chi_{U_i}\in C^\infty(\RR)$ is not sent to $0\in C^\infty(\RR^0)$, its image $f(\chi_{U_i})$ is not contained in the maximal ideal of $\pi_0F$, so a dotted arrow lift exists making the diagram commute up to homotopy.  This proves the proposition.

\end{proof}

In the following theorem, we use the notation $|A|$ to denote the simplicial set $A(\RR)=\Map_\sC(C^\infty(\RR),A)$ underlying a simplicial $C^\infty$-ring $A\in\sC$.

\begin{theorem}\label{structure theorem}

Let $\mcX=(X,\mcO_X)$ be a local $C^\infty$-ringed space, and let $\i\RR=(\RR,C^\infty_\RR)$ denote the (image under $\i$ of the) real line.  There is a natural homotopy equivalence of simplicial sets $$\Map_\LRS((X,\mcO_X),(\RR,C^\infty_\RR))\To{\we}|\mcO_X(X)|.$$

\end{theorem}

\begin{proof}

We will construct morphisms \begin{align*}K\taking\Map_\LRS(\mcX,\i\RR)\to|\mcO_X(X)| &\tn{, and}\\ L\taking|\mcO_X(X)|\to\Map_\LRS(\mcX,\i\RR),\end{align*} and show that they are homotopy inverses.  For the reader's convenience, we recall the definition $$\Map_\LRS(\mcX,\RR)=\coprod_{f\taking X\to\RR}\Map_\loc(f^*C^\infty_\RR,\mcO_X).$$

The map $K$ is fairly easy and can be defined without use of the locality condition.  Suppose that $\phi\taking X\to\RR$ is a map of topological spaces.  The restriction of $K$ to the corresponding summand of $\Map_\LRS(\mcX,\i\RR)$ is given by taking global sections $$\Map_\loc(\phi^*C^\infty_\RR,\mcO_X)\to\Map(C^\infty(\RR),\mcO_X(X))\iso|\mcO_X(X)|.$$

To define $L$ is a bit harder and depends heavily on the assumption that $\mcO_X$ is a local sheaf on $X$.  First, given an $n$-simplex $g\in|\mcO_X(X)|_n$ we need to define a map of topological spaces $L(g)\taking X\to\RR$.  Let $g_0\in\pi_0|\mcO_X(X)|$ denote the connected component containing $g$.  By Proposition \ref{equiv locality cond}, $g_0$ gives rise to a function from open covers of $\RR$ to open covers of $X$, and this function commutes with refinement of open covers.  Since $\RR$ is Hausdorff, there is a unique map of topological spaces $X\to\RR$, which we take as $L(g)$, consistent with such a function.  Denote $L(g)$ by $G$, for ease of notation.

Now we need to define a map of sheaves of $C^\infty$-rings $$G^\flat\taking C^\infty_\RR\tensor\Delta^n\to G_*(\mcO_X).$$  On global sections, we have such a function already, since $g\in|\mcO_X(X)|_n$ can be considered as a map $g\taking C^\infty_\RR(\RR)\to|\mcO_X(X)|_n$.  Let $V\ss\RR$ denote an open subset and $g^\m1(V)\ss X$ its preimage under $g$.  The map $\rho\taking C^\infty_\RR(\RR)\to C^\infty_\RR(V)$ is a localization; hence, it is an epimorphism of $C^\infty$-rings.  

To define $G^\flat$, we will need to show that there exists a unique dotted arrow making the diagram \begin{eqnarray}\label{dia:structure}\xymatrix{C^\infty_\RR(\RR)\tensor\Delta^n\ar[r]^g\ar[d]_{\rho\tensor\Delta^n}&|\mcO_X(X)|\ar[d]\\C^\infty_\RR(V)\tensor\Delta^n\ar@{..>}[r]&|\mcO_X(g^\m1(V))|}\end{eqnarray} commute.  Such an arrow exists by definition of $g^\m1$.  It is unique because $\rho$ is an epimorphism of $C^\infty$-rings, so $\rho\tensor\Delta^n$ is as well.  We have now defined $G^\flat$, and we take $G^\sharp\taking G^*C^\infty_\RR\to\mcO_X$ to be the left adjoint of $G^\flat$.

We must show that for every $g\in|\mcO_X(X)|_0$, the map $G^\sharp\taking G^*C^\infty_\RR\to\mcO_X$ provided by $L$ is local.  (We can choose $g$ to be a vertex because, by definition, a simplex in $\Map(G^*C^\infty_\RR,\mcO_X)$ is local if all of its vertices are local.)  To prove this, we may take $X$ to be a point, $F=\mcO_X(X)$ a local $C^\infty$-ring, and $x=G(X)\in\RR$ the image point of $X$.  We have a morphism of $C^\infty$-rings $G^\sharp\taking(C^\infty_\RR)_x\to F$, in which both the domain and codomain are local.  It is a local ring homomorphism because all prime ideals in the local ring $(C^\infty_\RR)_x$ are maximal.  

The maps $K$ and $L$ have now been defined, and they are homotopy inverses by construction. 

\end{proof}

The following is a technical lemma that allows us to take homotopy limits component-wise in the model category of simplicial sets.

\begin{lemma}\label{holim lemma}

Let $\sSets$ denote the category of simplicial sets.  Let $I, J,$ and $K$ denote sets and let $$A=\coprod_{i\in I}A_i,\hspace{.4in} B=\coprod_{j\in J}B_j,\hspace{.2in} \tn{and}\hspace{.2in} C=\coprod_{k\in K}C_k$$ denote coproducts in $\sSets$ indexed by $I, J,$ and $K$.  

Suppose that $f\taking I\to J$ and $g\taking K\to J$ are functions and that $F\taking A\to B$ and $G\taking C\to B$ are maps in $\sSets$ that respect $f$ and $g$ in the sense that $F(A_i)\ss B_{f(i)}$ and $G(C_k)\ss B_{g(k)}$ for all $i\in I$ and $k\in K$.  Let $$I\cross_JK=\{(i,j,k)\in I\cross J\cross K \| f(i)=j=g(k)\}$$ denote the fiber product of sets.  For typographical reasons, we use $\cross^h$  to denote homotopy limit in $\sSets$, and $\cross$ to denote the 1-categorical limit.

Then the natural map $$\left(\coprod_{(i,j,k)\in I\cross_JK}A_i\cross^h_{B_j}C_k\right)\too A\cross^h_BC$$ is a weak equivalence in $\sSets$.

\end{lemma}

\begin{proof}

If we replace $F$ by a fibration, then each component $F_i:=F|_{A_i}\taking A_i\to B_{f(i)}$ is a fibration.  We reduce to showing that the map \begin{align}\label{dia:straight limit}\left(\coprod_{(i,j,k)\in I\cross_JK}A_i\cross_{B_j}C_k\right)\too A\cross_BC\end{align} is an isomorphism of simplicial sets.  Restricting to the $n$-simplicies of both sides, we may assume that $A,B,$ and $C$ are (discrete simplicial) sets.  It is an easy exercise to show that the map in (\ref{dia:straight limit}) is injective and surjective, i.e. an isomorphism in $\Sets$.

\end{proof}

\begin{proposition}\label{transverse intersections}

Suppose that $a:M_0\to M$ and $b:M_1\to M$ are morphisms of manifolds, and suppose that a fiber product $N$ exists in the category of manifolds.  If $\mcX=(X,\mcO_X)$ is the fiber product $$\hPull{\mcX}{M_0}{M_1}{M}{}{}{a}{b}$$ taken
in the category of derived manifolds, then the natural map $g\taking N\to\mcX$ is an equivalence if and only if $a$ and $b$ are transverse.

\end{proposition}

\begin{proof}

Since limits taken in $\Man$ and in $\dMan$ commute with taking underlying topological spaces, the map $N\to X$ is a homeomorphism.  We have a commutative diagram $$\xymatrix{N\ar[r]^g\ar[d]_f&\mcX\ar[dl]^{f'}\\M_0\cross M_1,}$$ in which $f$ and $f'$ are closed immersions (pullbacks of the diagonal $M\to M\cross M$).

If $a$ and $b$ are not transverse, one shows easily that the first homology group $H_1L_\mcX\neq 0$ of the cotangent complex for $\mcX$ is nonzero, whereas $H_1L_N=0$ because $N$ is a manifold; hence $\mcX$ is not equivalent to $N$.

If $a$ and $b$ are transverse, then one can show that $g$ induces a quasi-isomorphism $g^*L_{f'}\to L_f$.  By Lemma \ref{Nakayama}, the map $g$ is an equivalence of derived manifolds.

\end{proof}

\begin{proposition}\label{fin lim in lrs}

The simplicial category $\LRS$ of local $C^\infty$-ringed spaces is closed under taking finite homotopy limits.

\end{proposition}

\begin{proof}

The local $C^\infty$-ringed space $(\RR^0,C^\infty(\RR^0))$ is a homotopy terminal object in $\LRS$.  Hence it suffices to show that a homotopy limit exists for any diagram $$(X,\mcO_X)\From{F}(Y,\mcO_Y)\To{G}(Z,\mcO_Z)$$ in $\LRS$.  We first describe the appropriate candidate for this homotopy limit.

The underlying space of the candidate is $X\cross_YZ$, and we label the maps as in the diagram $$\xymatrix{X\cross_YZ\ar[dr]^h\ar[r]^g\ar[d]_f&Z\ar[d]^G\\ X\ar[r]_F&Y}$$  The structure sheaf on the candidate is the homotopy colimit of pullback sheaves \begin{eqnarray}\label{eqn:sheaf}\mcO_{X\cross_YZ}:=(g^*\mcO_Z)\tensor_{(h^*\mcO_Y)}(f^*\mcO_X).\end{eqnarray}

To show that $\mcO_{X\cross_YZ}$ is a sheaf of local $C^\infty$-rings, we take the stalk at a point, apply $\pi_0$, and show that it is a local $C^\infty$-ring.  The homotopy colimit written in Equation (\ref{eqn:sheaf}) becomes the $C^\infty$-tensor product of pointed local $C^\infty$-rings.  By \cite[3.12]{MR}, the result is indeed a local $C^\infty$-ring.

One shows that $(X\cross_YZ,\mcO_{X\cross_YZ})$ is the homotopy limit of the diagram in the usual way.  We do not prove it here, but refer the reader to \cite[2.3.21]{Spi} or, for a much more general result, to \cite[2.4.21]{Lur-DAG5}.

\end{proof}

\begin{theorem}\label{fin lim in dman}

Let $M$ be a manifold, let $\mcX$ and $\mcY$ be derived manifolds, and let $f\taking\mcX\to M$ and $g\taking\mcY\to M$ be morphisms of derived manifolds.  Then a fiber product $\mcX\cross_M\mcY$ exists in the category of derived manifolds.

\end{theorem}

\begin{proof}

We showed in Proposition \ref{fin lim in lrs} that $\mcX\cross_M\mcY$ exists as a local $C^\infty$-ringd space.  To show that it is a derived manifold, we must only show that it is locally an affine derived manifold.  This is a local property, so it suffices to look locally on $M$, $\mcX,$ and $\mcY$.  We will prove the result by first showing that affine derived manifolds are closed under taking products, then that they are closed under solving equations, and finally that these two facts combine to prove the result.

Given affine derived manifolds $\RR^n_{a=0}$ and $\RR^m_{b=0}$, it follows formally that $\RR^{n+m}_{(f,g)=0}$ is their product, and it is an affine derived manifold.

Now let $\mcX=\RR^n_{a=0}$, where $a\taking\RR^n\to\RR^m$, and suppose that $b\taking\mcX\to\RR^k$ is a morphism.  By Theorem \ref{structure theorem}, we can consider $b$ as an element of $\mcO_X(X)(\RR^k)$.  By Lemma \ref{pi_0 on sC}, it is homotopic to a composite $\mcX\to\RR^n\To{b'}\RR^k$, where $\mcX\to\RR^n$ is the canonical imbedding.  Now we can realize $\mcX_{b=0}$ as the homotopy limit in the all-Cartesian diagram $$\xymatrix{\mcX_{b=0}\ar[r]\ar[d]\ulhlimit&\RR^0\ar[d]\\\mcX\ar[r]^b\ar[d]\ulhlimit&\RR^k\ar[r]\ar[d]\ulhlimit&\RR^0\ar[d]\\\RR^n\ar[r]_-{(b',a)}&\RR^k\cross\RR^m\ar[r]&\RR^m.}$$  Therefore, $\mcX_{b=0}=\RR^n_{(b',a)=0}$ is affine.

Finally, suppose that $\mcX$ and $\mcY$ are affine and that $M=\RR^p$.  Let $-\taking\RR^p\cross\RR^p\to\RR^p$ denote the coordinate-wise subtraction map.  Then there is an all-Cartesian diagram $$\xymatrix{\mcX\cross_{\RR^p}\mcY\ar[r]\ar[d]\ulhlimit&\RR^p\ar[d]_{diag}\ar[r]\ulhlimit&\RR^0\ar[d]^0\\\mcX\cross\mcY\ar[r]&\RR^p\cross\RR^p\ar[r]_-{-}&\RR^p,}$$ where $diag\taking\RR^p\to\RR^p\cross\RR^p$ is the diagonal map.  We have seen that $\mcX\cross\mcY$ is affine, so since $\mcX\cross_{\RR^p}\mcY$ is the solution to an equation on an affine derived manifold, it too is affine.  This completes the proof.

\end{proof}

\begin{remark}

Note that Theorem \ref{fin lim in dman} {\em does not} say that the category $\dMan$ is closed under arbitrary fiber products.  Indeed, if $M$ is not assumed to be a smooth manifold, then the fiber product of derived manifolds over $M$ need not be a derived manifold in our sense.  The cotangent complex of any derived manifold has homology concentrated in degrees 0 and 1 (see Corollary \ref{euler of der}), whereas a fiber product of derived manifolds (over a non-smooth base) would not have that property.   

Of course, using the spectrum functor Spec, defined in Remark \ref{rmk:spec for cinf}, one could define a more general category $\mcC$ of ``derived manifolds" in the usual scheme-theoretic way.  Then our $\dMan$ would form a full subcategory of $\mcC$, which one might call the subcategory of {\em quasi-smooth} objects (see \cite{Spi}).  The reason we did not introduce this category $\mcC$ is that it does not have the general cup product formula in cobordism; i.e. Theorem \ref{main theorem} does not apply to $\mcC$.  

\end{remark}

\section{Derived manifolds are good for doing intersection theory}\label{gfdi}

Recall from Section \ref{cinf} that $H_{\RR^i}\in\sC$ is the discrete $C^\infty$-ring corepresenting $\RR^i\in\EE$.  A smooth map $f\taking\RR^i\to\RR^j$ (contravariantly) induces a morphism of $C^\infty$-rings $H_{\RR^j}\to H_{\RR^i}$, which we often denote by $f$ for convenience.  Recall also that in $\sC$, there is a canonical weak equivalence $H_{\RR^{i+j}}\To{\we}H_{\RR^i}\amalg H_{\RR^j}.$

Let $U\taking\sC\to\sR$ denote the ``underlying simplicial commutative ring" functor from Corollary \ref{we on und alg}.

\begin{lemma}\label{hadamard}

Let $m,n,p\in\NN$ with $p\leq m$.  Let $x\taking\RR^{n+m}\to\RR^m$ denote the projection onto the first $m$ coordinates, let $g\taking\RR^p\to\RR^m$ denote the inclusion of a $p$-plane in $\RR^m$, and let $\Psi$ be the diagram $$\xymatrix{H_{\RR^m}\ar[r]^x\ar[d]_g&H_{\RR^{n+m}}\\H_{\RR^p}}$$ of $C^\infty$-rings.  The homotopy colimit of $\Psi$ is weakly equivalent to $H_{\RR^{n+p}}$.  

Moreover, the application of $U\taking\sC\to\sR$ commutes with taking homotopy colimit of $\Psi$ in the sense that the natural map $$\hocolim(U\Psi)\to U\hocolim(\Psi)$$ is a weak equivalence of simplicial commutative $\RR$-algebras.

\end{lemma}

\begin{proof}

We may assume that $g$ sends the origin to the origin.  For now, we assume that $p=0$ and $m=1$.

Consider the all-Cartesian diagram of manifolds $$\xymatrix{\RR^n\ar[r]\ar[d]\ullimit&\RR^{n+1}\ar[r]\ar[d]_x\ullimit&\RR^n\ar[d]\\\RR^0\ar[r]_g&\RR\ar[r]&\RR^0.}$$  Apply $H\taking\Man\op\to\sC$ to obtain the diagram $$\xymatrix{H_{\RR^n}&H_{\RR^{n+1}}\ar[l]&H_{\RR^n}\ar[l]\\H_{\RR^0}\ar[u]&H_{\RR}\ar[l]_g\ar[u]^x&H_{\RR^0}\ar[u]\ar[l].}$$  Since $H$ sends products in $\E$ to homotopy pushouts in $\sC$, the right-hand square and the big rectangle are homotopy pushouts.  Hence the left square is as well, proving the first assertion (in case $p=0$ and $m=1$).

For notational reasons, let $U^i$ denote $U(H_{\RR^i})$, so that $U^i$ is the (discrete simplicial) commutative ring whose elements are smooth maps $\RR^i\to\RR$.
Define a simplicial commutative $\RR$-algebra $D$ as the homotopy colimit in the diagram $$\hPush{U^1}{U^{n+1}}{U^0}{D.}{x}{g}{}{}$$  We must show that the natural map $D\to U^n$ is a weak equivalence of simplicial commutative rings.  Recall that the homotopy groups of a homotopy pushout of simplicial commutative rings are the $\Tor$ groups of the corresponding tensor product of chain complexes.

Since $x$ is a nonzerodivisor in the ring $U^{n+1}$, the homotopy groups 
of $D$ are $$\pi_0(D)=U^{n+1}\tensor_{U^1}U^0; \tn{ and } \pi_i(D)=0, i>0.$$ Thus, $\pi_0(D)$ can be identified with the equivalence 
classes of smooth functions $f\taking\RR^{n+1}\to\RR$, where $f\sim g$ if $f-g$ is a multiple of $x$.

On the other hand, we can identify elements of the ring $U^n$ of smooth functions on $\RR^n$  with the equivalence classes of functions $f\taking\RR^{n+1}\to\RR$, where $f\sim g$ if 
$$f(0,x_2,x_3,\ldots x_{n+1})=g(0,x_2,x_3,\ldots,x_{n+1}).$$  Thus to prove that the map $D\to U^n$ is a weak 
equivalence, we must only show that if a smooth function $f(x_1,\ldots,x_{n+1})\taking\RR^{n+1}\to\RR$ vanishes whenever $x_1=0$, then $x_1$ 
divides $f$.  This is called Hadamard's lemma, and it follows from the definition of smooth functions.  Indeed, given a smooth function $f(x_1,\ldots,x_{n+1})\taking\RR^{n+1}\to\RR$ which vanishes whenever $x_1=0$, define a 
function $g\taking\RR^{n+1}\to\RR$ by the formula $$g(x_1,\ldots,x_{n+1})=\lim_{a\to x_1}\frac{f(a,x_2,x_3,\ldots,x_{n+1})}{a}.$$  It is clear 
that $g$ is smooth, and $xg=f$.  

We have now proved both assertions in the case that $p=0$ and $m=1$; we continue to assume $p=0$ and prove the result for general $m+1$ by induction.  The inductive step follows from the all-Cartesian diagram below, in which each vertical arrow is an inclusion of a plane and each horizontal arrow is a coordinate projection: $$\xymatrix{\RR^n\ar[r]\ar[d]\ullimit&\RR^0\ar[d]\\\RR^{m+n}\ar[r]\ar[d]\ullimit&\RR^m\ar[r]\ar[d]\ullimit&\RR^0\ar[d]\\\RR^{n+m+1}\ar[r]&\RR^{m+1}\ar[r]&\RR.}$$  Apply $H$ to this diagram.  The arguments above implies that the two assertions hold for the right square and bottom rectangle; thus they hold for the bottom left square.  The inductive hypothesis implies that the two assertions hold for the top square, so they hold for the left rectangle, as desired.

For the case of general $p$, let $k\taking\RR^m\to\RR^{m-p}$ denote the projection orthogonal to $g$.  Applying $H$ to the all-Cartesian diagram $$\xymatrix{\RR^{n+p}\ar[r]\ar[d]\ullimit&\RR^p\ar[r]\ar[d]^g\ullimit&\RR^0\ar[d]\\\RR^{n+m}\ar[r]_x&\RR^m\ar[r]_k&\RR^{m-p},}$$ the result holds for the right square and the big rectangle (for both of which $p=0$), so it holds for the left square as well.

\end{proof}

\begin{remark}

Note that (the non-formal part of) Lemma \ref{hadamard} relies heavily on the fact that we are dealing with smooth maps.  It is this lemma which fails in the setting of topological manifolds, piecewise linear manifolds, etc.  

\end{remark}

Let $\mcM$ be a model category, and let ${\bf \Delta}$ denote the simplicial indexing category.  Recall that a simplicial resolution of an object $X\in\mcM$ consists of a simplicial diagram $X'\res\taking{\bf\Delta}\op\to\mcM$ and an augmentation map $X'\res\to X$, such that the induced map $\hocolim(X')\to X$ is a weak equivalence in $\mcM$.  

Conversely, the geometric realization of a simplicial object $Y\res\taking{\bf\Delta}\op\to\mcM$ is the object in $\mcM$ obtained by taking the homotopy colimit of the diagram $Y\res$.  

\begin{proposition}\label{und pre geo rea}

The functors $U\taking\sC\to\sR$ and $-(\RR)\taking\sC\to\sSets$ each preserve geometric realizations.

\end{proposition}

\begin{proof}

By \cite[A.1]{Lur-DAG1}, the geometric realization of any simplicial object in any of the model categories $\sC$, $\sR$, and $\sSets$ is equivalent to the diagonal of 
the corresponding bisimplicial object.  Since both $U$ and $-(\RR)$ are functors which preserve the diagonal, they each commute with geometric realization.

\end{proof}

Let $-(\RR)\taking\sC\to\sSets$ denote the functor $F\mapsto F(\RR)$.  It is easy to see that $-(\RR)$ is a right Quillen functor.  Its left adjoint is $(-\tensor H_\RR)\taking K\mapsto K\tensor H_\RR$ (although note that $K\tensor H_\RR$ is not generally fibrant in $\sC$).  We call a $C^\infty$-ring {\em free} if it is in the essential image of this functor $-\tensor H_\RR$, and similarly, we call a morphism of free $C^\infty$-rings {\em free} if it is in the image of this functor.

Thus, a morphism $d'\taking H_{\RR^k}\to H_{\RR^\ell}$ is free if and only if it is induced by a function $d\taking\{1,\ldots,k\}\to\{1,\ldots,\ell\}$.  To make $d'$ explicit, we just need to provide a dual map $\RR^\ell\to\RR^k$; for each $1\leq i\leq k$, we supply the map $\RR^\ell\to\RR$ given by projection onto the $d(i)$ coordinate.  

If $K$ is a simplicial set such that each $K_i$ is a finite set with cardinality $|K_i|=n_i$, then $X:=K\tensor H_\RR$ is the simplicial $C^\infty$-ring with $X_i=H_{\RR^{n_i}}$, and for a map $[\ell]\to[k]$ in $\bf{\Delta}$, the structure map $X_k\to X_\ell$ is the free map defined by $K_k\to K_\ell$.

\begin{lemma}\label{fun sim res}

Let $X$ be a simplicial $C^\infty$-ring.  There exists a functorial simplicial resolution $X'\res\to X$ in which $X'_n$ is a free simplicial $C^\infty$-ring for each $n$.  

Moreover, if $f\taking X\to Y$ is a morphism of $C^\infty$-rings and $f'\res\taking X'\res\to Y'\res$ is the induced map on simplicial resolutions, then for each $n\in\NN$, the map $f'_n\taking X'_n\to Y'_n$ is a morphism of free $C^\infty$-rings.

\end{lemma}

\begin{proof}

Let $R$ denote the underlying simplicial set functor $-(\RR):\sC\to\sSets$, and let $F$ denote its left adjoint.  The 
comonad $FR$ gives rise to an augmented simplicial $C^\infty$-ring, 
$$\xymatrix@1{\Phi=\cdots\arrr{r}&FRFR(X)\arr{r}&FR(X)\ar[r]&X.}$$  By \cite[8.6.10]{Wei}, the induced augmented simplicial set 
$$\xymatrix@1{R\Phi=\cdots\arrr{r}&RFRFR(X)\arr{r}&RFR(X)\ar[r]&R(X)}$$ is a weak equivalence.  By Propositions \ref{we on und alg} and \ref{und pre geo rea}, $\Phi$ is a simplicial resolution.

The second assertion is clear by construction.

\end{proof}

In the following theorem, we use a basic fact about simplicial sets: if $g\taking F\to H$ is a fibration of simplicial sets and $\pi_0g$ is a surjection of sets, then for each $n\in\NN$ the function $g_n\taking F_n\to H_n$ is surjective.  This is proved using the left lifting property for the cone point inclusion $\Delta^0\to\Delta^{n+1}$.

\begin{theorem}\label{commuting functors}

Let $\Psi$ be a diagram $$\xymatrix{F\ar[r]^f\ar[d]_g&G\\H}$$ of cofibrant-fibrant $C^\infty$-rings.  Assume that $\pi_0g\taking\pi_0F\to\pi_0H$ is a surjection.  Then application of $U\taking\sC\to\sR$ commutes with taking the homotopy colimit of $\Psi$ in the sense that the natural map $$\hocolim(U\Psi)\to U\hocolim(\Psi)$$ is a weak equivalence of simplicial commutative rings.

\end{theorem}

\begin{proof}

We prove the result by using simplicial resolutions to reduce to the case proved in Lemma \ref{hadamard}.  We  begin with a series of replacements and simplifications of the diagram $\Psi$, each of which preserves both $\hocolim(U\Psi)$ and $U\hocolim(\Psi)$.  

First, replace $g$ with a fibration and $f$ with a cofibration.  Note that since $g$ is surjective on $\pi_0$ and is a fibration, it is surjective in each degree; note also that $f$ is injective in every degree.  

Next, replace the diagram by the simplicial resolution given by Lemma \ref{fun sim res}.  This is a diagram $H'\res\From{g'\res}F'\res\To{f'\res}G'\res$, which has the same homotopy colimit.  We can compute this homotopy colimit by first computing the homotopy colimits $\hocolim(H'_n\From{g'_n}F'_n\To{f'_n}G'_n)$ in each degree, and then taking the geometric realization.   Since $U$ preserves geometric realization (Lemma \ref{und pre geo rea}), we may assume that $\Psi$ is a diagram $H\From{g}F\To{f}G$, in which $F,G,$ and $H$ are free simplicial $C^\infty$-rings, and in which $g$ is surjective and $f$ is injective.  By performing another simplicial resolution, we may assume further that $F,G,$ and $H$ are discrete. 

A free $C^\infty$-ring is the filtered colimit of its finitely generated (free) sub-$C^\infty$-rings; hence we may assume that $F,G,$ and $H$ are finitely generated.  In other words, each is of the form $S\tensor H_\RR$, where $S$ is a finite set.  We are reduced to the case in which $\Psi$ is the diagram $$\xymatrix{H_{\RR^m}\ar[r]^f\ar[d]_g&H_{\RR^n}\\ H_{\RR^p},}$$ where again $g$ is surjective and $f$ is injective.  Since $f$ and $g$ are free maps, they are induced by maps of sets $f_1\taking\{1,\ldots m\}\inj\{1,\ldots,n\}$ and $g_1\taking\{1,\ldots,m\}\to\{1,\ldots,p\}$.  The map $\RR^n\to\RR^m$ underlying $f$ is given by $$(a_1,\ldots,a_n)\mapsto (a_{f(1)},\ldots,a_{f(m)}),$$ which is a projection onto a coordinate plane through the origin, and we may arrange that it is a projection onto the first $m$ coordinates.  The result now follows from Lemma \ref{hadamard}.

\end{proof}

\begin{corollary}\label{com fun for imb}

Suppose that the diagram $$\xymatrix{(A,\mcO_A)\ar[r]^{G}\ar[d]_{F}\ulhlimit&(Y,\mcO_Y)\ar[d]^f\\(X,\mcO_X)\ar[r]_g&(Z,\mcO_Z)}$$ is a homotopy pullback of local $C^\infty$-ringed spaces.  If $g$ is an imbedding then the underlying diagram of ringed spaces is also a homotopy pullback.

\end{corollary}

\begin{proof}

In both the context of local $C^\infty$-ringed spaces and local ringed spaces, the space $A$ is the pullback of the diagram $X\to Z\from Y$.  The sheaf on $A$ is the homotopy colimit of the diagram $$F^*\mcO_X\From{F^*(g^\sharp)} F^*g^*\mcO_Z\To{G^*(f^\sharp)} G^*\mcO_Y,$$ either as a sheaf of $C^\infty$-rings or as a sheaf of simplicial commutative rings.  Note that taking inverse-image sheaves commutes with taking underlying simplicial commutative rings.

Since $g$ is an imbedding, we have seen that $g^\sharp\taking g^*\mcO_Z\to\mcO_X$ is surjective on $\pi_0$, and thus so is its pullback $F^*(g^\sharp)$.  The result now follows from Theorem \ref{commuting functors}.

\end{proof}

We will now give another way of viewing the ``locality condition" on ringed spaces.  Considering sections of the structure sheaf as functions to affine space, a ringed space is local if these functions pull covers back to covers.  This point of view can be found in \cite{SGA4} and \cite{BD}, for example.

Recall the notation $|F|=F(\RR)$ for a $C^\infty$-ring $F$; see Notation \ref{notation:abs val}.  Recall also that $C^\infty(\RR)$ denotes the free (discrete) $C^\infty$-ring on one generator.

\begin{definition}\label{inverse image}

Let $X$ be a topological space and $\mcF$ a sheaf of $C^\infty$-rings on $X$.  Given an open set $U\ss\RR$, let $\chi_U\taking\RR\to\RR$ denote a characteristic function of $U$ (i.e. $\chi_U$ vanishes precisely on $\RR-U$).  Let $f\in|\mcF(X)|_0$ denote a global section.  We will say that an open subset $V\ss X$ is {\em contained in the preimage under $f$ of $U$} if there exists a dotted arrow making the diagram of $C^\infty$-rings \begin{eqnarray}\label{dia:preimage}\xymatrix{C^\infty(\RR)\ar[r]^f\ar[d]&\mcF(X)\ar[d]^{\rho_{X,V}}\\ C^\infty(\RR)[\chi_U^\m1]\ar@{..>}[r]&\mcF(V)}\end{eqnarray} commute up to homotopy.  We say that $V$ {\em is the preimage under $f$ of $U$}, and by abuse of notation write $V=f^\m1(U)$, if it is maximal with respect to being contained in the preimage.  Note that these notions are independent of the choice of characteristic function $\chi_U$ for a given $U\ss\RR$.  Note also that since localization is an epimorphism of $C^\infty$-rings (as it is for ordinary commutative rings; see \cite[2.2,2.6]{MR}), the dotted arrow is unique if it exists.  

If $f,g\in|\mcF(X)|_0$ are homotopic vertices, then for any open subset $U\ss\RR$, one has $f^\m1(U)=g^\m1(U)\ss X$.  Therefore, this preimage functor is well-defined on the set of connected components $\pi_0|\mcF(X)|$.  Furthermore, if $f\in|\mcF(X)|_n$ is any simplex, all of its vertices are found in the same connected component, roughly denoted $\pi_0(f)\in\pi_0|\mcF(X)|$, so we can write $f^\m1(U)$ to denote $\pi_0(f)^\m1(U)$.

\end{definition}

\begin{example}

The above definition can be understood from the viewpoint of algebraic geometry.  Given a scheme $(X,\mcO_X)$ and a global section $f\in\mcO_X(X)=\pi_0\mcO_X(X)$, one can consider $f$ as a scheme morphism from $X$ to the affine line $\AA^1$.  Given a principle open subset $U=\Spec(k[x][g^\m1])\ss\AA^1$, we are interested in its preimage in $X$.  This preimage will be the largest $V\ss X$ on which the map $k[x]\To{f}\mcO_X(X)$ can be lifted to a map $k[x][g^\m1]\to\mcO_X(V)$.  This is the content of Diagram \ref{dia:preimage}.

Up next, we will give an alternate formulation of the condition that a sheaf of $C^\infty$-rings be a local in terms of these preimages.  In the algebro-geometric setting, it comes down to the fact that a sheaf of rings $\mcF$ is a sheaf of {\em local rings} on $X$ if and only if, for every global section $f\in\mcF(X)$, the preimages under $f$ of a cover of $\Spec(k[x])$ form a cover of $X$.  

\end{example}

\begin{lemma}\label{ope cov and fac}

Suppose that $U\ss\RR$ is an open subset of $\RR$, and let $\chi_U$ denote a characteristic function for $U$.  Then $C^\infty(U)=C^\infty(\RR)[\chi_U^\m1]$.  

There is a natural bijection between the set of points $p\in\RR$ and the set of $C^\infty$-functions $A_p\taking C^\infty(\RR)\to C^\infty(\RR^0)$.  Under this correspondence, $p$ is in $U$ if and only if $A_p$ factors through $C^\infty(\RR)[\chi_U^\m1]$.

\end{lemma}

\begin{proof}

This follows from \cite[1.5 and 1.6]{MR}.

\end{proof}

Recall that $F\in\sC$ is said to be a local $C^\infty$-ring if the commutative ring underlying $\pi_0F$ is a local ring (see Definition \ref{local cinf rings}).

\begin{proposition}\label{equiv locality cond}

Let $X$ be a space and $\mcF$ a sheaf of finitely presented $C^\infty$-rings on $X$.  Then $\mcF$ is local if and only if, for any cover of $\RR$ by open subsets $\RR=\bigcup_iU_i$ and for any open $V\ss X$ and local section $f\in\pi_0|\mcF(V)|$, the preimages $f^\m1(U_i)$ form a cover of $V$.

\end{proposition}

\begin{proof}

Choose a representative for $f\in\pi_0|F|$, call it $f\in|F|_0$ for simplicity, and recall that it can be identified with a map $f\taking C^\infty(\RR)\to F$, which is unique up to homotopy.

Both the property of being a sheaf of local $C^\infty$-rings and the above ``preimage of a covering is a covering" property is local on $X$.  Thus we may assume that $X$ is a point.  We are reduced to proving that a $C^\infty$-ring $F$ is a local $C^\infty$-ring if and only if, for any cover of $\RR$ by open sets $\RR=\bigcup_iU_i$ and for any element $f\in|F|_0$, there exists an index $i$ and a dotted arrow making the diagram $$\xymatrix{C^\infty(\RR)\ar[d]\ar[r]^f&F\\C^\infty(\RR)[\chi_{U_i}^\m1]\ar@{..>}[ur]&}$$ commute up to homotopy.   

Suppose first that for any cover of $\RR$ by open subsets $\RR=\bigcup_iU_i$ and for any element $f\in F_0$, there exists a lift as above.  Let $U_1=(-\infty,\frac{1}{2})$ and let $U_2=(0,\infty)$.  By assumption, either $f$ factors through $C^\infty(U_1)$ or through $C^\infty(U_2)$, and $1-f$ factors through the other by Lemma \ref{ope cov and fac}.  It is easy to show that any element of $\pi_0F$ which factors through $C^\infty(U_2)$ is invertible (using the fact that $0\not\in U_2$).  Hence, either $f$ or $1-f$ is invertible in $\pi_0F$, so $\pi_0F$ is a local $C^\infty$-ring.

Now suppose that $F$ is local, i.e. that it has a unique maximal ideal, and suppose $\RR=\bigcup_iU_i$ is an open cover.  Choose $f\in|F|_0$, considered as a map of $C^\infty$-rings $f\taking C^\infty(\RR)\to F$.  By \cite[3.8]{MR}, $\pi_0F$ has a unique point $F\to\pi_0F\to C^\infty(\RR^0)$. By Lemma \ref{ope cov and fac}, there exists $i$ such that the composition $C^\infty(\RR)\To{f} F\to C^\infty(\RR^0)$ factors through $C^\infty(\RR)[\chi_{U_i}^\m1]$, giving the solid arrow square $$\xymatrix{C^\infty(\RR)\ar[r]^f\ar[d]&F\ar[d]\\C^\infty(\RR)[\chi_{U_i}^\m1]\ar[r]\ar@{..>}[ur]&C^\infty(\RR^0).}$$  Since $\chi_{U_i}\in C^\infty(\RR)$ is not sent to $0\in C^\infty(\RR^0)$, its image $f(\chi_{U_i})$ is not contained in the maximal ideal of $\pi_0F$, so a dotted arrow lift exists making the diagram commute up to homotopy.  This proves the proposition.

\end{proof}

In the following theorem, we use the notation $|A|$ to denote the simplicial set $A(\RR)=\Map_\sC(C^\infty(\RR),A)$ underlying a simplicial $C^\infty$-ring $A\in\sC$.

\begin{theorem}\label{structure theorem}

Let $\mcX=(X,\mcO_X)$ be a local $C^\infty$-ringed space, and let $\i\RR=(\RR,C^\infty_\RR)$ denote the (image under $\i$ of the) real line.  There is a natural homotopy equivalence of simplicial sets $$\Map_\LRS((X,\mcO_X),(\RR,C^\infty_\RR))\To{\we}|\mcO_X(X)|.$$

\end{theorem}

\begin{proof}

We will construct morphisms \begin{align*}K\taking\Map_\LRS(\mcX,\i\RR)\to|\mcO_X(X)| &\tn{, and}\\ L\taking|\mcO_X(X)|\to\Map_\LRS(\mcX,\i\RR),\end{align*} and show that they are homotopy inverses.  For the reader's convenience, we recall the definition $$\Map_\LRS(\mcX,\RR)=\coprod_{f\taking X\to\RR}\Map_\loc(f^*C^\infty_\RR,\mcO_X).$$

The map $K$ is fairly easy and can be defined without use of the locality condition.  Suppose that $\phi\taking X\to\RR$ is a map of topological spaces.  The restriction of $K$ to the corresponding summand of $\Map_\LRS(\mcX,\i\RR)$ is given by taking global sections $$\Map_\loc(\phi^*C^\infty_\RR,\mcO_X)\to\Map(C^\infty(\RR),\mcO_X(X))\iso|\mcO_X(X)|.$$

To define $L$ is a bit harder and depends heavily on the assumption that $\mcO_X$ is a local sheaf on $X$.  First, given an $n$-simplex $g\in|\mcO_X(X)|_n$ we need to define a map of topological spaces $L(g)\taking X\to\RR$.  Let $g_0\in\pi_0|\mcO_X(X)|$ denote the connected component containing $g$.  By Proposition \ref{equiv locality cond}, $g_0$ gives rise to a function from open covers of $\RR$ to open covers of $X$, and this function commutes with refinement of open covers.  Since $\RR$ is Hausdorff, there is a unique map of topological spaces $X\to\RR$, which we take as $L(g)$, consistent with such a function.  Denote $L(g)$ by $G$, for ease of notation.

Now we need to define a map of sheaves of $C^\infty$-rings $$G^\flat\taking C^\infty_\RR\tensor\Delta^n\to G_*(\mcO_X).$$  On global sections, we have such a function already, since $g\in|\mcO_X(X)|_n$ can be considered as a map $g\taking C^\infty_\RR(\RR)\to|\mcO_X(X)|_n$.  Let $V\ss\RR$ denote an open subset and $g^\m1(V)\ss X$ its preimage under $g$.  The map $\rho\taking C^\infty_\RR(\RR)\to C^\infty_\RR(V)$ is a localization; hence, it is an epimorphism of $C^\infty$-rings.  

To define $G^\flat$, we will need to show that there exists a unique dotted arrow making the diagram \begin{eqnarray}\label{dia:structure}\xymatrix{C^\infty_\RR(\RR)\tensor\Delta^n\ar[r]^g\ar[d]_{\rho\tensor\Delta^n}&|\mcO_X(X)|\ar[d]\\C^\infty_\RR(V)\tensor\Delta^n\ar@{..>}[r]&|\mcO_X(g^\m1(V))|}\end{eqnarray} commute.  Such an arrow exists by definition of $g^\m1$.  It is unique because $\rho$ is an epimorphism of $C^\infty$-rings, so $\rho\tensor\Delta^n$ is as well.  We have now defined $G^\flat$, and we take $G^\sharp\taking G^*C^\infty_\RR\to\mcO_X$ to be the left adjoint of $G^\flat$.

We must show that for every $g\in|\mcO_X(X)|_0$, the map $G^\sharp\taking G^*C^\infty_\RR\to\mcO_X$ provided by $L$ is local.  (We can choose $g$ to be a vertex because, by definition, a simplex in $\Map(G^*C^\infty_\RR,\mcO_X)$ is local if all of its vertices are local.)  To prove this, we may take $X$ to be a point, $F=\mcO_X(X)$ a local $C^\infty$-ring, and $x=G(X)\in\RR$ the image point of $X$.  We have a morphism of $C^\infty$-rings $G^\sharp\taking(C^\infty_\RR)_x\to F$, in which both the domain and codomain are local.  It is a local ring homomorphism because all prime ideals in the local ring $(C^\infty_\RR)_x$ are maximal.  

The maps $K$ and $L$ have now been defined, and they are homotopy inverses by construction. 

\end{proof}

The following is a technical lemma that allows us to take homotopy limits component-wise in the model category of simplicial sets.

\begin{lemma}\label{holim lemma}

Let $\sSets$ denote the category of simplicial sets.  Let $I, J,$ and $K$ denote sets and let $$A=\coprod_{i\in I}A_i,\hspace{.4in} B=\coprod_{j\in J}B_j,\hspace{.2in} \tn{and}\hspace{.2in} C=\coprod_{k\in K}C_k$$ denote coproducts in $\sSets$ indexed by $I, J,$ and $K$.  

Suppose that $f\taking I\to J$ and $g\taking K\to J$ are functions and that $F\taking A\to B$ and $G\taking C\to B$ are maps in $\sSets$ that respect $f$ and $g$ in the sense that $F(A_i)\ss B_{f(i)}$ and $G(C_k)\ss B_{g(k)}$ for all $i\in I$ and $k\in K$.  Let $$I\cross_JK=\{(i,j,k)\in I\cross J\cross K \| f(i)=j=g(k)\}$$ denote the fiber product of sets.  For typographical reasons, we use $\cross^h$  to denote homotopy limit in $\sSets$, and $\cross$ to denote the 1-categorical limit.

Then the natural map $$\left(\coprod_{(i,j,k)\in I\cross_JK}A_i\cross^h_{B_j}C_k\right)\too A\cross^h_BC$$ is a weak equivalence in $\sSets$.

\end{lemma}

\begin{proof}

If we replace $F$ by a fibration, then each component $F_i:=F|_{A_i}\taking A_i\to B_{f(i)}$ is a fibration.  We reduce to showing that the map \begin{align}\label{dia:straight limit}\left(\coprod_{(i,j,k)\in I\cross_JK}A_i\cross_{B_j}C_k\right)\too A\cross_BC\end{align} is an isomorphism of simplicial sets.  Restricting to the $n$-simplicies of both sides, we may assume that $A,B,$ and $C$ are (discrete simplicial) sets.  It is an easy exercise to show that the map in (\ref{dia:straight limit}) is injective and surjective, i.e. an isomorphism in $\Sets$.

\end{proof}

\begin{proposition}\label{transverse intersections}

Suppose that $a:M_0\to M$ and $b:M_1\to M$ are morphisms of manifolds, and suppose that a fiber product $N$ exists in the category of manifolds.  If $\mcX=(X,\mcO_X)$ is the fiber product $$\hPull{\mcX}{M_0}{M_1}{M}{}{}{a}{b}$$ taken
in the category of derived manifolds, then the natural map $g\taking N\to\mcX$ is an equivalence if and only if $a$ and $b$ are transverse.

\end{proposition}

\begin{proof}

Since limits taken in $\Man$ and in $\dMan$ commute with taking underlying topological spaces, the map $N\to X$ is a homeomorphism.  We have a commutative diagram $$\xymatrix{N\ar[r]^g\ar[d]_f&\mcX\ar[dl]^{f'}\\M_0\cross M_1,}$$ in which $f$ and $f'$ are closed immersions (pullbacks of the diagonal $M\to M\cross M$).

If $a$ and $b$ are not transverse, one shows easily that the first homology group $H_1L_\mcX\neq 0$ of the cotangent complex for $\mcX$ is nonzero, whereas $H_1L_N=0$ because $N$ is a manifold; hence $\mcX$ is not equivalent to $N$.

If $a$ and $b$ are transverse, then one can show that $g$ induces a quasi-isomorphism $g^*L_{f'}\to L_f$.  By Lemma \ref{Nakayama}, the map $g$ is an equivalence of derived manifolds.

\end{proof}

\begin{proposition}\label{fin lim in lrs}

The simplicial category $\LRS$ of local $C^\infty$-ringed spaces is closed under taking finite homotopy limits.

\end{proposition}

\begin{proof}

The local $C^\infty$-ringed space $(\RR^0,C^\infty(\RR^0))$ is a homotopy terminal object in $\LRS$.  Hence it suffices to show that a homotopy limit exists for any diagram $$(X,\mcO_X)\From{F}(Y,\mcO_Y)\To{G}(Z,\mcO_Z)$$ in $\LRS$.  We first describe the appropriate candidate for this homotopy limit.

The underlying space of the candidate is $X\cross_YZ$, and we label the maps as in the diagram $$\xymatrix{X\cross_YZ\ar[dr]^h\ar[r]^g\ar[d]_f&Z\ar[d]^G\\ X\ar[r]_F&Y}$$  The structure sheaf on the candidate is the homotopy colimit of pullback sheaves \begin{eqnarray}\label{eqn:sheaf}\mcO_{X\cross_YZ}:=(g^*\mcO_Z)\tensor_{(h^*\mcO_Y)}(f^*\mcO_X).\end{eqnarray}

To show that $\mcO_{X\cross_YZ}$ is a sheaf of local $C^\infty$-rings, we take the stalk at a point, apply $\pi_0$, and show that it is a local $C^\infty$-ring.  The homotopy colimit written in Equation (\ref{eqn:sheaf}) becomes the $C^\infty$-tensor product of pointed local $C^\infty$-rings.  By \cite[3.12]{MR}, the result is indeed a local $C^\infty$-ring.

One shows that $(X\cross_YZ,\mcO_{X\cross_YZ})$ is the homotopy limit of the diagram in the usual way.  We do not prove it here, but refer the reader to \cite[2.3.21]{Spi} or, for a much more general result, to \cite[2.4.21]{Lur-DAG5}.

\end{proof}

\begin{theorem}\label{fin lim in dman}

Let $M$ be a manifold, let $\mcX$ and $\mcY$ be derived manifolds, and let $f\taking\mcX\to M$ and $g\taking\mcY\to M$ be morphisms of derived manifolds.  Then a fiber product $\mcX\cross_M\mcY$ exists in the category of derived manifolds.

\end{theorem}

\begin{proof}

We showed in Proposition \ref{fin lim in lrs} that $\mcX\cross_M\mcY$ exists as a local $C^\infty$-ringd space.  To show that it is a derived manifold, we must only show that it is locally an affine derived manifold.  This is a local property, so it suffices to look locally on $M$, $\mcX,$ and $\mcY$.  We will prove the result by first showing that affine derived manifolds are closed under taking products, then that they are closed under solving equations, and finally that these two facts combine to prove the result.

Given affine derived manifolds $\RR^n_{a=0}$ and $\RR^m_{b=0}$, it follows formally that $\RR^{n+m}_{(f,g)=0}$ is their product, and it is an affine derived manifold.

Now let $\mcX=\RR^n_{a=0}$, where $a\taking\RR^n\to\RR^m$, and suppose that $b\taking\mcX\to\RR^k$ is a morphism.  By Theorem \ref{structure theorem}, we can consider $b$ as an element of $\mcO_X(X)(\RR^k)$.  By Lemma \ref{pi_0 on sC}, it is homotopic to a composite $\mcX\to\RR^n\To{b'}\RR^k$, where $\mcX\to\RR^n$ is the canonical imbedding.  Now we can realize $\mcX_{b=0}$ as the homotopy limit in the all-Cartesian diagram $$\xymatrix{\mcX_{b=0}\ar[r]\ar[d]\ulhlimit&\RR^0\ar[d]\\\mcX\ar[r]^b\ar[d]\ulhlimit&\RR^k\ar[r]\ar[d]\ulhlimit&\RR^0\ar[d]\\\RR^n\ar[r]_-{(b',a)}&\RR^k\cross\RR^m\ar[r]&\RR^m.}$$  Therefore, $\mcX_{b=0}=\RR^n_{(b',a)=0}$ is affine.

Finally, suppose that $\mcX$ and $\mcY$ are affine and that $M=\RR^p$.  Let $-\taking\RR^p\cross\RR^p\to\RR^p$ denote the coordinate-wise subtraction map.  Then there is an all-Cartesian diagram $$\xymatrix{\mcX\cross_{\RR^p}\mcY\ar[r]\ar[d]\ulhlimit&\RR^p\ar[d]_{diag}\ar[r]\ulhlimit&\RR^0\ar[d]^0\\\mcX\cross\mcY\ar[r]&\RR^p\cross\RR^p\ar[r]_-{-}&\RR^p,}$$ where $diag\taking\RR^p\to\RR^p\cross\RR^p$ is the diagonal map.  We have seen that $\mcX\cross\mcY$ is affine, so since $\mcX\cross_{\RR^p}\mcY$ is the solution to an equation on an affine derived manifold, it too is affine.  This completes the proof.

\end{proof}

\begin{remark}

Note that Theorem \ref{fin lim in dman} {\em does not} say that the category $\dMan$ is closed under arbitrary fiber products.  Indeed, if $M$ is not assumed to be a smooth manifold, then the fiber product of derived manifolds over $M$ need not be a derived manifold in our sense.  The cotangent complex of any derived manifold has homology concentrated in degrees 0 and 1 (see Corollary \ref{euler of der}), whereas a fiber product of derived manifolds (over a non-smooth base) would not have that property.   

Of course, using the spectrum functor Spec, defined in Remark \ref{rmk:spec for cinf}, one could define a more general category $\mcC$ of ``derived manifolds" in the usual scheme-theoretic way.  Then our $\dMan$ would form a full subcategory of $\mcC$, which one might call the subcategory of {\em quasi-smooth} objects (see \cite{Spi}).  The reason we did not introduce this category $\mcC$ is that it does not have the general cup product formula in cobordism; i.e. Theorem \ref{main theorem} does not apply to $\mcC$.  

\end{remark}

\section{Relationship to similar work}\label{Lurie}

In this section we discuss other research which is in some way related to the present paper.  Most relevant is Section \ref{subsec:lurie}, in which we discuss the relationship to Lurie's work on derived algebraic geometry, and in particular to structured spaces.  The other sections discuss manifolds with singularities, Chen spaces and diffeological spaces, synthetic differential geometry, and a catch-all section to concisely state how one might orient our work within the canon.

\subsection{Lurie's Structured spaces}\label{subsec:lurie}

There is a version of the above work on derived manifolds, presented in the author's PhD dissertation \cite{Spi}, which very closely follows Jacob Lurie's theory of structured spaces, as presented in \cite{Lur-DAG5}.  Similar in spirit is Toen and Vezzosi's work \cite{TV1}, \cite{TV2} on homotopical algebraic geometry.  In this section we attempt to orient the reader to Lurie's theory of structured spaces.

In order to define structured spaces, Lurie begins with the definition of a {\em geometry}.  A geometry is an $\infty$-category, equipped with a given choice of ``admissible" morphisms that generate a Grothendieck topology, and satisfying certain conditions.  For example, there is a geometry whose underlying category is the category of affine schemes $\Spec R$, with admissible morphisms given by principal open sets $\Spec R[a^\m1]\to\Spec R$, and with the usual Grothendieck topology of open coverings.  

Given a geometry $G$ and a topological space $X$, a $G$-structure on $X$ is roughly a functor $\mcO_X\taking G\to\Shv(X)$ which preserves finite limits and sends covering sieves on $G$ to effective epimorphisms in $\Shv(X)$.  

One should visualize the objects of $G$ as spaces and the admissible morphisms in $G$ as open inclusions.  In this visualization, a $G$-structure on $X$ provides, for each ``space" $g\in G$, a sheaf $\mcO_X(g)$, whose sections are seen as ``maps" from $X$ to $g$.  Since a map from $X$ to a limit of $g$'s is a limit of maps, one sees immediately why we require $\mcO_X$ to be left exact.  Given a covering sieve in $G$, we want to be able to say that to give a map from $X$ to the union of the cover is accomplished by giving local maps to the pieces of the cover, such that these maps agree on overlaps.  This is the covering sieve condition in the definition of $G$-structure.

Our approach follows Lurie's in spirit, but not in practice.  The issue is in his definition \cite[1.2.1]{Lur-DAG5} of admissibility structure, which we recall here.  

\begin{definition}

Let $\mcG$ be an $\infty$-category.  An {\em admissibility structure} on $\mcG$ consists of the following data: \begin{enumerate}[(1)]\item A subcategory $\mcG^{\tn{ad}}\ss\mcG$, containing every object of $\mcG$.  Morphisms of $\mcG$ which belong to $\mcG^{\tn{ad}}$ will be called {\em admissible} morphisms in $\mcG$.\item A Grothendieck topology on $\mcG^\tn{ad}$.\end{enumerate}  These data are required to satisfy the following conditions: \begin{enumerate}[(i)] \item Let $f\taking U\to X$ be an admissible morphism in $\mcG$, and g$\taking X'\to X$ any morphism.  Then there exists a pullback diagram $$\xymatrix{U'\ar[r]\ar[d]_{f'}&U\ar[d]^f\\X'\ar[r]_g&X,}$$ where $f'$ is admissible.\item Suppose given a commutative triangle $$\xymatrix{&Y\ar[dr]^g&\\X\ar[rr]_h\ar[ur]^f&&Z}$$ in $\mcG$, where $g$ and $h$ are admissible.  Then $f$ is admissible.\item Every retract of an admissible morphism of $\mcG$ is admissible.\end{enumerate}

\end{definition}

In our case, the role of $\mcG$ is played by $\EE$, the category of Euclidean spaces, and the role of $\mcG^\tn{ad}$ is played by open inclusions $\RR^n\inj\RR^n$ (see Sections \ref{cinf} and \ref{lrs}).  But, as such, $\mcG^\tn{ad}$ is not an admissibility structure on $\mcG$ because it does not satisfy condition (i): the pullback of a Euclidean open subset is not necessarily Euclidean. 

However such a pullback is {\em locally} Euclidean.  This should be enough to define something like a ``pre-admissibility structure," which does the same job as an admissibility structure.  In private correspondence, Lurie told me that such a notion would be useful -- perhaps this issue will be resolved in a later version of \cite{Lur-DAG5}.

In the author's dissertation, however, we did not use $C^\infty$-rings as our basic algebraic objects.  Instead, we used something called ``smooth rings."  A smooth ring is a functor $\Man\to\sSets$ which preserves pullbacks along submersions (see \cite[Definition 2.1.3]{Spi}).  In this case, the role of $\mcG$ is played by $\Man$, and the role of $\mcG^\tn{ad}$ is played by submersions.  This is an admissibility structure in Lurie's sense, and it should not be hard to prove that the category of structured spaces one obtains in this case is equivalent to our category of local $C^\infty$-ringed spaces.

\subsection{Manifolds with singularities}

A common misconception about derived manifolds is that every singular homology class $a\in H_*(M,\ZZ)$ of a manifold $M$ should be representable by an oriented derived manifold.  The misconception seems to arise from the idea that manifolds with singularities, objects obtained by ``coning off a submanifold of $M$," should be examples of derived manifolds.  This is not the case. 

By the phrase ``coning off a submanifold $A\ss M$," one means taking the colimit of a diagram $\{*\}\from A\to M$.  Derived manifolds are not closed under taking arbitrary colimits, e.g. quotients.  In particular, one cannot naturally obtain a derived manifold structure on a manifold with singularities.  Instead, one obtains derived manifolds from taking the zero-set of a section of a smooth vector bundle: see Example \ref{main examples}.  The collection of derived manifolds is quite large, but it does not include manifolds with singularities in a natural way.  In order to obtain arbitrary colimits, perhaps one should consider stacks on derived manifolds, but we have not worked out this idea.

\subsection{Chen spaces, diffeological spaces}

Another common generalization of the category of manifolds was invented by Kuo Tsai Chen in \cite{che}.  Let $Conv$ denote the category whose objects are convex subsets of $\RR^\infty$, and whose morphisms are smooth maps.  With the Grothendieck topology in which coverings are given by open covers in the usual sense, we can define the topos $\Shv(Conv)$.  The category of Chen spaces is roughly this topos, the difference being that points are given more importance in Chen spaces than in $\Shv(Conv)$, in a sense known as {\em concreteness} (see \cite{BH} for a precise account).  Diffeological spaces are similar -- they are defined roughly as sheaves on the category of open subsets of Euclidean spaces.

The difference between these approaches and our own is that Chen spaces are based on ``maps in" to the object in question (the Chen space) whereas our objects carry information about ``maps out" of the object in question (the derived manifold).  In other words, the simplest question one can ask about a Chen space $X$ is ``what are the maps from $\RR^n$ to $X$?"  Since $X$ is a sheaf on a site in which $\RR^n$ is an object, the answer is simply $X(\RR^n)$.  In the case of derived manifolds, the simplest question one can ask is ``what are the maps from $X$ to $\RR^n$?"  By the structure theorem, Theorem \ref{structure theorem}, information about maps from $X$ to Euclidean spaces is carried by the structure sheaf $\mcO_X$ -- the answer to the question is $\mcO_X(X)^n.$

If one is interested in generalizing manifolds to better study maps in to them, one should probably use Chen spaces or diffeological spaces.  In our case, we were interested in cohomological properties (intersection theory and cup product); since elements of cohomology on $X$ are determined by maps out of $X$, we constructed our generalized manifolds to be well-behaved with respect to maps out.  It may be possible to generalize further and talk about ``derived Chen spaces," but we have not yet pursued this idea.

\subsection{Synthetic differential geometry}

In the book \cite{MR}, Moerdijk and Reyes discuss yet another generalization of manifolds, called smooth functors.  These are functors from the category of (discrete) $C^\infty$-rings to sets that satisfy a descent condition (see \cite[3.1.1]{MR}).  A smooth functor can be considered as a patching of local neighborhoods, each of which is a formal $C^\infty$-variety.  

Neither our setup nor theirs is more general than the other.  While both are based on $C^\infty$-rings, our approach uses homotopical ideas, whereas theirs does not; their approach gives a topos, whereas ours does not.  It certainly may be possible to combine these ideas into ``derived smooth functors," but we have not pursued this idea either.   The non-homotopical approach does not seem adequate for a general cup product formula in the sense of Definition \ref{general cup}.

\subsection{Other similar work}

There has been far too much written about the intersection theory of manifolds for us to list here.  In our work, we achieve an intersection pairing at the level of spaces: the intersection of two submanifolds is still a geometric object (i.e. an object in a geometric category in the sense of Definition \ref{def geo}), and this geometric object has an appropriate fundamental class in cohomology (see Definition \ref{general cup}).  No ``general position" requirements are necessary.  We hope that this is enough to distinguish our results from previous ones.

\bibliographystyle{abbrv}

\end{document}